\documentclass[letter,10pt]{amsart}
\usepackage{amssymb,amsmath,amsfonts,amsthm}
\usepackage{mathrsfs}
\usepackage[left=1 in, right=1 in,top=1 in, bottom=1 in]{geometry}

\newtheorem{theorem}{Theorem}[section]
\newtheorem{lemma}[theorem]{Lemma}

\newtheorem{remark}[theorem]{Remark}

\makeatletter
\renewcommand \theequation {%
\ifnum \c@section>\z@ \@arabic\c@section.%
\fi\@arabic\c@equation} \@addtoreset{equation}{section}
\makeatother

\providecommand{\ud}[1]{\mathrm{d}{#1}}
\providecommand{\abs}[1]{\left\vert#1\right\vert}
\providecommand{\nm}[1]{\left\Vert#1\right\Vert}

\providecommand{\br}[1]{\langle #1 \rangle}
\providecommand{\tm}[2]{\left\Vert#1\right\Vert_{L^2(#2)}}
\providecommand{\im}[2]{\left\Vert#1\right\Vert_{L^{\infty}(#2)}}

\providecommand{\lnnm}[1]{{\left\Vert#1\right\Vert}_{L^{\infty}L^{\infty}}}
\providecommand{\lnm}[1]{\left\Vert#1\right\Vert_{L^{\infty}}}
\providecommand{\tnm}[1]{\left\Vert#1\right\Vert_{L^{2}}}
\providecommand{\tnnm}[1]{{\left\Vert#1\right\Vert}_{L^{2}L^2}}
\providecommand{\ltnm}[1]{{\left\Vert#1\right\Vert}_{L^{\infty}L^{2}}}

\def\p{\partial}
\def\u{U^{\e}}
\def\ub{\mathscr{U}^{\e}}
\def\bu{\bar U^{\e}}
\def\bub{\bar{\mathscr{U}}^{\e}}
\def\uc{U}
\def\ubc{\mathscr{U}}
\def\buc{\bar U}
\def\bubc{\bar{\mathscr{U}}}

\def\half{\frac{1}{2}}
\def\e{\epsilon}
\def\s{\mathcal{S}}
\def\vx{\vec x}
\def\vw{\vec w}

\def\nx{\nabla_{x}}
\def\t{\mathcal{T}}
\def\a{\mathcal{A}}

\def\px{\p_{\eta}}
\def\ps{\p_{\theta}}
\def\rt{\rightarrow}
\def\l{\lambda}

\def\ll{\mathcal{L}}
\def\pp{\mathcal{P}}
\def\r{\mathbb{R}}
\def\f{f}
\def\k{\kappa}
\def\q{Q}
\def\qb{\mathscr{Q}}
\def\v{\mathscr{V}}

\begin{document}
\title{Geometric Correction for Diffusive Expansion of
Steady Neutron Transport Equation}
\author{Lei Wu and Yan Guo}
\address{
Division of Applied Mathematics\\
Brown University \\
182 George Street, Providence, RI 02912, USA } \email[L.
Wu]{Lei\_Wu@brown.edu}
\address{
Division of Applied Mathematics\\
Brown University \\
182 George Street, Providence, RI 02912, USA } \email[Y.
Guo]{Yan\_Guo@brown.edu} \subjclass[2000]{35L65, 82B40, 34E05}

\begin{abstract}
We revisit the diffusive limit of a steady neutron transport
equation in a $2$-D unit disk $\Omega=\{\vx=(x_1,x_2):\ \abs{\vx}\leq 1\}$
with one-speed velocity $\Sigma=\{\vw=(w_1,w_2):\ \vw\in\s^1\}$ as
\begin{eqnarray}
\left\{
\begin{array}{rcl}
\e \vw\cdot\nabla_x u^{\e}+u^{\e}-\bar u^{\e}&=&0\label{transport}\
\ \ \text{in}\ \ \Omega,\\\rule{0ex}{1.0em}
u^{\e}(\vx_0,\vw)&=&g(\vx_0,\vw)\ \ \text{for}\ \ \vw\cdot\vec n<0\
\ \text{and}\ \ \vx_0\in\p\Omega,
\end{array}
\right.
\end{eqnarray}
where
\begin{eqnarray}\label{average 1}
\bar u^{\e}(\vx)=\frac{1}{2\pi}\int_{\s^1}u^{\e}(\vx,\vw)\ud{\vw}.
\end{eqnarray}
and $\vec n$ is the outward normal vector on $\p\Omega$, with the
Knudsen number $0<\e<<1$. A classical result in \cite{book1} states
that
\begin{eqnarray}\label{wrong}
\lnm{u^\e-\uc_0-\ubc_0}=O(\e)
\end{eqnarray}
where $\ubc_0$ is the Knudsen layer solution to the Milne problem
(\ref{classical temp 1}) while $\uc_0$ is the corresponding interior
solution to the Laplace equation (\ref{classical temp 2}). We
observe that the construction of the first order Knudsen layer fails
in \cite{book1}, due to the intrinsic singularity in the Milne
problem. Instead, we are able to establish
\begin{eqnarray}
\lnm{u^\e-\u_0-\ub_0}=O(\e)
\end{eqnarray}
where $\ub_0$ is the solution to the $\e$-Milne problem
(\ref{expansion temp 9}) while $\u_0$ is the corresponding interior
solution to the Laplace equation (\ref{expansion temp 8}).
Consequently, we deduce that
\begin{eqnarray}
\lnm{u^\e-\uc_0-\ubc_0}=O(1)
\end{eqnarray}
for some data, and the classical Knudsen layer theory (\ref{wrong})
is invalid in
$L^\infty$.\\
\textbf{Keywords:} $\e$-Milne problem, Knudsen layer solution,
Geometric correction.
\end{abstract}

\maketitle

\pagestyle{myheadings} \thispagestyle{plain} \markboth{LEI WU AND
YAN GUO}{GEOMETRIC CORRECTION FOR DIFFUSIVE EXPANSION OF NEUTRON
TRANSPORT EQUATION}

\section{Introduction and Notation}

The diffusive limit $\e\rt0$ of the neutron transport equation
(\ref{transport}) is a classical problem in kinetic theory. In the domain
$\Omega\times\Sigma$, the neutron density $u^{\e}(\vx,\vw)$
satisfies the equation (\ref{transport}).

Based on the flow direction, we can divide the boundary
$\Gamma=\{(\vx,\vw):\ \vx\in\p\Omega\}$ into the in-flow boundary
$\Gamma^-$, the out-flow boundary $\Gamma^+$, and the grazing set
$\Gamma^0$ as
\begin{eqnarray}
\Gamma^{-}&=&\{(\vx,\vw):\ \vx\in\p\Omega,\ \vw\cdot\vec n<0\},\\
\Gamma^{+}&=&\{(\vx,\vw):\ \vx\in\p\Omega,\ \vw\cdot\vec n>0\},\\
\Gamma^{0}&=&\{(\vx,\vw):\ \vx\in\p\Omega,\ \vw\cdot\vec n=0\}.
\end{eqnarray}
It is easy to see $\Gamma=\Gamma^+\cup\Gamma^-\cup\Gamma^0$. Hence,
the boundary condition is only given on $\Gamma^{-}$.

\subsection{Interior Expansion}

We define the interior expansion as follows:
\begin{eqnarray}\label{interior expansion}
\uc(\vx,\vw)\sim\sum_{k=0}^{\infty}\e^k\uc_k(\vx,\vw),
\end{eqnarray}
where $\u_k$ can be defined by comparing the order of $\e$ via
plugging (\ref{interior expansion}) into the equation
(\ref{transport}). Thus, we have
\begin{eqnarray}
\uc_0-\buc_0&=&0,\label{expansion temp 1}\\
\uc_1-\buc_1&=&-\vw\cdot\nx\uc_0,\label{expansion temp 2}\\
\uc_2-\buc_2&=&-\vw\cdot\nx\uc_1,\label{expansion temp 3}\\
\ldots\nonumber\\
\uc_k-\buc_k&=&-\vw\cdot\nx\uc_{k-1}.
\end{eqnarray}
\ \\
The following analysis reveals the equation satisfied by
$\uc_k$:\\
Plugging (\ref{expansion temp 1}) into (\ref{expansion temp 2}), we
obtain
\begin{eqnarray}
\uc_1=\buc_1-\vw\cdot\nx\buc_0.\label{expansion temp 4}
\end{eqnarray}
Plugging (\ref{expansion temp 4}) into (\ref{expansion temp 3}), we
get
\begin{eqnarray}\label{expansion temp 13}
\uc_2-\buc_2=-\vw\cdot\nx(\buc_1-\vw\cdot\nx\buc_0)=-\vw\cdot\nx\buc_1+\vw^2\Delta_x\buc_0+2w_1w_2\p_{x_1x_2}\buc_0.
\end{eqnarray}
Integrating (\ref{expansion temp 13}) over $\vw\in\s^1$, we achieve
the final form
\begin{eqnarray}
\Delta_x\buc_0=0,
\end{eqnarray}
which further implies $\uc_0(\vx,\vw)$ satisfies the equation
\begin{eqnarray}\label{interior 1}
\left\{
\begin{array}{rcl}
\uc_0&=&\buc_0\\
\Delta_x\uc_0&=&0
\end{array}
\right.
\end{eqnarray}
Similarly, we can derive $\uc_k(\vx,\vw)$ for $k\geq1$ satisfies
\begin{eqnarray}\label{interior 2}
\left\{
\begin{array}{rcl}
\uc_k&=&\buc_k-\vw\cdot\nx\uc_{k-1}\\
\Delta_x\buc_k&=&0
\end{array}
\right.
\end{eqnarray}

\subsection{Milne Expansion}

In order to determine the boundary condition fir $\uc_k$, it is
well-known that we need to define the boundary layer expansion. Hence, we need several substitutions:\\
\ \\
Substitution 1:\\
We consider the substitution into quasi-polar coordinates
$u^{\e}(x_1,x_2,w_1,w_2)\rt u^{\e}(\mu,\theta,w_1,w_2)$ with
$(\mu,\theta,w_1,w_2)\in [0,1)\times[-\pi,\pi)\times\s^1$ defined as
\begin{eqnarray}\label{substitution 1}
\left\{
\begin{array}{rcl}
x_1&=&(1-\mu)\cos\theta,\\
x_2&=&(1-\mu)\sin\theta,\\
w_1&=&w_1,\\
w_2&=&w_2.
\end{array}
\right.
\end{eqnarray}
Here $\mu$ denotes the distance to the boundary $\p\Omega$ and
$\theta$ is the space angular variable. In these new variables,
equation (\ref{transport}) can be rewritten as
\begin{eqnarray}
\left\{ \begin{array}{l}\displaystyle
-\e\bigg(w_1\cos\theta+w_2\sin\theta\bigg)\frac{\p
u^{\e}}{\p\mu}-\frac{\e}{1-\mu}\bigg(w_1\sin\theta-w_2\cos\theta\bigg)\frac{\p
u^{\e}}{\p\theta}+u^{\e}-\frac{1}{2\pi}\int_{\s^1}u^{\e}\ud{\vw}=0,\\\rule{0ex}{1.0em}
u^{\e}(0,\theta,w_1,w_2)=g(\theta,w_1,w_2)\ \ \text{for}\ \
w_1\cos\theta+w_2\sin\theta<0.
\end{array}
\right.
\end{eqnarray}
\ \\
Substitution 2:\\
We further define the stretched variable $\eta$ by making the
scaling transform for $u^{\e}(\mu,\theta,w_1,w_2)\rt
u^{\e}(\eta,\theta,w_1,w_2)$ with $(\eta,\theta,w_1,w_2)\in
[0,1/\e)\times[-\pi,\pi)\times\s^1$ as
\begin{eqnarray}\label{substitution 2}
\left\{
\begin{array}{rcl}
\eta&=&\mu/\e,\\
\theta&=&\theta,\\
w_1&=&w_1,\\
w_2&=&w_2,
\end{array}
\right.
\end{eqnarray}
which implies
\begin{eqnarray}
\frac{\p u^{\e}}{\p\mu}=\frac{1}{\e}\frac{\p u^{\e}}{\p\eta}.
\end{eqnarray}
Then equation (\ref{transport}) is transformed into
\begin{eqnarray}
\left\{ \begin{array}{l}\displaystyle
-\bigg(w_1\cos\theta+w_2\sin\theta\bigg)\frac{\p
u^{\e}}{\p\eta}-\frac{\e}{1-\e\eta}\bigg(w_1\sin\theta-w_2\cos\theta\bigg)\frac{\p
u^{\e}}{\p\theta}+u^{\e}-\frac{1}{2\pi}\int_{\s^1}u^{\e}\ud{\vw}=0,\\\rule{0ex}{1.0em}
u^{\e}(0,\theta,w_1,w_2)=g(\theta,w_1,w_2)\ \ \text{for}\ \
w_1\cos\theta+w_2\sin\theta<0.
\end{array}
\right.
\end{eqnarray}
\ \\
Substitution 3:\\
Define the velocity substitution for $u^{\e}(\eta,\theta,w_1,w_2)\rt
u^{\e}(\eta,\theta,\xi)$ with $(\eta,\theta,\xi)\in
[0,1/\e)\times[-\pi,\pi)\times[-\pi,\pi)$ as
\begin{eqnarray}\label{substitution 3}
\left\{
\begin{array}{rcl}
\eta&=&\eta,\\
\theta&=&\theta,\\
w_1&=&-\sin\xi,\\
w_2&=&-\cos\xi.
\end{array}
\right.
\end{eqnarray}
Here $\xi$ denotes the velocity angular variable. We have the
succinct form for (\ref{transport}) as
\begin{eqnarray}\label{classical temp}
\left\{ \begin{array}{l}\displaystyle \sin(\theta+\xi)\frac{\p
u^{\e}}{\p\eta}-\frac{\e}{1-\e\eta}\cos(\theta+\xi)\frac{\p
u^{\e}}{\p\theta}+u^{\e}-\frac{1}{2\pi}\int_{-\pi}^{\pi}u^{\e}\ud{\xi}=0,\\\rule{0ex}{1.0em}
u^{\e}(0,\theta,\xi)=g(\theta,\xi)\ \ \text{for}\ \
\sin(\theta+\xi)>0.
\end{array}
\right.
\end{eqnarray}
\ \\
We now define the Milne expansion of boundary layer as follows:
\begin{eqnarray}\label{classical expansion}
\ubc(\eta,\theta,\phi)\sim\sum_{k=0}^{\infty}\e^k\ubc_k(\eta,\theta,\phi),
\end{eqnarray}
where $\ubc_k$ can be determined by comparing the order of $\e$ via
plugging (\ref{classical expansion}) into the equation
(\ref{classical temp}). Thus, in a neighborhood of the boundary, we
have
\begin{eqnarray}
\sin(\theta+\xi)\frac{\p
\ubc_0}{\p\eta}+\ubc_0-\bubc_0&=&0,\label{cexpansion temp 5}\\
\sin(\theta+\xi)\frac{\p
\ubc_1}{\p\eta}+\ubc_1-\bubc_1&=&\frac{1}{1-\e\eta}\cos(\theta+\xi)\frac{\p
\ubc_0}{\p\theta},\label{cexpansion temp 6}\\
\ldots\nonumber\\
\sin(\theta+\xi)\frac{\p
\ubc_k}{\p\eta}+\ubc_k-\bubc_k&=&\frac{1}{1-\e\eta}\cos(\theta+\xi)\frac{\p
\ubc_{k-1}}{\p\theta},
\end{eqnarray}
where
\begin{eqnarray}
\bar
\ubc_k(\eta,\theta)=\frac{1}{2\pi}\int_{-\pi}^{\pi}\ubc_k(\eta,\theta,\xi)\ud{\xi}.
\end{eqnarray}
The construction of $\uc_k$ and $\ubc_k$ in \cite{book1} can be
summarized as follows:\\
\ \\
Step 1: Construction of $\ubc_0$ and $\uc_0$.\\
Assume the cut-off function $\psi$ and $\psi_0$ are defined as
\begin{eqnarray}\label{cut-off 1}
\psi(\mu)=\left\{
\begin{array}{ll}
1&0\leq\mu\leq1/2,\\
0&3/4\leq\mu\leq\infty.
\end{array}
\right.
\end{eqnarray}
\begin{eqnarray}\label{cut-off 2}
\psi_0(\mu)=\left\{
\begin{array}{ll}
1&0\leq\mu\leq1/4,\\
0&3/8\leq\mu\leq\infty.
\end{array}
\right.
\end{eqnarray}
Then the zeroth order boundary layer solution is defined as
\begin{eqnarray}\label{classical temp 1}
\left\{
\begin{array}{rcl}
\ubc_0(\eta,\theta,\xi)&=&\psi_0(\e\eta)\bigg(\f_0(\eta,\theta,\xi)-f_0(\infty,\theta)\bigg),\\
\sin(\theta+\xi)\dfrac{\p \f_0}{\p\eta}+\f_0-\bar \f_0&=&0,\\
\f_0(0,\theta,\xi)&=&g(\theta,\xi)\ \ \text{for}\ \
\sin(\theta+\xi)>0,\\\rule{0ex}{1em}
\lim_{\eta\rt\infty}\f_0(\eta,\theta,\xi)&=&f_0(\infty,\theta).
\end{array}
\right.
\end{eqnarray}
Assuming $g\in L^{\infty}(\Gamma^-)$, by Theorem \ref{Milne theorem
1}, we can show there exists a unique solution
$\f_0(\eta,\theta,\xi)\in
L^{\infty}([0,\infty)\times[-\pi,\pi)\times[-\pi,\pi))$. Hence,
$\ubc_0$ is well-defined. Then we can define the zeroth order
interior solution as
\begin{eqnarray}\label{classical temp 2}
\left\{
\begin{array}{rcl}
\uc_0&=&\buc_0,\\\rule{0ex}{1em} \Delta_x\buc_0&=&0\ \ \text{in}\ \
\Omega,\\\rule{0ex}{1em} \buc_0&=&f_0(\infty,\theta)\ \ \text{on}\ \
\p\Omega.
\end{array}
\right.
\end{eqnarray}
\ \\
Step 2: Construction of $\ubc_1$ and $\uc_1$. \\
Define the first order boundary layer solution as
\begin{eqnarray}\label{classical temp 3}
\left\{
\begin{array}{rcl}
\ubc_1(\eta,\theta,\xi)&=&\psi_0(\e\eta)\bigg(\f_1(\eta,\theta,\xi)-f_1(\infty,\theta)\bigg),\\
\sin(\theta+\xi)\dfrac{\p \f_1}{\p\eta}+\f_1-\bar
\f_1&=&\cos(\theta+\xi)\dfrac{\psi(\e\eta)}{1-\e\eta}\dfrac{\p
\ubc_0}{\p\theta},\\\rule{0ex}{1em}
\f_1(0,\theta,\xi)&=&\vw\cdot\nx\uc_0(\vx_0,\vw)\ \ \text{for}\ \
\sin(\theta+\xi)>0,\\\rule{0ex}{1em}
\lim_{\eta\rt\infty}\f_1(\eta,\theta,\xi)&=&f_1(\infty,\theta).
\end{array}
\right.
\end{eqnarray}
where $(\vx_0,\vw)$ is the same point as $(0,\theta,\xi)$. Define
the first order interior solution as
\begin{eqnarray}\label{classical temp 5}
\left\{
\begin{array}{rcl}
\uc_1&=&\buc_1-\vw\cdot\nx\uc_0,\\\rule{0ex}{1em}
\Delta_x\buc_1&=&0\ \ \text{in}\ \ \Omega,\\\rule{0ex}{1em}
\buc_1&=&f_1(\infty,\theta)\ \ \text{on}\ \ \p\Omega.
\end{array}
\right.
\end{eqnarray}
\ \\
Step 3: Generalization to arbitrary $k$.\\
Similar to above procedure, we can define the $k^{th}$ order
boundary layer solution as
\begin{eqnarray}
\left\{
\begin{array}{rcl}
\ubc_k(\eta,\theta,\xi)&=&\psi_0(\e\eta)\bigg(\f_k(\eta,\theta,\xi)-f_k(\infty,\theta)\bigg),\\
\sin(\theta+\xi)\dfrac{\p \f_k}{\p\eta}+\f_k-\bar
\f_k&=&\cos(\theta+\xi)\dfrac{\psi(\e\eta)}{1-\e\eta}\dfrac{\p
\ubc_{k-1}}{\p\theta},\\\rule{0ex}{1em}
\f_k(0,\theta,\xi)&=&\vw\cdot\nx\uc_{k-1}(\vx_0,\vw)\ \ \text{for}\
\ \sin(\theta+\xi)>0,\\\rule{0ex}{1em}
\lim_{\eta\rt\infty}\f_k(\eta,\theta,\xi)&=&f_k(\infty,\theta).
\end{array}
\right.
\end{eqnarray}
Define the $k^{th}$ order interior solution as
\begin{eqnarray}
\left\{
\begin{array}{rcl}
\uc_k&=&\buc_k-\vw\cdot\nx\uc_{k-1},\\\rule{0ex}{1em}
\Delta_x\buc_k&=&0\ \ \text{in}\ \ \Omega,\\\rule{0ex}{1em}
\buc_k&=&f_k(\infty,\theta)\ \ \text{on}\ \ \p\Omega.
\end{array}
\right.
\end{eqnarray}
In \cite[pp.136]{book1}, the author proved the following result:
\begin{theorem}\label{main fake 1}
Assume $g(\vx_0,\vw)$ is sufficiently smooth. Then for the steady
neutron transport equation (\ref{transport}), the unique solution
$u^{\e}(\vx,\vw)\in L^{\infty}(\Omega\times\s^1)$ satisfies
\begin{eqnarray}\label{main fake theorem 1}
\lnm{u^{\e}-\uc_0-\ubc_0}=O(\e).
\end{eqnarray}
\end{theorem}
\ \\
The goal of our paper is to reexamine the validity of Theorem
\ref{main fake 1}. Our work begins with a crucial observation that
based on Remark \ref{Milne remark}, the existence of solution $\f_1$
requires the source term
\begin{eqnarray}
\cos(\theta+\xi)\frac{\psi}{1-\e\eta}\frac{\p \ubc_0}{\p\theta}\in
L^{\infty}([0,\infty)\times[-\pi,\pi)\times[-\pi,\pi)).
\end{eqnarray}
Since the support of $\psi(\e\eta)$ depends on $\e$, by
(\ref{classical temp 1}), this in turn requires
\begin{eqnarray}
\frac{\p
}{\p\theta}\bigg(\f_0(\eta,\theta,\xi)-f_0(\infty,\theta)\bigg)\in
L^{\infty}([0,\infty)\times[-\pi,\pi)\times[-\pi,\pi))
\end{eqnarray}
Note that $Z=\ps (\f_0-f_0(\infty,\theta))$ satisfies the equation
\begin{eqnarray}
\left\{
\begin{array}{rcl}
\sin(\theta+\xi)\dfrac{\p Z}{\p\eta}+Z-\bar
Z&=&-\cos(\theta+\xi)\dfrac{\p \f_0}{\p\eta},\\\rule{0ex}{2em}
Z(0,\theta,\xi)&=&\dfrac{\p g(\theta,\xi)}{\p\theta}-\dfrac{\p
f_0(\infty,\theta)}{\p\theta}\ \ \text{for}\ \
\sin(\theta+\xi)>0,\\\rule{0ex}{1.5em}
\lim_{\eta\rt\infty}Z(\eta,\theta,\xi)&=& Z(\infty,\theta).
\end{array}
\right.
\end{eqnarray}
In order for $Z\in
L^{\infty}([0,\infty)\times[-\pi,\pi)\times[-\pi,\pi))$, assuming
the boundary data $\ps g\in L^{\infty}(\Gamma^-)$, we require the
source term
\begin{eqnarray}
-\cos(\theta+\xi)\frac{\p \f_0}{\p\eta}\in
L^{\infty}([0,\infty)\times[-\pi,\pi)\times[-\pi,\pi)).
\end{eqnarray}
On the other hand, as shown by Lemma \ref{counter theorem 1}, we can
show for specific $g$, it holds that $\px\f_0\notin
L^{\infty}([0,\infty)\times[-\pi,\pi)\times[-\pi,\pi))$. Due to
intrinsic singularity for (\ref{classical temp 1}), the construction
in \cite{book1} breaks down.

In fact, in general geometry with curved boundary, we need to
control the normal derivative of the boundary layer solution for the
Milne expansion.

\subsection{$\e$-Milne Expansion with Geometric Correction}

Our main goal is to overcome the difficulty in estimating
\begin{eqnarray}\label{problematic}
\cos(\theta+\xi)\frac{\psi}{1-\e\eta}\frac{\p \ubc_k}{\p\theta}.
\end{eqnarray}
We introduce one more substitution to decompose the term (\ref{problematic}). \\
\ \\
Substitution 4:\\
We make the rotation substitution for $u^{\e}(\eta,\theta,\xi)\rt
u^{\e}(\eta,\theta,\phi)$ with $(\eta,\theta,\phi)\in
[0,1/\e)\times[-\pi,\pi)\times[-\pi,\pi)$ as
\begin{eqnarray}\label{substitution 4}
\left\{
\begin{array}{rcl}
\eta&=&\eta,\\
\theta&=&\theta,\\
\phi&=&\theta+\xi,
\end{array}
\right.
\end{eqnarray}
and transform the equation (\ref{transport}) into
\begin{eqnarray}\label{transport temp}
\left\{ \begin{array}{l}\displaystyle \sin\phi\frac{\p
u^{\e}}{\p\eta}-\frac{\e}{1-\e\eta}\cos\phi\bigg(\frac{\p
u^{\e}}{\p\phi}+\frac{\p
u^{\e}}{\p\theta}\bigg)+u^{\e}-\frac{1}{2\pi}\int_{-\pi}^{\pi}u^{\e}\ud{\phi}=0,\\\rule{0ex}{1.0em}
u^{\e}(0,\theta,\phi)=g(\theta,\phi)\ \ \text{for}\ \ \sin\phi>0.
\end{array}
\right.
\end{eqnarray}
Inspired by \cite{book8}, \cite{book7} and \cite{book10}, the most
important idea is to include the most singular term
\begin{eqnarray}
-\frac{\e}{1-\e\eta}\cos\phi\frac{\p u^{\e}}{\p\phi}
\end{eqnarray}
in the Milne problem.

We define the $\e$-Milne expansion with geometric correction of
boundary layer as follows:
\begin{eqnarray}\label{boundary layer expansion}
\ub(\eta,\theta,\phi)\sim\sum_{k=0}^{\infty}\e^k\ub_k(\eta,\theta,\phi),
\end{eqnarray}
where $\ub_k$ can be determined by comparing the order of $\e$ via
plugging (\ref{boundary layer expansion}) into the equation
(\ref{transport temp}). Thus, in a neighborhood of the boundary, we
have
\begin{eqnarray}
\sin\phi\frac{\p \ub_0}{\p\eta}+\frac{\e}{1-\e\eta}\cos\phi\frac{\p
\ub_0}{\p\phi}+\ub_0-\bub_0&=&0,\label{expansion temp 5}\\
\sin\phi\frac{\p \ub_1}{\p\eta}+\frac{\e}{1-\e\eta}\cos\phi\frac{\p
\ub_1}{\p\phi}+\ub_1-\bub_1&=&\frac{1}{1-\e\eta}\cos\phi\frac{\p
\ub_0}{\p\theta},\label{expansion temp 6}\\
\ldots\nonumber\\
\sin\phi\frac{\p \ub_k}{\p\eta}+\frac{\e}{1-\e\eta}\cos\phi\frac{\p
\ub_k}{\p\phi}+\ub_k-\bub_k&=&\frac{1}{1-\e\eta}\cos\phi\frac{\p
\ub_{k-1}}{\p\theta}.
\end{eqnarray}
where
\begin{eqnarray}
\bub_k(\eta,\theta)=\frac{1}{2\pi}\int_{-\pi}^{\pi}\ub_k(\eta,\theta,\phi)\ud{\phi}.
\end{eqnarray}
It is important to note the solution $\ub_k$ depends on $\e$.

We refer to the cut-off function $\psi$ and $\psi_0$ as
(\ref{cut-off 1}) and (\ref{cut-off 2}), and define the force as
\begin{eqnarray}\label{force}
F(\e;\eta)=-\frac{\e\psi(\e\eta)}{1-\e\eta},
\end{eqnarray}
Define the interior expansion as follows:
\begin{eqnarray}
\u(\vx,\vw)\sim\sum_{k=0}^{\infty}\e^k\u_k(\vx,\vw)
\end{eqnarray}
where $\u_k$ satisfies the same equations as $\uc_k$ in (\ref{interior 1}) and (\ref{interior 2}).
Here, to highlight its dependence on $\e$ via the $\e$-Milne problem and boundary data, we add the
superscript $\e$.

The bridge between the interior solution and the boundary layer
solution is the boundary condition of (\ref{transport}), so we first
consider the boundary condition expansion
\begin{eqnarray}
\u_0+\ub_0&=&g,\\
\u_1+\ub_1&=&0,\\
\ldots\nonumber\\
\u_k+\ub_k&=&0.
\end{eqnarray}
The construction of $\u_k$ and $\ub_k$ are as follows:\\
\ \\
Step 1: Construction of $\ub_0$ and $\u_0$.\\
Define the zeroth order boundary layer solution as
\begin{eqnarray}\label{expansion temp 9}
\left\{
\begin{array}{rcl}
\ub_0(\eta,\theta,\phi)&=&\psi_0(\e\eta)\bigg(f_0^{\e}(\eta,\theta,\phi)-f^{\e}_0(\infty,\theta)\bigg),\\
\sin\phi\dfrac{\p f_0^{\e}}{\p\eta}+F(\e;\eta)\cos\phi\dfrac{\p
f_0^{\e}}{\p\phi}+f_0^{\e}-\bar f_0^{\e}&=&0,\\
f_0^{\e}(0,\theta,\phi)&=&g(\theta,\phi)\ \ \text{for}\ \
\sin\phi>0,\\\rule{0ex}{1em}
\lim_{\eta\rt\infty}f_0^{\e}(\eta,\theta,\phi)&=&f^{\e}_0(\infty,\theta).
\end{array}
\right.
\end{eqnarray}
In contrast to the classical Milne problem (\ref{classical temp 1}),
the key advantage is, due to the geometry, $\dfrac{\p
F(\e;\eta)}{\p\theta}=0$, such that (\ref{expansion temp 9}) is
invariant in $\theta$.

Then we define the zeroth order interior solution $\u_0(\vx)$ as
\begin{eqnarray}\label{expansion temp 8}
\left\{
\begin{array}{rcl}
\u_0&=&\bu_0,\\\rule{0ex}{1em} \Delta_x\bu_0&=&0\ \ \text{in}\ \
\Omega,\\\rule{0ex}{1em} \bu_0&=&f_0^{\e}(\infty,\theta)\ \
\text{on}\ \ \p\Omega.
\end{array}
\right.
\end{eqnarray}
\ \\
Step 2: Estimates of $\dfrac{\p\ub_0}{\p\theta}$.\\
By Theorem \ref{Milne theorem 2}, we can easily see $f_0^{\e}$ is
well-defined in $L^{\infty}(\Omega\times\s^1)$ and approaches
$f_0^{\e}(\infty)$ exponentially fast as $\eta\rt\infty$. Then we
can naturally derive $Z=\ps (f_0^{\e}-f^{\e}_0(\infty))$ also
satisfies the same type of $\e$-Milne problem
\begin{eqnarray}
\left\{ \begin{array}{rcl}\displaystyle \sin\phi\frac{\p
Z}{\p\eta}+F(\e;\eta)\cos\phi\frac{\p
Z}{\p\phi}+Z-\bar Z&=&0,\\
Z(0,\theta,\phi)&=&\dfrac{\p g}{\p\theta}-\dfrac{\p
f_0^{\e}(\infty)}{\p\theta}\ \ \text{for}\ \
\sin\phi>0,\\\rule{0ex}{1em} \lim_{\eta\rt\infty}Z(\eta,\phi)&=&C.
\end{array}
\right.
\end{eqnarray}
By Theorem \ref{Milne theorem 2}, we can see $Z\rt C$ exponentially
fast as $\eta\rt\infty$. It is natural to obtain this constant $C$
must be zero. Hence, if $g\in C^r(\Gamma^-)$, it is obvious to check
$f_0^{\e}(\infty)\in C^r(\p\Omega)$. By the standard elliptic
estimate in (\ref{expansion temp 8}), there exists a unique solution
$\bu_0\in W^{r,p}(\Omega)$ for arbitrary $p\geq2$ satisfying
\begin{eqnarray}
\nm{\bu_0}_{W^{r,p}(\Omega)}\leq
C(\Omega)\nm{f_0^{\e}(\infty)}_{W^{r-1/p,p}(\p\Omega)},
\end{eqnarray}
which implies $\nx\bu_0\in W^{r-1,p}(\Omega)$, $\nx\bu_0\in
W^{r-1-1/p,p}(\p\Omega)$ and $\bu_0\in C^{r-1,1-2/p}(\Omega)$.\\
\ \\
Step 3: Construction of $\ub_1$ and $\u_1$.\\
Define the first order boundary layer solution as
\begin{eqnarray}\label{expansion temp 11}
\left\{
\begin{array}{rcl}
\ub_1(\eta,\theta,\phi)&=&\psi_0(\e\eta)\bigg(f_1^{\e}(\eta,\theta,\phi)-f_{1}^{\e}(\infty,\theta)\bigg),\\
\sin\phi\dfrac{\p f_1^{\e}}{\p\eta}+F(\e;\eta)\cos\phi\dfrac{\p
f_1^{\e}}{\p\phi}+f_1^{\e}-\bar
f_1^{\e}&=&\dfrac{\psi(\e\eta)}{1-\e\eta}\cos\phi\dfrac{\p
\ub_0}{\p\theta},\\\rule{0ex}{1em}
f_1^{\e}(0,\theta,\phi)&=&\vw\cdot\nx\u_0(\vx_0,\vw)\ \ \text{for}\
\ \sin\phi>0,\\\rule{0ex}{1em}
\lim_{\eta\rt\infty}f_1^{\e}(\eta,\theta,\phi)&=&f_{1}^{\e}(\infty,\theta).
\end{array}
\right.
\end{eqnarray}
where $(\vx_0,\vw)$ is the same point as $(0,\theta,\phi)$. Then we
define the first order interior solution $\u_1(\vx)$ as
\begin{eqnarray}\label{expansion temp 10}
\left\{
\begin{array}{rcl}
\u_1&=&\bu_1-\vw\cdot\nx\u_0,\\\rule{0ex}{1em} \Delta_x\bu_1&=&0\ \
\text{in}\ \ \Omega,\\\rule{0ex}{1em}
\bu_1&=&f^{\e}_{1}(\infty,\theta)\ \ \text{on}\ \ \p\Omega.
\end{array}
\right.
\end{eqnarray}
\ \\
Step 4: Estimates of $\dfrac{\p\ub_1}{\p\theta}$.\\
By Theorem \ref{Milne theorem 2}, we can easily see $f_1^{\e}$ is
well-defined in $L^{\infty}(\Omega\times\s^1)$ and approaches
$f_{1}^{\e}(\infty)$ exponentially fast as $\eta\rt\infty$. Also,
since $\vw\cdot\nx\u_0\in W^{r-1-1/p,p}(\p\Omega)$, $\ps f_1^{\e}$
is well-defined and decays exponentially fast. Hence,
$f^{\e}_{1}(\infty,\theta)\in W^{r-1-1/p,p}(\p\Omega)$. By the
standard elliptic estimate in (\ref{expansion temp 10}), there
exists a unique solution $\bu_1\in W^{r-1,p}(\Omega)$ and satisfies
\begin{eqnarray}
\nm{\bu_1}_{W^{r-1,p}(\Omega)}\leq
C(\Omega)\nm{f^{\e}_{1}(\infty)}_{W^{r-1-1/p,p}(\p\Omega)},
\end{eqnarray}
which implies $\nx\bu_1\in W^{r-2,p}(\Omega)$, $\nx\bu_1\in
W^{r-2-1/p,p}(\p\Omega)$ and $\bu_1\in C^{r-2,1-2/p}(\Omega)$.\\
\ \\
Step 5: Generalization to arbitrary $k$.\\
In a similar fashion, as long as $g$ is sufficiently smooth, above
process can go on. We construct the $k^{th}$ order boundary layer
solution as
\begin{eqnarray}
\left\{
\begin{array}{rcl}
\ub_k(\eta,\theta,\phi)&=&\psi_0(\e\eta)\bigg(f_k^{\e}(\eta,\theta,\phi)-f_{k}^{\e}(\infty,\theta)\bigg),\\
\sin\phi\dfrac{\p f_k^{\e}}{\p\eta}+F(\e;\eta)\cos\phi\dfrac{\p
f_k^{\e}}{\p\phi}+f_k^{\e}-\bar
f_k^{\e}&=&\dfrac{\psi(\e\eta)}{1-\e\eta}\cos\phi\dfrac{\p
\ub_{k-1}}{\p\theta},\\\rule{0ex}{1em}
f_k^{\e}(0,\theta,\phi)&=&\vw\cdot\nx\u_{k-1}(\vx_0,\vw)\ \
\text{for}\ \ \sin\phi>0,\\\rule{0ex}{1em}
\lim_{\eta\rt\infty}f_k^{\e}(\eta,\theta,\phi)&=&f_{k}^{\e}(\infty,\theta).
\end{array}
\right.
\end{eqnarray}
Then we define the $k^{th}$ order interior solution as
\begin{eqnarray}
\left\{
\begin{array}{rcl}
\u_k&=&\bu_k-\vw\cdot\nx\u_{k-1},\\\rule{0ex}{1em}
\Delta_x\bu_k&=&0\ \ \text{in}\ \ \Omega,\\\rule{0ex}{1em}
\bu_k&=&f^{\e}_{k}(\infty,\theta)\ \ \text{on}\ \ \p\Omega.
\end{array}
\right.
\end{eqnarray}
For $g\in C^{k+1}(\Gamma^-)$, the interior solution and boundary
layer solution can be well-defined up to $k^{th}$ order, i.e. up to
$\u_k$ and $\ub_k$.

\subsection{Main Results}

\begin{theorem}\label{main 1}
Assume $g(\vx_0,\vw)\in C^3(\Gamma^-)$. Then for the steady neutron
transport equation (\ref{transport}), the unique solution
$u^{\e}(\vx,\vw)\in L^{\infty}(\Omega\times\s^1)$ satisfies
\begin{eqnarray}\label{main theorem 1}
\lnm{u^{\e}-\u_0-\ub_0}=O(\e)
\end{eqnarray}
where $\u_0$ and $\ub_0$ are defined in (\ref{expansion temp 8}) and
(\ref{expansion temp 9}). Moreover, if $g(\theta,\phi)=\cos\phi$,
then there exists a $C>0$ such that
\begin{eqnarray}\label{main theorem 2}
\lnm{u^{\e}-\uc_0-\ubc_0}\geq C>0
\end{eqnarray}
when $\e$ is sufficiently small, where $\uc_0$ and $\ubc_0$ are
defined in (\ref{classical temp 2}) and (\ref{classical temp 1}).
\end{theorem}
For the diffusive boundary case, the zeroth order classical Knudsen
layer is always absent. Our method leads to the invalidity of the
classical Knudsen layer expansion at first order. See Theorem
\ref{main 2}.

Our results demonstrates that the classical Knudsen layer expansion
in Theorem \ref{main fake 1} breaks down in a unit disk in
$L^{\infty}$. Even though the equation (\ref{expansion temp 9}) only
modifies the equation (\ref{classical temp 1}) with an $\epsilon$
order term, the difference between the solutions can be order $1$ via
a contradiction argument in Step 5 of Section 4. We remark that
$L^{\infty}$ is a natural space to characterize boundary layer
contributions, whose $L^p$ norm is $O(\e^{1/p})$ for $1\leq
p<\infty$. Unfortunately, our new expansion can only be established
in a disk, because the $\theta$ derivative can be controlled due to
the constant curvature along the boundary. However, a new
mathematical theory is needed to characterize such a diffusive limit
in a general domain. Our analysis is based on a careful study of the
$\e$-Milne problem with geometric correction. Our paper is
self-contained and without any use of probabilistic techniques as in
\cite{book1}.

Throughout this paper, $C>0$ denotes a constant that only depends on
the parameter $\Omega$, but does not depend on the data. It is
referred as universal and can change from one inequality to another.
When we write $C(z)$, it means a certain positive constant depending
on the quantity $z$.

Our paper is organized as follows: in Section 2, we first establish
the $L^{\infty}$ well-posedness of the equation (\ref{transport});
in Section 3, we give a complete analysis of the $\e$-Milne problem
with geometric correction; in Section 4, we give the detailed proof
of Theorem \ref{main 1} and finally, in Section 5, we discuss the
case of diffusive boundary.

\section{Well-posedness of Steady Neutron Transport Equation}

In this section, we consider the well-posedness of the
steady neutron transport equation
\begin{eqnarray}
\left\{ \begin{array}{rcl} \epsilon\vec w\cdot\nabla_xu+u-\bar
u&=&f(\vx,\vw)\ \ \text{in}\ \ \Omega\label{neutron},\\\rule{0ex}{1.0em} u(\vec x_0,\vec
w)&=&g(\vec x_0,\vec w)\ \ \text{for}\ \ \vec x_0\in\p\Omega\ \
\text{and}\ \vw\cdot\vec n<0.
\end{array}
\right.
\end{eqnarray}
We define the $L^2$ and $L^{\infty}$ norms in $\Omega\times\s^1$ as
usual:
\begin{eqnarray}
\nm{f}_{L^2(\Omega\times\s^1)}&=&\bigg(\int_{\Omega}\int_{\s^1}\abs{f(\vx,\vw)}^2\ud{\vw}\ud{\vx}\bigg)^{1/2},\\
\nm{f}_{L^{\infty}(\Omega\times\s^1)}&=&\sup_{(\vx,\vw)\in\Omega\times\s^1}\abs{f(\vx,\vw)}.
\end{eqnarray}
Define the $L^2$ and $L^{\infty}$ norms on the boundary as follows:
\begin{eqnarray}
\nm{f}_{L^2(\Gamma)}&=&\bigg(\iint_{\Gamma}\abs{f(\vx,\vw)}^2\abs{\vw\cdot\vec n}\ud{\vw}\ud{\vx}\bigg)^{1/2},\\
\nm{f}_{L^2(\Gamma^{\pm})}&=&\bigg(\iint_{\Gamma^{\pm}}\abs{f(\vx,\vw)}^2\abs{\vw\cdot\vec
n}\ud{\vw}\ud{\vx}\bigg)^{1/2},\\
\nm{f}_{L^{\infty}(\Gamma)}&=&\sup_{(\vx,\vw)\in\Gamma}\abs{f(\vx,\vw)},\\
\nm{f}_{L^{\infty}(\Gamma^{\pm})}&=&\sup_{(\vx,\vw)\in\Gamma^{\pm}}\abs{f(\vx,\vw)}.
\end{eqnarray}

\subsection{Preliminaries}

In order to show the $L^2$ and $L^{\infty}$ well-posedness of the
equation (\ref{neutron}), we start with some preparations with the
penalized neutron transport equation.
\begin{lemma}\label{well-posedness lemma 1}
Assume $f(\vx,\vw)\in L^{\infty}(\Omega\times\s^1)$ and
$g(x_0,\vw)\in L^{\infty}(\Gamma^-)$. Then for the penalized
transport equation
\begin{eqnarray}\label{penalty equation}
\left\{
\begin{array}{rcl}
\lambda u_{\l}+\epsilon\vec
w\cdot\nabla_xu_{\l}+u_{\l}&=&f(\vx,\vw)\ \ \text{in}\ \ \Omega,\\\rule{0ex}{1.0em} u_{\l}(\vec x_0,\vec w)&=&g(\vec
x_0,\vec w)\ \ \text{for}\ \ \vec x_0\in\p\Omega\ \ \text{and}\
\vw\cdot\vec n<0.
\end{array}
\right.
\end{eqnarray}
with $\l>0$ as a penalty parameter, there exists a solution
$u_{\l}(\vx,\vw)\in L^{\infty}(\Omega\times\s^1)$ satisfying
\begin{eqnarray}
\im{u_{\l}}{\Omega\times\s^1}\leq
\im{f}{\Omega\times\s^1}+\im{g}{\Gamma^-}.
\end{eqnarray}
\end{lemma}
\begin{proof}
The characteristics $(X(s),W(s))$ of the equation (\ref{penalty
equation}) which goes through $(\vx,\vw)$ is defined by
\begin{eqnarray}\label{character}
\left\{
\begin{array}{rcl}
(X(0),W(0))&=&(\vx,\vw)\\\rule{0ex}{2.0em}
\dfrac{\ud{X(s)}}{\ud{s}}&=&\e W(s),\\\rule{0ex}{2.0em}
\dfrac{\ud{W(s)}}{\ud{s}}&=&0.
\end{array}
\right.
\end{eqnarray}
which implies
\begin{eqnarray}
\left\{
\begin{array}{rcl}
X(s)&=&\vx+(\e\vw)s\\
W(s)&=&\vw
\end{array}
\right.
\end{eqnarray}
Hence, we can rewrite the equation (\ref{penalty equation}) along the
characteristics as
\begin{eqnarray}\label{well-posedness temp 31}
u_{\l}(\vx,\vw)&=& g(\vx-\e
t_b\vw,\vw)e^{-(1+\l)t_b}+\int_{0}^{t_b}f(\vx-\e(t_b-s)\vw,\vw)e^{-(1+\l)(t_b-s)}\ud{s},
\end{eqnarray}
where the backward exit time $t_b$ is defined as
\begin{equation}\label{exit time}
t_b(\vx,\vw)=\inf\{t\geq0: (\vx-\e t\vw,\vw)\in\Gamma^-\}.
\end{equation}
Then we can naturally estimate
\begin{eqnarray}
\im{u_{\l}}{\Omega\times\s^1}&\leq&e^{-(1+\l)t_b}\im{g}{\Gamma^-}+\frac{1-e^{(1+\l)t_b}}{1+\l}\im{f}{\Omega\times\s^1}\\
&\leq&\im{g}{\Gamma^-}+\im{f}{\Omega\times\s^1}\nonumber.
\end{eqnarray}
Since $u_{\l}$ can be explicitly traced back to the boundary data,
the existence naturally follows from above estimate.
\end{proof}
\begin{lemma}\label{well-posedness lemma 2}
Assume $f(\vx,\vw)\in L^{\infty}(\Omega\times\s^1)$ and
$g(x_0,\vw)\in L^{\infty}(\Gamma^-)$. Then for the penalized neutron
transport equation
\begin{eqnarray}
\left\{ \begin{array}{rcl} \lambda u_{\l}+\epsilon\vec
w\cdot\nabla_xu_{\l}+u_{\l}-\bar u_{\l}&=& f(\vx,\vw)\ \ \text{in}\ \ \Omega\label{well-posedness penalty
equation},\\\rule{0ex}{1.0em} u_{\l}(\vec x_0,\vec w)&=&g(\vec
x_0,\vec w)\ \ \text{for}\ \ \vec x_0\in\p\Omega\ \ \text{and}\
\vw\cdot\vec n<0.
\end{array}
\right.
\end{eqnarray}
with $\l>0$, there exists a solution $u_{\l}(\vx,\vw)\in
L^{\infty}(\Omega\times\s^1)$ satisfying
\begin{eqnarray}
\im{u_{\l}}{\Omega\times\s^1}\leq \frac{1+\l}{\l}\bigg(
\im{f}{\Omega\times\s^1}+\im{g}{\Gamma^-}\bigg).
\end{eqnarray}
\end{lemma}
\begin{proof}
We define an approximating sequence $\{u_{\l}^k\}_{k=0}^{\infty}$,
where $u_{\l}^0=0$ and
\begin{eqnarray}\label{penalty temp 1}
\left\{
\begin{array}{rcl}
\lambda u_{\l}^{k}+\epsilon\vec w\cdot\nabla_xu_{\l}^k+u_{\l}^k-\bar
u_{\l}^{k-1}&=&f(\vx,\vw)\ \ \text{in}\ \ \Omega,\\\rule{0ex}{1.0em}
u_{\l}^k(\vec x_0,\vec w)&=&g(\vec x_0,\vec w)\ \ \text{for}\ \ \vec
x_0\in\p\Omega\ \ \text{and}\ \vw\cdot\vec n<0.
\end{array}
\right.
\end{eqnarray}
By Lemma \ref{well-posedness lemma 1}, this sequence is well-defined
and $\im{u_{\l}^k}{\Omega\times\s^1}<\infty$.

The characteristics and the backward exit time are defined as
(\ref{character}) and (\ref{exit time}), so we rewrite equation
(\ref{penalty temp 1}) along the characteristics as
\begin{eqnarray}
u_{\l}^k(\vx,\vw)&=& g(\vx-\e
t_b\vw,\vw)e^{-(1+\l)t_b}+\int_{0}^{t_b}(f+\bar
u_{\l}^{k-1})(\vx-\e(t_b-s)\vw,\vw)e^{-(1+\l)(t_b-s)}\ud{s}.
\end{eqnarray}
We define the difference $v^k=u_{\l}^{k}-u_{\l}^{k-1}$ for $k\geq1$.
Recall (\ref{average 1}) for $\bar v^k$, then $v^k$ satisfies
\begin{eqnarray}
v^{k+1}(\vx,\vw)&=&\int_{0}^{t_b}\bar
v^{k}(\vx-\e(t_b-s)\vw,\vw)e^{-(1+\l)(t_b-s)}\ud{s}.\nonumber
\end{eqnarray}
Since $\im{\bar
v^k}{\Omega\times\s^1}\leq\im{v^k}{\Omega\times\s^1}$, we can
directly estimate
\begin{eqnarray}
\im{v^{k+1}}{\Omega\times\s^1}&\leq&\im{v^{k}}{\Omega\times\s^1}\int_0^{t_b}e^{-(1+\l)(t_b-s)}\ud{s}
\leq\frac{1-e^{-(1+\l)t_b}}{1+\l}\im{v^{k}}{\Omega\times\s^1}.
\end{eqnarray}
Hence, we naturally have
\begin{eqnarray}
\im{v^{k+1}}{\Omega\times\s^1}&\leq&\frac{1}{1+\l}\im{v^{k}}{\Omega\times\s^1}.
\end{eqnarray}
Thus, this is a contraction sequence for $\l>0$. Considering
$v^1=u_{\l}^1$, we have
\begin{eqnarray}
\im{v^{k}}{\Omega\times\s^1}\leq\bigg(\frac{1}{1+\l}\bigg)^{k-1}\im{u^{1}_{\l}}{\Omega\times\s^1},
\end{eqnarray}
for $k\geq1$. Therefore, $u_{\l}^k$ converges strongly in
$L^{\infty}$ to a limit solution $u_{\l}$ satisfying
\begin{eqnarray}\label{well-posedness temp 1}
\im{u_{\l}}{\Omega\times\s^1}\leq\sum_{k=1}^{\infty}\im{v^{k}}{\Omega\times\s^1}\leq\frac{1+\l}{\l}\im{u_{\l}^1}{\Omega\times\s^1}.
\end{eqnarray}
Since $u_{\l}^1$ can be rewritten along the characteristics as
\begin{eqnarray}
u_{\l}^1(\vx,\vw)&=&g(\vx-\e
t_b\vw,\vw)e^{-(1+\l)t_b}+\int_{0}^{t_b}f(\vx-\e(t_b-s)\vw,\vw)e^{-(1+\l)(t_b-s)}\ud{s},
\end{eqnarray}
based on Lemma \ref{well-posedness lemma 1}, we can directly
estimate
\begin{eqnarray}\label{well-posedness temp 2}
\im{u_{\l}^1}{\Omega\times\s^1}\leq
\im{f}{\Omega\times\s^1}+\im{g}{\Gamma^-}.
\end{eqnarray}
Combining (\ref{well-posedness temp 1}) and (\ref{well-posedness
temp 2}), we can easily deduce the lemma.
\end{proof}

\subsection{$L^2$ Estimate}

It is easy to see when $\l\rt0$, the estimate in Lemma
\ref{well-posedness lemma 2} blows up. Hence, we need to show a
uniform estimate of the solution to the penalized neutron transport
equation (\ref{well-posedness penalty equation}).
\begin{lemma}(Green's Identity)\label{well-posedness lemma 3}
Assume $f(\vx,\vw),\ g(\vx,\vw)\in L^2(\Omega\times\s^1)$ and
$\vw\cdot\nx f,\ \vw\cdot\nx g\in L^2(\Omega\times\s^1)$ with $f,\
g\in L^2(\Gamma)$. Then
\begin{eqnarray}
\iint_{\Omega\times\s^1}\bigg((\vw\cdot\nx f)g+(\vw\cdot\nx
g)f\bigg)\ud{\vx}\ud{\vw}=\int_{\Gamma}fg\ud{\gamma},
\end{eqnarray}
where $\ud{\gamma}=(\vw\cdot\vec n)\ud{s}$ on the boundary.
\end{lemma}
\begin{proof}
See \cite[Chapter 9]{book6}.
\end{proof}
\begin{lemma}\label{well-posedness lemma 4}
The solution $u_{\l}$ to the equation (\ref{well-posedness penalty
equation}) satisfies the uniform estimate
\begin{eqnarray}\label{well-posedness temp 3}
\e\tm{\bar u_{\l}}{\Omega\times\s^1}\leq C(\Omega)\bigg(
\tm{u_{\l}-\bar
u_{\l}}{\Omega\times\s^1}+\tm{f}{\Omega\times\s^1}+\e\tm{u_{\l}}{\Gamma^{+}}+\e\tm{g}{\Gamma^-}\bigg),
\end{eqnarray}
for $0\leq\l<<1$ and $0<\e<<1$.
\end{lemma}
\begin{proof}
Applying Lemma \ref{well-posedness lemma 3} to the solution of the
equation (\ref{well-posedness penalty equation}). Then for any
$\phi\in L^2(\Omega\times\s^1)$ satisfying $\vw\cdot\nx\phi\in
L^2(\Omega\times\s^1)$ and $\phi\in L^2(\Gamma)$, we have
\begin{eqnarray}\label{well-posedness temp 4}
\l\iint_{\Omega\times\s^1}u_{\l}\phi+\e\int_{\Gamma}u_{\l}\phi\ud{\gamma}
-\e\iint_{\Omega\times\s^1}(\vw\cdot\nx\phi)u_{\l}+\iint_{\Omega\times\s^1}(u_{\l}-\bar
u_{\l})\phi=\iint_{\Omega\times\s^1}f\phi.
\end{eqnarray}
Our goal is to choose a particular test function $\phi$. We first
construct an auxiliary function $\zeta$. Since $u_{\l}\in
L^{\infty}(\Omega\times\s^1)$, it naturally implies $\bar u_{\l}\in
L^{\infty}(\Omega)$ which further leads to $\bar u_{\l}\in
L^2(\Omega)$. We define $\zeta(\vx)$ on $\Omega$ satisfying
\begin{eqnarray}\label{test temp 1}
\left\{
\begin{array}{rcl}
\Delta \zeta&=&\bar u_{\l}\ \ \text{in}\ \
\Omega,\\\rule{0ex}{1.0em} \zeta&=&0\ \ \text{on}\ \ \p\Omega.
\end{array}
\right.
\end{eqnarray}
In the bounded domain $\Omega$, based on the standard elliptic
estimate, we have
\begin{eqnarray}\label{test temp 3}
\nm{\zeta}_{H^2(\Omega)}\leq C(\Omega)\nm{\bar
u_{\l}}_{L^2(\Omega)}.
\end{eqnarray}
We plug the test function
\begin{eqnarray}\label{test temp 2}
\phi=-\vw\cdot\nx\zeta
\end{eqnarray}
into the weak formulation (\ref{well-posedness temp 4}) and estimate
each term there. Naturally, we have
\begin{eqnarray}\label{test temp 4}
\nm{\phi}_{L^2(\Omega)}\leq C\nm{\zeta}_{H^1(\Omega)}\leq
C(\Omega)\nm{\bar u_{\l}}_{L^2(\Omega)}.
\end{eqnarray}
Easily we can decompose
\begin{eqnarray}\label{test temp 5}
-\e\iint_{\Omega\times\s^1}(\vw\cdot\nx\phi)u_{\l}&=&-\e\iint_{\Omega\times\s^1}(\vw\cdot\nx\phi)\bar
u_{\l}-\e\iint_{\Omega\times\s^1}(\vw\cdot\nx\phi)(u_{\l}-\bar
u_{\l}).
\end{eqnarray}
We estimate the two term on the right-hand side separately. By
(\ref{test temp 1}) and (\ref{test temp 2}), we have
\begin{eqnarray}\label{wellposed temp 1}
-\e\iint_{\Omega\times\s^1}(\vw\cdot\nx\phi)\bar
u_{\l}&=&\e\iint_{\Omega\times\s^1}\bar
u_{\l}\bigg(w_1(w_1\p_{11}\zeta+w_2\p_{12}\zeta)+w_2(w_1\p_{12}\zeta+w_2\p_{22}\zeta)\bigg)\\
&=&\e\iint_{\Omega\times\s^1}\bar
u_{\l}\bigg(w_1^2\p_{11}\zeta+w_2^2\p_{22}\zeta\bigg)\nonumber\\
&=&\e\pi\int_{\Omega}\bar u_{\l}(\p_{11}\zeta+\p_{22}\zeta)\nonumber\\
&=&\e\pi\nm{\bar u_{\l}}_{L^2(\Omega)}^2\nonumber\\
&=&\half\e\nm{\bar u_{\l}}_{L^2(\Omega\times\s^1)}^2\nonumber.
\end{eqnarray}
In the second equality, above cross terms vanish due to the symmetry
of the integral over $\s^1$. On the other hand, for the second term
in (\ref{test temp 5}), H\"older's inequality and the elliptic
estimate imply
\begin{eqnarray}\label{wellposed temp 2}
-\e\iint_{\Omega\times\s^1}(\vw\cdot\nx\phi)(u_{\l}-\bar
u_{\l})&\leq&C(\Omega)\e\nm{u_{\l}-\bar u_{\l}}_{L^2(\Omega\times\s^1)}\nm{\zeta}_{H^2(\Omega)}\\
&\leq&C(\Omega)\e\nm{u_{\l}-\bar
u_{\l}}_{L^2(\Omega\times\s^1)}\nm{\bar
u_{\l}}_{L^2(\Omega\times\s^1)}\nonumber.
\end{eqnarray}
Based on (\ref{test temp 3}), (\ref{test temp 4}), the boundary
condition of the penalized neutron transport equation
(\ref{well-posedness penalty equation}), the trace theorem,
H\"older's inequality and the elliptic estimate, we have
\begin{eqnarray}\label{wellposed temp 3}
\e\int_{\Gamma}u_{\l}\phi\ud{\gamma}&=&\e\int_{\Gamma^+}u_{\l}\phi\ud{\gamma}+\e\int_{\Gamma^-}u_{\l}\phi\ud{\gamma}\\
&\leq&C(\Omega)\bigg(\e\nm{u_{\l}}_{L^2(\Gamma^+)}\nm{\bar
u_{\l}}_{L^2(\Omega\times\s^1)}+\e\tm{g}{\Gamma^-}\nm{\bar
u_{\l}}_{L^2(\Omega\times\s^1)}\bigg)\nonumber,
\end{eqnarray}
\begin{eqnarray}\label{wellposed temp 4}
\l\iint_{\Omega\times\s^1}u_{\l}\phi&=&\l\iint_{\Omega\times\s^1}\bar
u_{\l}\phi+\l\iint_{\Omega\times\s^1}(u_{\l}-\bar u_{\l})\phi=\l\iint_{\Omega\times\s^1}(u_{\l}-\bar u_{\l})\phi\\
&\leq&C(\Omega)\l\nm{\bar
u_{\l}}_{L^2(\Omega\times\s^1)}\nm{u_{\l}-\bar
u_{\l}}_{L^2(\Omega\times\s^1)}\nonumber,
\end{eqnarray}
\begin{eqnarray}\label{wellposed temp 5}
\iint_{\Omega\times\s^1}(u_{\l}-\bar u_{\l})\phi\leq
C(\Omega)\nm{\bar u_{\l}}_{L^2(\Omega\times\s^1)}\nm{u_{\l}-\bar
u_{\l}}_{L^2(\Omega\times\s^1)},
\end{eqnarray}
\begin{eqnarray}\label{wellposed temp 6}
\iint_{\Omega\times\s^1}f\phi\leq C(\Omega)\nm{\bar
u_{\l}}_{L^2(\Omega\times\s^1)}\nm{f}_{L^2(\Omega\times\s^1)}.
\end{eqnarray}
Collecting terms in (\ref{wellposed temp 1}), (\ref{wellposed temp
2}), (\ref{wellposed temp 3}), (\ref{wellposed temp 4}),
(\ref{wellposed temp 5}) and (\ref{wellposed temp 6}), we obtain
\begin{eqnarray}
\\
\e\nm{\bar u_{\l}}_{L^2(\Omega\times\s^1)}\leq
C(\Omega)\bigg((1+\e+\l)\nm{u_{\l}-\bar
u_{\l}}_{L^2(\Omega\times\s^1)}+\e\nm{u_{\l}}_{L^2(\Gamma^+)}+\nm{f}_{L^2(\Omega\times\s^1)}+\e\tm{g}{\Gamma^-}\bigg)\nonumber,
\end{eqnarray}
When $0\leq\l<1$ and $0<\e<1$, we get the desired uniform estimate
with respect to $\lambda$.
\end{proof}
\begin{theorem}\label{LT estimate}
Assume $f(\vx,\vw)\in L^{\infty}(\Omega\times\s^1)$ and
$g(x_0,\vw)\in L^{\infty}(\Gamma^-)$. Then for the steady neutron
transport equation (\ref{neutron}), there exists a unique solution
$u(\vx,\vw)\in L^2(\Omega\times\s^1)$ satisfying
\begin{eqnarray}
\tm{u}{\Omega\times\s^1}\leq C(\Omega)\bigg(
\frac{1}{\e^2}\tm{f}{\Omega\times\s^1}+\frac{1}{\e^{1/2}}\tm{g}{\Gamma^-}\bigg).
\end{eqnarray}
\end{theorem}
\begin{proof}
In the weak formulation (\ref{well-posedness temp 4}), we may take
the test function $\phi=u_{\l}$ to get the energy estimate
\begin{eqnarray}
\l\nm{u_{\l}}_{L^2(\Omega\times\s^1)}^2+\half\e\int_{\Gamma}\abs{u_{\l}}^2\ud{\gamma}+\nm{u_{\l}-\bar
u_{\l}}_{L^2(\Omega\times\s^1)}^2=\iint_{\Omega\times\s^1}fu_{\l}.
\end{eqnarray}
Hence, this naturally implies
\begin{eqnarray}\label{well-posedness temp 5}
\half\e\nm{u_{\l}}_{L^2(\Gamma^+)}^2+\nm{u_{\l}-\bar
u_{\l}}_{L^2(\Omega\times\s^1)}^2\leq\iint_{\Omega\times\s^1}fu_{\l}+\half\e\nm{g}_{L^2(\Gamma^-)}^2.
\end{eqnarray}
On the other hand, we can square on both sides of
(\ref{well-posedness temp 3}) to obtain
\begin{eqnarray}\label{well-posedness temp 6}
\\
\e^2\tm{\bar u_{\l}}{\Omega\times\s^1}^2\leq C(\Omega)\bigg(
\tm{u_{\l}-\bar
u_{\l}}{\Omega\times\s^1}^2+\tm{f}{\Omega\times\s^1}^2+\e^2\tm{u_{\l}}{\Gamma^{+}}^2+\e^2\tm{g}{\Gamma^-}^2\bigg).\nonumber
\end{eqnarray}
Multiplying a sufficiently small constant on both sides of
(\ref{well-posedness temp 6}) and adding it to (\ref{well-posedness
temp 5}) to absorb $\nm{u_{\l}}_{L^2(\Gamma^+)}^2$ and
$\nm{u_{\l}-\bar u_{\l}}_{L^2(\Omega\times\s^1)}^2$, we deduce
\begin{eqnarray}
&&\e\nm{u_{\l}}_{L^2(\Gamma^+)}^2+\e^2\nm{\bar
u_{\l}}_{L^2(\Omega\times\s^1)}^2+\nm{u_{\l}-\bar
u_{\l}}_{L^2(\Omega\times\s^1)}^2\\&&\qquad\qquad\qquad\leq
C(\Omega)\bigg(\tm{f}{\Omega\times\s^1}^2+
\iint_{\Omega\times\s^1}fu_{\l}+\e\nm{g}_{L^2(\Gamma^-)}^2\bigg).\nonumber
\end{eqnarray}
Hence, we have
\begin{eqnarray}\label{well-posedness temp 7}
\e\nm{u_{\l}}_{L^2(\Gamma^+)}^2+\e^2\nm{u_{\l}}_{L^2(\Omega\times\s^1)}^2\leq
C(\Omega)\bigg(\tm{f}{\Omega\times\s^1}^2+
\iint_{\Omega\times\s^1}fu_{\l}+\e\nm{g}_{L^2(\Gamma^-)}^2\bigg).
\end{eqnarray}
A simple application of Cauchy's inequality leads to
\begin{eqnarray}
\iint_{\Omega\times\s^1}fu_{\l}\leq\frac{1}{4C\e^2}\tm{f}{\Omega\times\s^1}^2+C\e^2\tm{u_{\l}}{\Omega\times\s^1}^2.
\end{eqnarray}
Taking $C$ sufficiently small, we can divide (\ref{well-posedness
temp 7}) by $\e^2$ to obtain
\begin{eqnarray}\label{well-posedness temp 21}
\frac{1}{\e}\nm{u_{\l}}_{L^2(\Gamma^+)}^2+\nm{u_{\l}}_{L^2(\Omega\times\s^2)}^2\leq
C(\Omega)\bigg(
\frac{1}{\e^4}\nm{f}_{L^2(\Omega\times\s^2)}^2+\frac{1}{\e}\nm{g}_{L^2(\Gamma^-)}^2\bigg).
\end{eqnarray}
Since above estimate does not depend on $\l$, it gives a uniform
estimate for the penalized neutron transport equation
(\ref{well-posedness penalty equation}). Thus, we can extract a
weakly convergent subsequence $u_{\l}\rt u$ as $\l\rt0$. The weak
lower semi-continuity of norms $\nm{\cdot}_{L^2(\Omega\times\s^2)}$
and $\nm{\cdot}_{L^2(\Gamma^+)}$ implies $u$ also satisfies the
estimate (\ref{well-posedness temp 21}). Hence, in the weak
formulation (\ref{well-posedness temp 4}), we can take $\l\rt0$ to
deduce that $u$ satisfies equation (\ref{neutron}). Also $u_{\l}-u$
satisfies the equation
\begin{eqnarray}
\left\{
\begin{array}{rcl}
\epsilon\vec w\cdot\nabla_x(u_{\l}-u)+(u_{\l}-u)-(\bar u_{\l}-\bar
u)&=&-\l u_{\l}\ \ \text{in}\ \
\Omega\label{remainder},\\\rule{0ex}{1.0em} (u_{\l}-u)(\vec x_0,\vec
w)&=&0\ \ \text{for}\ \ \vec x_0\in\p\Omega\ \ and\ \vw\cdot\vec
n<0.
\end{array}
\right.
\end{eqnarray}
By a similar argument as above, we can achieve
\begin{eqnarray}
\nm{u_{\l}-u}_{L^2(\Omega\times\s^2)}^2\leq
C(\Omega)\bigg(\frac{\l}{\e^4}\nm{u_{\l}}_{L^2(\Omega\times\s^2)}^2\bigg).
\end{eqnarray}
When $\l\rt0$, the right-hand side approaches zero, which implies
the convergence is actually in the strong sense. The uniqueness
easily follows from the energy estimates.
\end{proof}

\subsection{$L^{\infty}$ Estimate}

\begin{theorem}\label{LI estimate}
Assume $f(\vx,\vw)\in L^{\infty}(\Omega\times\s^1)$ and
$g(x_0,\vw)\in L^{\infty}(\Gamma^-)$. Then for the neutron transport
equation (\ref{neutron}), there exists a unique solution
$u(\vx,\vw)\in L^{\infty}(\Omega\times\s^1)$ satisfying
\begin{eqnarray}
\im{u}{\Omega\times\s^1}\leq C(\Omega)\bigg(
\frac{1}{\e^{3}}\im{f}{\Omega\times\s^1}+\frac{1}{\e^{3/2}}\im{g}{\Gamma^-}\bigg).
\end{eqnarray}
\end{theorem}
\begin{proof}
We divide the proof into several steps to bootstrap an $L^2$ solution to an $L^{\infty}$ solution:\\
\ \\
Step 1: Double Duhamel iterations.\\
The characteristics of the equation (\ref{neutron}) is given by
(\ref{character}). Hence, we can rewrite the equation
(\ref{neutron}) along the characteristics as
\begin{eqnarray}
u(\vx,\vw)&=&g(\vx-\e t_b\vw,\vw)e^{-t_b}+\int_{0}^{t_b}f(\vx-\e(t_b-s)\vw,\vw)e^{-(t_b-s)}\ud{s}\\
&&+\frac{1}{2\pi}\int_{0}^{t_b}\bigg(\int_{\s^1}u(\vx-\e(t_b-s)\vw,\vw_t)\ud{\vw_t}\bigg)e^{-(t_b-s)}\ud{s}\nonumber.
\end{eqnarray}
where the backward exit time $t_b$ is defined as (\ref{exit time}). Note we have replaced $\bar u$ by the integral of $u$ over the dummy velocity
variable $\vw_t$. For the last term in this formulation, we apply
the Duhamel's principle again to
$u(\vx-\e(t_b-s)\vw,\vw_t)$ and obtain
\begin{eqnarray}\label{well-posedness temp 8}
\\
u(\vx,\vw)&=&g(\vx-\e t_b\vw,\vw)e^{-t_b}+\int_{0}^{t_b}f(\vx-\e(t_b-s)\vw,\vw)e^{-(t_b-s)}\ud{s}\nonumber\\
&&+\frac{1}{2\pi}\int_{0}^{t_b}\int_{\s^1}g(\vx-\e(t_b-s)\vw-\e
s_b\vw_t,\vw_t)e^{-s_b}\ud{\vw_t}e^{-(t_b-s)}\ud{s}\nonumber\\
&&+\frac{1}{2\pi}\int_{0}^{t_b}\int_{\s^1}\bigg(\int_{0}^{s_b}f(\vx-\e(t_b-s)\vw-\e
(s_b-r)\vw_t,\vw_t)e^{-(s_b-r)}\ud{r}\bigg)\ud{\vw_t}e^{-(t_b-s)}\ud{s}\nonumber\\
&&+\bigg(\frac{1}{2\pi}\bigg)^2\int_{0}^{t_b}\int_{\s^1}e^{-(t_b-s)}\bigg(\int_{0}^{s_b}\int_{\s^1}e^{-(s_b-r)}u(\vx-\e(t_b-s)\vw-\e
(s_b-r)\vw_t,\vw_s)\ud{\vw_s}\ud{r}\bigg)\ud{\vw_t}\ud{s}\nonumber,
\end{eqnarray}
where we introduce another dummy velocity variable $\vw_s$ and
\begin{eqnarray}
s_b(\vx,\vw,s,\vw_t)=\inf\{r\geq0: (\vx-\e(t_b-s)\vw-\e
r\vw_t,\vw_t)\in\Gamma^-\}.
\end{eqnarray}
\ \\
Step 2: Estimates of all but the last term in (\ref{well-posedness temp 8}).\\
We can directly estimate as follows:
\begin{eqnarray}\label{im temp 1}
\abs{g(\vx-\e t_b\vw,\vw)e^{-t_b}}\leq\im{g}{\Gamma^-},
\end{eqnarray}
\begin{eqnarray}\label{im temp 2}
\abs{\frac{1}{2\pi}\int_{0}^{t_b}\int_{\s^1}g(\vx-\e(t_b-s)\vw-\e
s_b\vw_t,\vw_t)e^{-s_b}\ud{\vw_t}e^{-(t_b-s)}\ud{s}} \leq
\im{g}{\Gamma^-},
\end{eqnarray}
\begin{eqnarray}\label{im temp 3}
\abs{\int_{0}^{t_b}f(\vx-\e(t_b-s)\vw,\vw)e^{-(t_b-s)}\ud{s}}\leq
\im{f}{\Omega\times\s^1},
\end{eqnarray}
\begin{eqnarray}\label{im temp 4}
\\
\abs{\frac{1}{2\pi}\int_{0}^{t_b}\int_{\s^1}\bigg(\int_{0}^{s_b}f(\vx-\e(t_b-s)\vw-\e
(s_b-r)\vw_t,\vw_t)e^{-(s_b-r)}\ud{r}\bigg)\ud{\vw_t}e^{-(t_b-s)}\ud{s}}
\leq \im{f}{\Omega\times\s^1}\nonumber.
\end{eqnarray}
\ \\
Step 3: Estimates of the last term in (\ref{well-posedness temp 8}).\\
Now we decompose the last term in (\ref{well-posedness temp 8}) as
\begin{eqnarray}
\int_{0}^{t_b}\int_{\s^1}\int_0^{s_b}\int_{\s^1}=\int_{0}^{t_b}\int_{\s^1}\int_{s_b-r\leq\delta}\int_{\s^1}+
\int_{0}^{t_b}\int_{\s^1}\int_{s_b-r\geq\delta}\int_{\s^1}=I_1+I_2,
\end{eqnarray}
for some $\delta>0$. We can estimate $I_1$ directly as
\begin{eqnarray}\label{im temp 5}
I_1
&\leq&\int_{0}^{t_b}e^{-(t_b-s)}\bigg(\int_{\max(0,s_b-\delta)}^{s_b}\im{u}{\Omega\times\s^1}\ud{r}\bigg)\ud{s}\leq\delta\im{u}{\Omega\times\s^1}.
\end{eqnarray}
Then we can bound $I_2$ as
\begin{eqnarray}
I_2&\leq&C\int_{0}^{t_b}\int_{\s^1}\int_{0}^{\max(0,s_b-\delta)}\int_{\s^1}\abs{u(\vx-\e(t_b-s)\vw-\e
(s_b-r)\vw_t,\vw_s)}e^{-(t_b-s)}\ud{\vw_s}\ud{r}\ud{\vw_t}\ud{s}.
\end{eqnarray}
By the definition of $t_b$ and $s_b$, we always have
$\vx-\e(t_b-s)\vw-\e (s_b-r)\vw_t\in\bar\Omega$. Hence, we may
interchange the order of integration and apply H\"older's inequality
to obtain
\begin{eqnarray}\label{well-posedness temp 22}
I_2&\leq&C\int_{0}^{t_b}\int_{0}^{\max(0,s_b-\delta)}\int_{\s^1}\int_{\s^1}{\bf{1}}_{\Omega}(\vx-\e(t_b-s)\vw-\e
(s_b-r)\vw_t)\\
&&\abs{u(\vx-\e(t_b-s)\vw-\e
(s_b-r)\vw_t,\vw_s)}e^{-(t_b-s)}\ud{\vw_t}\ud{\vw_s}\ud{r}\ud{s}\nonumber\\
&\leq&C\int_{0}^{t_b}\int_{\s^1}\bigg(\int_{0}^{\max(0,s_b-\delta)}\int_{\s^1}{\bf{1}}_{\Omega}(\vx-\e(t_b-s)\vw-\e
(s_b-r)\vw_t)\nonumber\\
&&\abs{u(\vx-\e(t_b-s)\vw-\e
(s_b-r)\vw_t,\vw_s)}^2\ud{\vw_t}\ud{r}\bigg)^{1/2}e^{-(t_b-s)}\ud{\vw_s}\ud{s}\nonumber.
\end{eqnarray}
Note $\vw_t\in\s^1$, which is essentially a one-dimensional
variable. Thus, we may write it in a new variable $\psi$ as
$\vw_t=(\cos\psi,\sin\psi)$. Then we define the change of variable
$[-\pi,\pi)\times\r\rt \Omega: (\psi,r)\rt(y_1,y_2)=\vec
y=\vx-\e(t_b-s)\vw-\e (s_b-r)\vw_t$, i.e.
\begin{eqnarray}
\left\{
\begin{array}{rcl}
y_1&=&x_1-\e(t_b-s)w_1-\e (s_b-r)\cos\psi,\\
y_2&=&x_2-\e(t_b-s)w_2-\e (s_b-r)\sin\psi.
\end{array}
\right.
\end{eqnarray}
Therefore, for $s_b-r\geq\delta$, we can directly compute the Jacobian
\begin{eqnarray}
\abs{\frac{\p{(y_1,y_2)}}{\p{(\psi,r)}}}=\abs{\abs{\begin{array}{rc}
-\e(s_b-r)\sin\psi&\e\cos\psi\\
\e(s_b-r)\cos\psi&\e\sin\psi
\end{array}}}=\e^2(s_b-r)\geq\e^2\delta.
\end{eqnarray}
Hence, we may simplify (\ref{well-posedness temp 22}) as
\begin{eqnarray}
I_2&\leq&\frac{C}{\e\sqrt{\delta}}\int_{0}^{t_b}\int_{\s^1}\bigg(\int_{\Omega}\abs{u(\vec
y,\vw_s)}^2\ud{\vec y}\bigg)^{1/2}e^{-(t_b-s)}\ud{\vw_s}\ud{s}.
\end{eqnarray}
Then we may further utilize Cauchy's inequality and the $L^2$
estimate of $u$ in Theorem \ref{LT estimate} to obtain
\begin{eqnarray}\label{im temp 6}
I_2&\leq&\frac{C}{\e\sqrt{\delta}}\int_{0}^{t_b}\bigg(\int_{\s^1}\int_{\Omega}\abs{u(\vec
y,\vw_s)}^2\ud{\vec y}\ud{\vw_s}\bigg)^{1/2}e^{-(t_b-s)}\ud{s}\\
&=&\frac{C}{\e\sqrt{\delta}}\int_{0}^{t_b}e^{-(t_b-s)}\tm{u}{\Omega\times\s^1}\ud{s}\nonumber\\
&\leq&\frac{C}{\e\sqrt{\delta}}\tm{u}{\Omega\times\s^1}\nonumber\\
&\leq&\frac{C(\Omega)}{\sqrt{\delta}}\bigg(\frac{1}{\e^{3}}\tm{f}{\Omega\times\s^1}+\frac{1}{\e^{3/2}}\tm{g}{\Gamma^-}\bigg)\nonumber\\
&\leq&\frac{C(\Omega)}{\sqrt{\delta}}\bigg(\frac{1}{\e^{3}}\im{f}{\Omega\times\s^1}+\frac{1}{\e^{3/2}}\im{g}{\Gamma^-}\bigg)\nonumber.
\end{eqnarray}
\ \\
In summary, collecting (\ref{im temp 1}), (\ref{im temp 2}),
(\ref{im temp 3}), (\ref{im temp 4}), (\ref{im temp 5}) and (\ref{im
temp 6}), for fixed $0<\delta<1$, we have
\begin{eqnarray}
\abs{u(\vx,\vw)}\leq \delta
\im{u}{\Omega\times\s^1}+\frac{C(\Omega)}{\sqrt{\delta}}
\bigg(\frac{1}{\e^{3}}\im{f}{\Omega\times\s^1}+\frac{1}{\e^{3/2}}\im{g}{\Gamma^-}\bigg).
\end{eqnarray}
Then we may take $0<\delta\leq1/2$ to obtain
\begin{eqnarray}
\abs{u(\vx,\vw)}\leq
\half\im{u}{\Omega\times\s^1}+\frac{C(\Omega)}{\sqrt{\delta}}
\bigg(\frac{1}{\e^{3}}\im{f}{\Omega\times\s^1}+\frac{1}{\e^{3/2}}\im{g}{\Gamma^-}\bigg).
\end{eqnarray}
Taking supremum of $u$ over all $(\vx,\vw)$, we have
\begin{eqnarray}
\im{u}{\Omega\times\s^1}\leq
\half\im{u}{\Omega\times\s^1}+\frac{C(\Omega)}{\sqrt{\delta}}
\bigg(\frac{1}{\e^{3}}\im{f}{\Omega\times\s^1}+\frac{1}{\e^{3/2}}\im{g}{\Gamma^-}\bigg).
\end{eqnarray}
Finally, absorbing $\im{u}{\Omega\times\s^1}$, for fixed $0<\delta\leq1/2$, we get
\begin{eqnarray}
\im{u}{\Omega\times\s^1}\leq C(\Omega)\bigg(
\frac{1}{\e^{3}}\im{f}{\Omega\times\s^1}+\frac{1}{\e^{3/2}}\im{g}{\Gamma^-}\bigg).
\end{eqnarray}

\end{proof}

\subsection{Well-posedness of Transport Equation}

\begin{theorem}\label{well-posedness 2}
Assume $g(x_0,\vw)\in L^{\infty}(\Gamma^-)$. Then for the steady
neutron transport equation (\ref{transport}), there exists a unique
solution $u^{\e}(\vx,\vw)\in L^{\infty}(\Omega\times\s^1)$
satisfying
\begin{eqnarray}
\im{u^{\e}}{\Omega\times\s^1}\leq
C(\Omega)\frac{1}{\e^{3/2}}\im{g}{\Gamma^-}.
\end{eqnarray}
\end{theorem}
\begin{proof}
We can apply Theorem \ref{LI estimate} to the equation
(\ref{transport}). The result naturally follows.
\end{proof}

\section{$\e$-Milne Problem}

We consider the $\e$-Milne problem for $f^{\e}(\eta,\theta,\phi)$ in
the domain
$(\eta,\theta,\phi)\in[0,\infty)\times[-\pi,\pi)\times[-\pi,\pi)$
\begin{eqnarray}\label{Milne problem}
\left\{
\begin{array}{rcl}\displaystyle
\sin\phi\frac{\p f^{\e}}{\p\eta}+F(\e;\eta)\cos\phi\frac{\p
f^{\e}}{\p\phi}+f^{\e}-\bar f^{\e}&=&S^{\e}(\eta,\theta,\phi),\\
f^{\e}(0,\theta,\phi)&=&h^{\e}(\theta,\phi)\ \ \text{for}\ \ \sin\phi>0,\\
\lim_{\eta\rt\infty}f^{\e}(\eta,\theta,\phi)&=&f^{\e}_{\infty}(\theta),
\end{array}
\right.
\end{eqnarray}
where
\begin{eqnarray}\label{Milne average}
\bar f^{\e}(\eta,\theta)=\frac{1}{2\pi}\int_{-\pi}^{\pi}f^{\e}(\eta,\theta,\phi)\ud{\phi},
\end{eqnarray}
$F(\e;\eta)$ is defined as (\ref{force}),
\begin{eqnarray}\label{Milne bounded}
\abs{h^{\e}(\theta,\phi)}\leq M,
\end{eqnarray}
and
\begin{eqnarray}\label{Milne decay}
\abs{S^{\e}(\eta,\theta,\phi)}\leq Ce^{-K\eta},
\end{eqnarray}
for $M$ and $K$ uniform in $\e$ and $\theta$.

We may further define a potential function $V(\e;\eta)$ satisfying
$V(\e;0)=0$ and $F(\e;\eta)=-\p_{\eta}V(\e,\eta)$. The following
lemma illustrates the main properties of $F$ and $V$.
\begin{lemma}\label{Milne force}
$V(\e;\eta)$ is monotonically increasing with respect to
$\eta$ and satisfies
\begin{eqnarray}\label{force temp 1}
0\leq V(\e;\eta)\leq \ln4.
\end{eqnarray}
Define
\begin{eqnarray}
V_{\infty}(\e)=\lim_{\eta\rt\infty}V(\e;\eta).
\end{eqnarray}
Then $V_{\infty}(\e)=V_{\infty}$ independent of
$\e$.
For any $0\leq y\leq z\leq\infty$
\begin{eqnarray}\label{force temp 2}
-\ln4\leq\int_{y}^{z}F(\e;\eta)\ud{\eta}\leq 0.
\end{eqnarray}
For any $\sigma>0$ and $\eta\geq0$, we have
\begin{eqnarray}\label{force temp 6}
e^{V(\eta+\sigma)-V(\eta)}\leq 1+4\e\sigma.
\end{eqnarray}
Moreover, we have
\begin{eqnarray}\label{force temp 3}
\int_0^{\infty}\int_{\eta}^{\infty}\abs{F(y)}^2\ud{y}\ud{\eta}\leq
3-\ln4.
\end{eqnarray}
\begin{eqnarray}\label{force temp 5}
\int_0^{\infty}\abs{F(\e;\eta)}^2\ud{\eta}\leq 3\e.
\end{eqnarray}
\begin{eqnarray}\label{force temp 4}
\lnm{F(\e;\eta)}&\leq& 4\e.
\end{eqnarray}
\end{lemma}
\begin{proof}
Since $F(\e;\eta)\leq0$, by definition, we know $V(\e;\eta)$ is
monotonically increasing with respect to $\eta$. Since $V(\e;0)=0$,
$V(\e;\eta)\geq0$. Then based on (\ref{cut-off 1}), we can direct
estimate
\begin{eqnarray}
V(\e;\eta)&=&V(\e;0)-\int_0^{\eta}F(\e;y)\ud{y}=\int_0^{\eta}\frac{\e\psi(\e y)}{1-\e y}\ud{y}\\
&\leq&\int_0^{\frac{3}{4\e}}\frac{\e}{1-\e y}\ud{y}=-\ln(1-\e
y)\bigg|_{y=0}^{y=\frac{3}{4\e}}=\ln4\nonumber.
\end{eqnarray}
This verifies (\ref{force temp 1}).

Also, letting $\mu=\e\eta$, we have
\begin{eqnarray}
V_{\infty}(\e)&=&-\int_0^{\infty}\frac{\e\psi(\e\eta)}{1-\e\eta}\ud{\eta}=-\int_0^{\infty}\frac{\psi(\e\eta)}{1-\e\eta}\ud{(\e\eta)}=
-\int_0^{\infty}\frac{\psi(\mu)}{1-\mu}\ud{\mu}\leq \ln4.
\end{eqnarray}
Hence, we have $V_{\infty}(\e)=V_{\infty}$ independent of $\e$.
Since
\begin{eqnarray}
\int_{y}^{z}F(\e;\eta)\ud{\eta}=V(\e;y)-V(\e;z),
\end{eqnarray}
we can naturally obtain (\ref{force temp 2}).

Moreover, we have
\begin{eqnarray}
V(\eta+\sigma)-V(\eta)&=&-\int_{\eta}^{\eta+\sigma}F(\e;y)\ud{y}=\int_{\eta}^{\eta+\sigma}\frac{\e\psi(\e y)}{1-\e y}\ud{y}\\
&\leq&\int_{\eta}^{\min\{(\eta+\sigma),\frac{3}{4\e}\}}\frac{\e\psi(\e
y)}{1-\e y}\ud{y}\leq-\ln(1-\e
y)\bigg|_{y=\eta}^{y=\min\{(\eta+\sigma),\frac{3}{4\e}\}}.
\end{eqnarray}
If $\eta+\sigma\leq\frac{3}{4\e}$, we have
\begin{eqnarray}
e^{V(\eta+\sigma)-V(\eta)}&\leq&\frac{1-\e\eta}{1-\e(\eta+\sigma)}=1+\frac{\e\sigma}{1-\e(\eta+\sigma)}\leq
1+4\e\sigma.
\end{eqnarray}
On the other hand, if $\eta+\sigma\geq\frac{3}{4\e}$ and
$\eta\leq\frac{3}{4\e}$, we define $\sigma'=\frac{3}{4\e}-\eta$ to
obtain
\begin{eqnarray}
e^{V(\eta+\sigma)-V(\eta)}&\leq&e^{V(\eta+\sigma')-V(\eta)}\leq
1+4\e\sigma'\leq1+4\e\sigma.
\end{eqnarray}
Finally, if $\eta+\sigma\geq\frac{3}{4\e}$ and
$\eta\geq\frac{3}{4\e}$, we have
\begin{eqnarray}
e^{V(\eta+\sigma)-V(\eta)}=e^0=1.
\end{eqnarray}
Therefore, (\ref{force temp 6}) follows.

Furthermore, considering the cut-off function (\ref{cut-off 1}), we
have
\begin{eqnarray}
\int_0^{\infty}\int_{\eta}^{\infty}\abs{F(y)}^2\ud{y}\ud{\eta}&\leq&\int_0^{\infty}\int_{\eta}^{\infty}\abs{\frac{\e\psi(\e
y)}{1-\e y}}^2\ud{y}\ud{\eta}\leq\int_0^{\frac{3}{4\e}}\int_{\eta}^{\frac{3}{4\e}}\frac{\e^2}{(1-\e y)^2}\ud{y}\ud{\eta}\\
&=&\int_0^{\frac{3}{4\e}}\bigg(\frac{\e}{1-\e
y}\bigg|_{y=\eta}^{y=\frac{3}{4\e}}\bigg)\ud{\eta}=
\int_0^{\frac{3}{4\e}}\bigg(4\e-\frac{\e}{1-\e\eta}\bigg)\ud{\eta}\nonumber\\
&=&\bigg(4\e\eta+\ln(1-\e\eta)\bigg)\bigg|_{\eta=0}^{\eta=\frac{3}{4\e}}=3-\ln4\nonumber.
\end{eqnarray}
Also, we can directly estimate
\begin{eqnarray}
\int_0^{\infty}\abs{F(\e;\eta)}^2\ud{\eta}&=&\int_0^{\infty}\frac{\e^2\psi^2(\e\eta)}{(1-\e\eta)^2}\ud{\eta}
\leq\int_0^{\frac{3}{4\e}}\frac{\e^2}{(1-\e\eta)^2}\ud{\eta}\\
&=&\frac{\e}{1-\e\eta}\bigg|_{\eta=0}^{\eta=\frac{3}{4\e}}=3\e\nonumber.
\end{eqnarray}
Finally, based on the definition of $F$ and $\psi$, we obtain
\begin{eqnarray}
\abs{F(\e;\eta)}=\abs{\frac{\e\psi(\e\eta)}{1-\e\eta}}\leq
\frac{\e}{1-\e\frac{3}{4\e}}=4\e.
\end{eqnarray}
\end{proof}
For notational simplicity, we omit $\e$ and $\theta$ dependence in
$f^{\e}$ in this section. The same convention also applies to
$F(\e;\eta)$, $V(\e;\eta)$, $S^{\e}(\eta,\theta,\phi)$ and
$h^{\e}(\theta,\phi)$. However, our estimates are independent of
$\e$ and $\theta$.

In this section, we introduce some special notations to describe the
norms in the space $(\eta,\phi)\in[0,\infty)\times[-\pi,\pi)$.
Define the $L^2$ norm as follows:
\begin{eqnarray}
\tnm{f(\eta)}&=&\bigg(\int_{-\pi}^{\pi}\abs{f(\eta,\phi)}^2\ud{\phi}\bigg)^{1/2},\\
\tnnm{f}&=&\bigg(\int_0^{\infty}\int_{-\pi}^{\pi}\abs{f(\eta,\phi)}^2\ud{\phi}\ud{\eta}\bigg)^{1/2}.
\end{eqnarray}
Define the inner product in $\phi$ space
\begin{eqnarray}
\br{f,g}_{\phi}(\eta)=\int_{-\pi}^{\pi}f(\eta,\phi)g(\eta,\phi)\ud{\phi}.
\end{eqnarray}
Define the $L^{\infty}$ norm as follows:
\begin{eqnarray}
\lnm{f(\eta)}&=&\sup_{\phi\in[-\pi,\pi)}\abs{f(\eta,\phi)},\\
\lnnm{f}&=&\sup_{(\eta,\phi)\in[0,\infty)\times[-\pi,\pi)}\abs{f(\eta,\phi)},\\
\ltnm{f}&=&\sup_{\eta\in[0,\infty)}\bigg(\int_{-\pi}^{\pi}\abs{f(\eta,\phi)}^2\ud{\phi}\bigg)^{1/2}.
\end{eqnarray}
Since the boundary data $h(\phi)$ is only defined on $\sin\phi>0$,
we naturally extend above definitions on this half-domain as
follows:
\begin{eqnarray}
\tnm{h}&=&\bigg(\int_{\sin\phi>0}\abs{h(\phi)}^2\ud{\phi}\bigg)^{1/2},\\
\lnm{h}&=&\sup_{\sin\phi>0}\abs{h(\phi)}.
\end{eqnarray}
\begin{lemma}\label{Milne data}
We have
\begin{eqnarray}
\tnm{h}&\leq&C\lnm{h}\leq CM\\
\tnnm{S}&\leq&C\frac{M}{K}\\
\ltnm{S}&\leq&C\lnnm{S}\leq CM
\end{eqnarray}
\end{lemma}
\begin{proof}
They can be verified via direct computation, so we omit the proofs here.
\end{proof}

\subsection{$L^2$ Estimates}

\subsubsection{Finite Slab with $\bar S=0$}

Consider the $\e$-Milne problem for $f^{L}(\eta,\phi)$ in a finite slab $(\eta,\phi)\in[0,L]\times[-\pi,\pi)$
\begin{eqnarray}\label{Milne finite problem LT}
\left\{
\begin{array}{rcl}\displaystyle
\sin\phi\frac{\p f^{L}}{\p\eta}+F(\eta)\cos\phi\frac{\p
f^{L}}{\p\phi}+f^{L}-\bar f^{L}&=&S(\eta,\phi),\\
f^{L}(0,\phi)&=&h(\phi)\ \ \text{for}\ \ \sin\phi<0,\\
f^{L}(L,\phi)&=&f^{L}(L,R\phi),
\end{array}
\right.
\end{eqnarray}
where $R\phi=-\phi$ and $S$ satisfies $\bar S(\eta)=0$ for any
$\eta$. We may decompose the solution
\begin{eqnarray}
f^{L}(\eta,\phi)=q_f^{L}(\eta)+r_f^{L}(\eta,\phi),
\end{eqnarray}
where the hydrodynamical part $q_f^{L}$ is in the null space of the
operator $f^{L}-\bar f^{L}$, and the microscopic part $r_f^{L}$ is
the orthogonal complement, i.e.
\begin{eqnarray}\label{hydro}
q_f^{L}(\eta)=\frac{1}{2\pi}\int_{-\pi}^{\pi}f^{L}(\eta,\phi)\ud{\phi}\quad
r_f^{L}(\eta,\phi)=f^{L}(\eta,\phi)-q_f^{L}(\eta).
\end{eqnarray}
In the following, when there is no confusion, we simply write
$f^{L}=q^{L}+r^{L}$.
\begin{lemma}\label{Milne finite LT}
Assume $\bar S(\eta)=0$ for any $\eta\in[0,L]$ with (\ref{Milne bounded}) and (\ref{Milne decay}). Then there exists a
solution $f(\eta,\phi)$ to the finite slab problem (\ref{Milne
finite problem LT}) satisfying
\begin{eqnarray}
\label{Milne temp 1}
\int_0^L\tnm{r^{L}(\eta)}^2\ud{\eta}&\leq&C\bigg(M+\frac{M}{K}\bigg)^2<\infty,\\
\tnm{q^{L}(\eta)}^2&\leq&
C\bigg(1+M+\frac{M}{K}\bigg)^2\bigg(1+\eta^{1/2}+\tnm{r^{L}(\eta)}\bigg)\label{Milne
temp 2},\\
\br{\sin\phi,r^{L}}_{\phi}(\eta)&=&0,\label{Milne temp 3}
\end{eqnarray}
for arbitrary $\eta\in[0,L]$.
\end{lemma}
\begin{proof}
We divide the proof into several steps:\\
\ \\
Step 1: Assume $\lnnm{H}<\infty$ and $\lnm{h}<\infty$, then the
solution $f_{\l}(\eta,\phi)$ to the penalized $\e$-transport
equation
\begin{eqnarray}\label{Milne finite problem LT penalty.}
\left\{
\begin{array}{rcl}\displaystyle
\l f_{\l}+\sin\phi\frac{\p f_{\l}}{\p\eta}+F(\eta)\cos\phi\frac{\p
f_{\l}}{\p\phi}+f_{\l}&=&H(\eta,\phi),\\
f_{\l}(0,\phi)&=&h(\phi)\ \ \text{for}\ \ \sin\phi<0,\\
f_{\l}(L,\phi)&=&f_{\l}(L,R\phi).
\end{array}
\right.
\end{eqnarray}
satisfies
\begin{eqnarray}\label{Milne t 01}
\lnnm{f_{\l}}\leq \lnm{h}+\lnnm{H}.
\end{eqnarray}
\emph{The proof of (\ref{Milne t 01}):} To construct the solution of
(\ref{Milne finite problem LT penalty.}),  we define the energy as
\begin{eqnarray}
E(\eta,\phi)=\cos\phi e^{-V(\eta)}.
\end{eqnarray}
This curve with constant energy is the characteristics of the
equation (\ref{Milne finite problem LT penalty.}). Hence, on this
curve the equation can be simplified as follows:
\begin{eqnarray}
\l f_{\l}+\sin\phi\frac{\ud{f_{\l}}}{\ud{\eta}}+f_{\l}=H.
\end{eqnarray}
An implicit function
$\eta^+(\eta,\phi)$ can be determined through
\begin{eqnarray}
\abs{E(\eta,\phi)}=e^{-V(\eta^+)}.
\end{eqnarray}
which means $(\eta^+,\phi_0)$ with $\sin\phi_0=0$ is on the same characteristics as $(\eta,\phi)$.
Define the quantities for $0\leq\eta'\leq\eta^+$ as follows:
\begin{eqnarray}
\phi'(\phi,\eta,\eta')&=&\cos^{-1}(\cos\phi e^{V(\eta')-V(\eta)}),\\
R\phi'(\phi,\eta,\eta')&=&-\cos^{-1}(\cos\phi e^{V(\eta')-V(\eta)})=-\phi'(\phi,\eta,\eta'),
\end{eqnarray}
where the inverse trigonometric function can be defined
single-valued in the domain $[0,\pi)$ and the quantities are always well-defined due to the monotonicity of $V$.
Finally we put
\begin{eqnarray}
G_{\eta,\eta'}^{\l}(\phi)&=&\int_{\eta'}^{\eta}\frac{1+\l}{\sin(\phi'(\phi,\eta,\xi))}\ud{\xi}.
\end{eqnarray}
We can define the solution to (\ref{Milne finite problem LT
penalty.}) along the characteristics
as follows:\\
Case I:\\
For $\sin\phi>0$,
\begin{eqnarray}
f_{\l}(\eta,\phi)&=&h(\phi'(\phi,\eta,0))\exp(-G^{\l}_{\eta,0})
+\int_0^{\eta}\frac{H(\eta',\phi'(\phi,\eta,\eta'))}{\sin(\phi'(\phi,\eta,\eta'))}\exp(-G^{\l}_{\eta,\eta'})\ud{\eta'}.
\end{eqnarray}
Case II:\\
For $\sin\phi<0$ and $\abs{E(\eta,\phi)}\leq e^{-V(L)}$,
\begin{eqnarray}
\\
f_{\l}(\eta,\phi)&=&h(\phi'(\phi,\eta,0))\exp(-G^{\l}_{L,0}-G^{\l}_{L,\eta})\nonumber\\
&+&\bigg(\int_0^{L}\frac{H(\eta',\phi'(\phi,\eta,\eta'))}{\sin(\phi'(\phi,\eta,\eta'))}
\exp(-G^{\l}_{L,\eta'}-G^{\l}_{L,\eta})\ud{\eta'}
+\int_{\eta}^{L}\frac{H(\eta',R\phi'(\phi,\eta,\eta'))}{\sin(\phi'(\phi,\eta,\eta'))}\exp(G^{\l}_{\eta,\eta'})\ud{\eta'}\bigg)\nonumber.
\end{eqnarray}
Case III:\\
For $\sin\phi<0$ and $\abs{E(\eta,\phi)}\geq e^{-V(L)}$,
\begin{eqnarray}
\\
f_{\l}(\eta,\phi)&=&h(\phi'(\phi,\eta,0))\exp(-G^{\l}_{\eta^+,0}-G^{\l}_{\eta^+,\eta})\nonumber\\
&+&\bigg(\int_0^{\eta^+}\frac{H(\eta',\phi'(\phi,\eta,\eta'))}{\sin(\phi'(\phi,\eta,\eta'))}
\exp(-G^{\l}_{\eta^+,\eta'}-G^{\l}_{\eta^+,\eta})\ud{\eta'}
+\int_{\eta}^{\eta^+}\frac{H(\eta',R\phi'(\phi,\eta,\eta'))}{\sin(\phi'(\phi,\eta,\eta'))}\exp(G^{\l}_{\eta,\eta'})\ud{\eta'}\bigg)\nonumber.
\end{eqnarray}
Note the fact
\begin{eqnarray}
\frac{\ud{}}{\ud{\eta'}}G_{\eta,\eta'}^{\l}(\phi)=-\frac{1+\l}{\sin(\phi'(\phi,\eta,\eta'))}.
\end{eqnarray}
Hence,
we can directly estimate as follows:\\
In Case I:
\begin{eqnarray}
\abs{f_{\l}(\eta,\phi)}&\leq&\exp(-G^{\l}_{\eta,0})\lnm{h}+\lnnm{H}\int_0^{\eta}
\frac{1}{\sin(\phi'(\phi,\eta,\eta'))}\exp(-G^{\l}_{\eta,\eta'})\ud{\eta'}\\
&=&\exp(-G^{\l}_{\eta,0})\lnm{h}+\frac{1}{1+\l}\lnnm{H}\exp(-G^{\l}_{\eta,\eta'})\bigg|_0^{\eta}\nonumber\\
&=&\exp(-G^{\l}_{\eta,0})\lnm{h}+\frac{1}{1+\l}\bigg(1-\exp(-G^{\l}_{\eta,0})\bigg)\lnnm{H}\leq
\lnm{h}+\lnnm{H}.\nonumber
\end{eqnarray}
In Case II:
\begin{eqnarray}
\\
\abs{f_{\l}(\eta,\phi)}&\leq&\exp(-G^{\l}_{L,0}-G^{\l}_{L,\eta})\lnm{h}\nonumber\\
&&+\frac{1}{1+\l}\bigg\{\exp(-G^{\l}_{L,\eta})\bigg(1-\exp(-G^{\l}_{L,0})\bigg)\lnnm{H}+\bigg(1-\exp(-G^{\l}_{L,\eta})\bigg)\lnnm{H}\bigg\}\nonumber\\
&\leq&\lnm{h}+\lnnm{H}\nonumber.
\end{eqnarray}
In Case III:
\begin{eqnarray}
\\
\abs{f_{\l}(\eta,\phi)}&\leq&\exp(-G^{\l}_{\eta^+,0}-G^{\l}_{\eta^+,\eta})\lnm{h}\nonumber\\
&&+\frac{1}{1+\l}\bigg\{\exp(-G^{\l}_{\eta^+,\eta})\bigg(1-\exp(-G^{\l}_{\eta^+,0})\bigg)\lnnm{H}
+\bigg(1-\exp(-G^{\l}_{\eta^+,\eta})\bigg)\lnnm{H}\bigg\}\nonumber\\
&\leq&\lnm{h}+\lnnm{H}\nonumber.
\end{eqnarray}
This completes the proof of (\ref{Milne t 01}).\\
\ \\
Step 2: Assume $\lnnm{S}<\infty$ and $\lnm{h}<\infty$, then the
solution $f_{\l}^{L}(\eta,\phi)$ to the penalized $\e$-Milne
equation
\begin{eqnarray}\label{Milne finite problem LT penalty}
\left\{
\begin{array}{rcl}\displaystyle
\l f_{\l}^{L}+\sin\phi\frac{\p f_{\l}^{L}}{\p\eta}+F(\eta)\cos\phi\frac{\p
f_{\l}^{L}}{\p\phi}+f_{\l}^{L}-\bar f_{\l}^{L}&=&S(\eta,\phi),\\
f_{\l}^{L}(0,\phi)&=&h(\phi)\ \ \text{for}\ \ \sin\phi<0,\\
f_{\l}^{L}(L,\phi)&=&f_{\l}^{L}(L,R\phi).
\end{array}
\right.
\end{eqnarray}
where $\bar f_{\l}^{L}$ is defined as (\ref{Milne average}), satisfies
\begin{eqnarray}\label{Milne t 02}
\lnnm{f_{\l}^{L}}\leq \frac{1+\l}{\l}\bigg(\lnm{h}+\lnnm{S}\bigg).
\end{eqnarray}
\emph{The proof of (\ref{Milne t 02}):} In order to construct the
solution of (\ref{Milne finite problem LT penalty}), we iteratively
define the sequence $\{f_{m}^{L}\}_{m=0}^{\infty}$ as $f_{0}^{L}=0$
and
\begin{eqnarray}
\left\{
\begin{array}{rcl}\displaystyle
\l f_{m}^{L}+\sin\phi\frac{\p
f_{m}^{L}}{\p\eta}+F(\eta)\cos\phi\frac{\p
f_{m}^{L}}{\p\phi}+f_{m}^{L}-\bar f_{m-1}^{L}&=&S(\eta,\phi),\\
f_{m}^{L}(0,\phi)&=&h(\phi)\ \ \text{for}\ \ \sin\phi<0,\\
f_{m}^{L}(L,\phi)&=&f_{m}^{L}(L,R\phi).
\end{array}
\right.
\end{eqnarray}
Based on the analysis in Step 1 with $H=S+\bar f_{m-1}^{L}$, we know $f_{m}^{L}$ is well-defined
and $\lnnm{f_{m}^{L}}<\infty$. We further define
$g_{m}^{L}=f_{m}^{L}-f_{m-1}^{L}$ for $m\geq1$. Then $g_{m}^{L}$ can
be rewritten along the characteristics as follows:
Case I:\\
For $\sin\phi>0$,
\begin{eqnarray}
g_{m+1}^{L}(\eta,\phi)&=&\int_0^{\eta}\frac{\bar
g_{m}^{L}(\eta')}{\sin(\phi'(\phi,\eta,\eta'))}\exp(-G_{\eta,\eta'})\ud{\eta'}.
\end{eqnarray}
Case II:\\
For $\sin\phi<0$ and $\abs{E(\eta,\phi)}\leq e^{-V(L)}$,
\begin{eqnarray}
\\
g_{m+1}^{L}(\eta,\phi)&=& \bigg(\int_0^{L}\frac{\bar
g_{m}^{L}(\eta')}{\sin(\phi'(\phi,\eta,\eta'))}
\exp(-G_{L,\eta'}-G_{L,\eta})\ud{\eta'}+\int_{\eta}^{L}\frac{\bar
g_{m}^{L}(\eta')}{\sin(\phi'(\phi,\eta,\eta'))}\exp(G_{\eta,\eta'})\ud{\eta'}\bigg)\nonumber.
\end{eqnarray}
Case III:\\
For $\sin\phi<0$ and $\abs{E(\eta,\phi)}\geq e^{-V(L)}$,
\begin{eqnarray}
\\
g_{m+1}^{L}(\eta,\phi)&=&\bigg(\int_0^{\eta^+}\frac{\bar
g_{m}^{L}(\eta')}{\sin(\phi'(\phi,\eta,\eta'))}
\exp(-G_{\eta^+,\eta'}-G_{\eta^+,\eta})\ud{\eta'}
+\int_{\eta}^{\eta^+}\frac{\bar
g_{m}^{L}(\eta')}{\sin(\phi'(\phi,\eta,\eta'))}\exp(G_{\eta,\eta'})\ud{\eta'}\bigg)\nonumber.
\end{eqnarray}
In all three cases, we can always obtain
\begin{eqnarray}
\lnnm{g_{m+1}^{L}}\leq \frac{1}{1+\l}\lnnm{g_{m}^{L}}.
\end{eqnarray}
For $\l>0$, this is a contraction sequence. Also, we have
\begin{eqnarray}
\lnnm{g_{1}^{L}}=\lnnm{f_{1}^{L}}\leq \lnm{h}+\lnnm{H}.
\end{eqnarray}
Hence, $f_{m}^{L}$ converges strongly in
$L^{\infty}([0,L]\times[-\pi,\pi))$ to $f_{\l}^{L}$ which satisfies
(\ref{Milne finite problem LT penalty}). Also, $f_{\l}^{L}$ satisfies
\begin{eqnarray}
\lnnm{f_{\l}^{L}}\leq \frac{1+\l}{\l}\bigg(\lnm{h}+\lnnm{S}\bigg).
\end{eqnarray}
Naturally, we obtain $f_{\l}^{L}\in L^2([0,L]\times[-\pi,\pi))$.

However, we can see the estimate blows up when $\l\rt0$. Therefore,
we need to show the uniform estimate of $f_{\l}^{L}$ with respect to
$\l$. Similarly, we can define $r_{\l}^{L}$ and $q_{\l}^{L}$ for $f_{\l}^{L}$ as in (\ref{hydro}).\\
\ \\
Step 3: $r_{\l}^{L}$ satisfies
\begin{eqnarray}\label{Milne temp 37}
\int_{0}^L\tnm{r_{\l}^{L}(\eta)}^2\ud{\eta}&\leq&
4\tnm{h}^2+8\int_{0}^L\tnm{S(\eta)}^2\ud{\eta}.
\end{eqnarray}
\emph{The proof of (\ref{Milne temp 37}):} Multiplying $f_{\l}^{L}$
on both sides of (\ref{Milne finite problem LT penalty}) and
integrating over $\phi\in[-\pi,\pi)$, we get the energy estimate
\begin{eqnarray}\label{Milne temp 31}
\\
\half\frac{\ud{}}{\ud{\eta}}\br{f_{\l}^{L}
,f_{\l}^{L}\sin\phi}_{\phi}(\eta)=-\l\tnm{f_{\l}^{L}(\eta)}^2-\tnm{r_{\l}^{L}(\eta)}^2-F(\eta)\br{\frac{\p
f_{\l}^{L}}{\p\phi},f_{\l}^{L}\cos\phi}_{\phi}(\eta)+\br{S,f_{\l}^{L}}_{\phi}(\eta).\nonumber
\end{eqnarray}
A further integration by parts reveals
\begin{eqnarray}
-F(\eta)\br{\frac{\p
f_{\l}^{L}}{\p\phi},f_{\l}^{L}\cos\phi}_{\phi}(\eta)&=&-\half
F(\eta)\br{f_{\l}^{L},f_{\l}^{L}\sin\phi}_{\phi}(\eta).
\end{eqnarray}
Also, the assumption $\bar S(\eta)=0$ leads to
\begin{eqnarray}
\br{S,f_{\l}^{L}}_{\phi}(\eta)&=&\br{S,q_{\l}^{L}}_{\phi}(\eta)+\br{S,r_{\l}^{L}}_{\phi}(\eta)=\br{S,r_{\l}^{L}}_{\phi}(\eta).
\end{eqnarray}
Hence, we have the simplified form of (\ref{Milne temp 31}) as
follows:
\begin{eqnarray}\label{Milne temp 32}
\\
\half\frac{\ud{}}{\ud{\eta}}\br{
f_{\l}^{L},f_{\l}^{L}\sin\phi}_{\phi}(\eta)=-\l\tnm{f_{\l}^{L}(\eta)}^2-\tnm{r_{\l}^{L}(\eta)}^2-\half
F(\eta)\br{
f_{\l}^{L},f_{\l}^{L}\sin\phi}_{\phi}(\eta)+\br{S,r_{\l}^{L}}_{\phi}(\eta).\nonumber
\end{eqnarray}
Define
\begin{eqnarray}
\alpha(\eta)=\half\br{f_{\l}^{L},f_{\l}^{L}\sin\phi }_{\phi}(\eta).
\end{eqnarray}
Then (\ref{Milne temp 32}) can be rewritten as follows:
\begin{eqnarray}
\frac{\ud{\alpha}}{\ud{\eta}}=-\l\tnm{f_{\l}^{L}(\eta)}^2-\tnm{r_{\l}^{L}(\eta)}^2-
F(\eta)\alpha(\eta)+\br{S,r_{\l}^{L}}_{\phi}(\eta).
\end{eqnarray}
We can integrate above on $[\eta,L]$ and $[0,\eta]$ respectively to obtain
\begin{eqnarray}
\label{Milne temp 4}
\alpha(\eta)&=&\alpha(L)\exp\bigg(\int_{\eta}^LF(y)\ud(y)\bigg)\\
&&+\int_{\eta}^L\exp\bigg(\int_{\eta}^yF(z)\ud{z}\bigg)
\bigg(\l\tnm{f_{\l}^{L}(y)}^2+\tnm{r_{\l}^{L}(y)}^2-\br{S,r_{\l}^{L}}_{\phi}(y)\bigg)\ud{y},\nonumber\\
\label{Milne temp 5}
\alpha(\eta)&=&\alpha(0)\exp\bigg(-\int_{0}^{\eta}F(y)\ud(y)\bigg)\\
&&+\int_{0}^{\eta}\exp\bigg(-\int_{y}^{\eta}F(z)\ud{z}\bigg)
\bigg(-\l\tnm{f_{\l}^{L}(y)}^2-\tnm{r_{\l}^{L}(y)}^2+\br{S,r_{\l}^{L}}_{\phi}(y)\bigg)\ud{y}.\nonumber
\end{eqnarray}
The specular reflexive boundary $f_{\l}^{L}(L,\phi)=f_{\l}^{L}(L,R\phi)$
ensures $\alpha(L)=0$. Hence, based on (\ref{Milne temp 4}), we have
\begin{eqnarray}\label{Milne temp 36}
\alpha(\eta)\geq\int_{\eta}^L\exp\bigg(\int_{\eta}^yF(z)\ud{z}\bigg)\bigg(-\br{S,r_{\l}^{L}}_{\phi}(y)\bigg)\ud{y}.
\end{eqnarray}
Also, (\ref{Milne temp 5}) implies
\begin{eqnarray}
\alpha(\eta)&\leq&\alpha(0)\exp[V(\eta)]+\int_{0}^{\eta}\exp\bigg(-\int_{y}^{\eta}F(z)\ud{z}\bigg)
\bigg(\br{S,r_{\l}^{L}}_{\phi}(y)\bigg)\ud{y}\\
&\leq&2\tnm{h}^2+\int_{0}^{\eta}\exp\bigg(-\int_{y}^{\eta}F(z)\ud{z}\bigg)
\bigg(\br{S,r_{\l}^{L}}_{\phi}(y)\bigg)\ud{y}\nonumber,
\end{eqnarray}
due to the fact
\begin{eqnarray}
\alpha(0)=\half\br{\sin\phi
f^{L}_{\l},f^{L}_{\l}}_{\phi}(0)\leq\half\bigg(\int_{\sin\phi>0}h(\phi)^2\sin\phi
\ud{\phi}\bigg)\leq \half\tnm{h}^2.
\end{eqnarray}
Then in (\ref{Milne temp 5}) taking $\eta=L$, from $\alpha(L)=0$, we have
\begin{eqnarray}\label{Milne temp 6}
\int_{0}^L\exp\bigg(\int_{0}^yF(z)\ud{z}\bigg)\tnm{r_{\l}^{L}(y)}^2\ud{y}
&\leq&\alpha(0)+\int_{0}^L\exp\bigg(\int_{0}^yF(z)\ud{z}\bigg)\br{S,r_{\l}^{L}}_{\phi}(y)\ud{y}\\
&\leq&
\half\tnm{h}^2+\int_{0}^L\exp\bigg(\int_{0}^yF(z)\ud{z}\bigg)\br{S,r_{\l}^{L}}_{\phi}(y)\ud{y}\nonumber.
\end{eqnarray}
On the other hand, by (\ref{force temp 2}), we can directly
estimate as follows:
\begin{eqnarray}\label{Milne temp 7}
\int_{0}^L\exp\bigg(\int_{0}^yF(z)\ud{z}\bigg)\tnm{r_{\l}^{L}(y)}^2\ud{y}\geq
\frac{1}{4}\int_{0}^L\tnm{r_{\l}^{L}(y)}^2\ud{y}.
\end{eqnarray}
Combining (\ref{Milne temp 6}) and (\ref{Milne temp 7}) yields
\begin{eqnarray}\label{Milne temp 33}
\int_{0}^L\tnm{r_{\l}^{L}(\eta)}^2\ud{\eta}&\leq&
2\tnm{h}^2+4\int_{0}^L\exp\bigg(\int_{0}^yF(z)\ud{z}\bigg)\br{S,r_{\l}^{L}}_{\phi}(y)\ud{y}.
\end{eqnarray}
By Cauchy's inequality and (\ref{force temp 2}), we have
\begin{eqnarray}\label{Milne temp 34}
\abs{\int_{0}^L\exp\bigg(\int_{0}^yF(z)\ud{z}\bigg)\br{S,r_{\l}^{L}}_{\phi}(y)\ud{y}}&\leq&\abs{\int_0^L\br{S,r_{\l}^{L}}_{\phi}(y)\ud{y}}\\
&\leq&
\frac{1}{8}\int_{0}^L\tnm{r_{\l}^{L}(\eta)}^2\ud{\eta}+2\int_{0}^L\tnm{S(\eta)}^2\ud{\eta}\nonumber.
\end{eqnarray}
Therefore, summarizing (\ref{Milne temp 33}) and (\ref{Milne temp
34}), we deduce (\ref{Milne temp 37}).\\
\ \\
Step 4: $q_{\l}^{L}$ satisfies
\begin{eqnarray}\label{Milne temp 38}
\\
\tnm{q^{L}_{\l}(\eta)}^2&\leq&
256\pi^2(1+\l)\bigg(1+\eta^{1/2}+\tnm{r_{\l}^{L}(\eta)}\bigg)\bigg(1+\tnm{h}^2+\int_{0}^L\tnm{S(\eta)}^2\ud{\eta}+\int_0^{\eta}
\lnm{S(y)}\ud{y}\bigg)\nonumber.
\end{eqnarray}
\emph{The proof of (\ref{Milne temp 38}):} Multiplying $\sin\phi$ on
both sides of (\ref{Milne finite problem LT penalty}) and
integrating over $\phi\in[-\pi,\pi)$ lead to
\begin{eqnarray}
\label{Milne temp 18}\\
\frac{\ud{}}{\ud{\eta}}\br{\sin^2\phi,f_{\l}^{L}}_{\phi}(\eta)=-\l\br{\sin\phi,
f_{\l}^{L}}_{\phi}(\eta)
-\br{\sin\phi,r_{\l}^{L}}_{\phi}(\eta)-F(\eta)\br{\sin\phi\cos\phi,\frac{\p
f_{\l}^{L}}{\p\phi}}_{\phi}(\eta)+\br{\sin\phi,S}_{\phi}(\eta)\nonumber.
\end{eqnarray}
We can further integrate by parts as follows:
\begin{eqnarray}
-F(\eta)\br{\sin\phi\cos\phi,\frac{\p
f_{\l}^{L}}{\p\phi}}_{\phi}(\eta)=F(\eta)\br{\cos(2\phi),f_{\l}^{L}}_{\phi}(\eta)=F(\eta)\br{\cos(2\phi),r_{\l}^{L}}_{\phi}(\eta),
\end{eqnarray}
to obtain
\begin{eqnarray}
\frac{\ud{}}{\ud{\eta}}\br{\sin^2\phi,f_{\l}^{L}}_{\phi}(\eta)=
-\br{\sin\phi,r_{\l}^{L}}_{\phi}(\eta)F(\eta)\br{\cos(2\phi),r_{\l}^{L}}_{\phi}(\eta)+\br{\sin\phi,S}_{\phi}(\eta).
\end{eqnarray}
Define
\begin{eqnarray}\label{Milne temp 81}
\beta_{\l}^{L}(\eta)=\br{\sin^2\phi,f_{\l}^{L}}_{\phi}(\eta).
\end{eqnarray}
Then we can simplify (\ref{Milne temp 18}) as follows:
\begin{eqnarray}\label{Milne temp 35}
\frac{\ud{\beta^{L}_{\l}}}{\ud{\eta}}=D^{L}_{\l}(\eta,\phi),
\end{eqnarray}
where
\begin{eqnarray}\label{Milne temp 82}
D^{L}_{\l}(\eta,\phi)=-\l\br{\sin\phi,
f_{\l}^{L}}_{\phi}(\eta)-\br{\sin\phi,r_{\l}^{L}}_{\phi}(\eta)+F(\eta)\br{\cos(2\phi),r_{\l}^{L}}_{\phi}(\eta)+\br{\sin\phi,S}_{\phi}(\eta).
\end{eqnarray}
Since
\begin{eqnarray}
-\l\br{\sin\phi, f_{\l}^{L}}_{\phi}(\eta)&=&-\l\br{\sin\phi,
r_{\l}^{L}}_{\phi}(\eta)-\l\br{\sin\phi,
q_{\l}^{L}}_{\phi}(\eta)=-\l\br{\sin\phi, r_{\l}^{L}}_{\phi}(\eta).
\end{eqnarray}
we can further get
\begin{eqnarray}
D^{L}_{\l}(\eta,\phi)=-\l\br{\sin\phi,
r_{\l}^{L}}_{\phi}(\eta)-\br{\sin\phi,r_{\l}^{L}}_{\phi}(\eta)+F(\eta)\br{\cos(2\phi),r_{\l}^{L}}_{\phi}(\eta)+\br{\sin\phi,S}_{\phi}(\eta).
\end{eqnarray}
We can integrate over $[0,\eta]$ in (\ref{Milne temp 35}) to obtain
\begin{eqnarray}\label{Milne t 03}
\beta^{L}_{\l}(\eta)=\beta^{L}_{\l}(0)+\int_0^{\eta}D^{L}_{\l}(y)\ud{y}.
\end{eqnarray}
It is important to note that $D^{L}_{\l}$ only depends on
$r_{\l}^{L}$ and is independent of $q_{\l}^{L}$. Then we can
directly estimate
\begin{eqnarray}\label{Milne t 04}
\\
\lnm{D^{L}_{\l}(\eta)}&\leq&
2\pi\bigg(1+\l+\abs{F(\eta)}\bigg)\tnm{r_{\l}^{L}(\eta)}+\lnm{S(\eta)}\leq4\pi(1+\l)\tnm{r_{\l}^{L}(\eta)}+\lnm{S(\eta)}.\nonumber
\end{eqnarray}
Also, for the initial data
\begin{eqnarray}
\\
\beta^{L}_{\l}(0)=\br{\sin^2\phi,f_{\l}^{L}}_{\phi}(0)\leq \bigg(\br{
f_{\l}^{L},f_{\l}^{L}\abs{\sin\phi}}_{\phi}(0)\bigg)^{1/2}\tnm{\sin\phi}^{3/2}\leq
8\bigg(\br{ f_{\l}^{L},f_{\l}^{L}\abs{\sin\phi}}_{\phi}(0)\bigg)^{1/2}.\nonumber
\end{eqnarray}
Obviously, we have
\begin{eqnarray}
\br{f_{\l}^{L}, f_{\l}^{L}\abs{\sin\phi}
}_{\phi}(0)=\int_{\sin\phi>0}h^2(\phi)\sin\phi
\ud{\phi}-\int_{\sin\phi<0}\bigg(f_{\l}^{L}(0,\phi)\bigg)^2\sin\phi\ud{\phi}.
\end{eqnarray}
However, based on the definition of $\alpha(\eta)$ and (\ref{Milne temp 36}), we can obtain
\begin{eqnarray}
\\
\int_{\sin\phi>0} h^2(\phi)\sin\phi\ud{\phi}+\int_{\sin\phi<0}
\bigg(f_{\l}^{L}(0,\phi)\bigg)^2\sin\phi\ud{\phi}&=&2\alpha(0)\nonumber\\
&\geq&
2\int_{0}^L\exp\bigg(\int_{0}^yF(z)\ud{z}\bigg)\bigg(-\br{S,r_{\l}^{L}}_{\phi}(y)\bigg)\ud{y}\nonumber\\
&\geq&-\half\int_0^L\br{S,r_{\l}^{L}}_{\phi}(y)\ud{y}\nonumber.
\end{eqnarray}
Hence, we can deduce
\begin{eqnarray}
-\int_{\sin\phi<0}
\bigg(f_{\l}^{L}(0,\phi)\bigg)^2\sin\phi\ud{\phi}&\leq&\int_{\sin\phi>0}
h^2(\phi)\sin\phi\ud{\phi}+\half\int_{0}^L\br{S,r_{\l}^{L}}_{\phi}(y)\ud{y}\\
&\leq&
\tnm{h}^2+\frac{1}{4}\int_{0}^L\tnm{r_{\l}^{L}(\eta)}^2\ud{\eta}+\frac{1}{4}\int_{0}^L\tnm{S(\eta)}^2\ud{\eta}\nonumber.
\end{eqnarray}
From (\ref{Milne temp 37}), we can deduce
\begin{eqnarray}\label{Milne t 05}
\beta^{L}_{\l}(0)^2&\leq& 64\br{f_{\l}^{L},f_{\l}^{L}\abs{\sin\phi}}_{\phi}(0)\leq
128\tnm{h}^2+16\int_{0}^L\tnm{r_{\l}^{L}(\eta)}^2\ud{\eta}+16\int_{0}^L\tnm{S(\eta)}^2\ud{\eta}\\
&\leq&192\tnm{h}^2+192\int_{0}^L\tnm{S(\eta)}^2\ud{\eta}\nonumber.
\end{eqnarray}
From (\ref{Milne temp 37}), (\ref{Milne t 03}), (\ref{Milne t 04}) and (\ref{Milne t 05}), we have
\begin{eqnarray}
&&\lnm{\beta^{L}_{\l}(\eta)}\\
&\leq&8+
192\tnm{h}^2+192\int_{0}^L\tnm{S(\eta)}^2\ud{\eta}+
2\pi(1+\l)\int_0^{\eta}\tnm{r_{\l}^{L}(y)}\ud{y}+\int_0^{\eta}\lnm{S(y)}\ud{y}\nonumber\\
&\leq&8+
192\tnm{h}^2+192\int_{0}^L\tnm{S(\eta)}^2\ud{\eta}+\int_0^{\eta}\lnm{S(y)}\ud{y}+
2\pi(1+\l)\eta^{1/2}\bigg(\int_0^{\eta}\tnm{r^{L}_{\l}(y)}^2\ud{y}\bigg)^{1/2}\nonumber\\
&\leq
&64\pi(1+\l)(1+\eta^{1/2})\bigg(\tnm{h}^2+\int_{0}^L\tnm{S(\eta)}^2\ud{\eta}+\int_0^{\eta}\lnm{S(y)}\ud{y}\bigg)\nonumber.
\end{eqnarray}
By (\ref{Milne temp 81}) this implies
\begin{eqnarray}
\tnm{q_{\l}^{L}(\eta)}^2&\leq&
2\tnm{\beta^{L}_{\l}(\eta)}+2\tnm{r_{\l}^{L}(\eta)}\leq
4\pi\lnm{\beta^{L}_{\l}(\eta)}+2\tnm{r_{\l}^{L}(\eta)}
\end{eqnarray}
which completes the proof of (\ref{Milne temp 38}).\\
\ \\
Step 5: Passing to the limit.\\
Since estimates (\ref{Milne temp 37}) and (\ref{Milne temp 38}) are
uniform in $\l$, we can take weakly convergent subsequence
$f_{\l}^{L}\rt f^{L}\in L^2([0,L]\times[-\pi,\pi))$ as $\l\rt0$.
Hence, $f^{L}$ is the solution of (\ref{Milne finite problem LT})
and satisfies the estimates (\ref{Milne temp 1}) and
(\ref{Milne temp 2}).\\
\ \\
Step 6: Orthogonality relation (\ref{Milne temp 3}).\\
A direct integration over $\phi\in[-\pi,\pi)$ in
(\ref{Milne finite problem LT}) implies
\begin{eqnarray}
\frac{\ud{}}{\ud{\eta}}\br{\sin\phi,f^{L}}_{\phi}(\eta)=-F\br{\cos\phi,\frac{\ud{f^{L}}}{\ud{\phi}}}_{\phi}(\eta)
+\bar S(\eta)=-F\br{\sin\phi,f^{L}}_{\phi}(\eta).
\end{eqnarray}
thanks to $\bar S=0$. The specular reflexive boundary
$f^{L}(L,\phi)=f^{L}(L,R\phi)$ implies
$\br{\sin\phi,f^{L}}_{\phi}(L)=0$. Then we have
\begin{eqnarray}
\br{\sin\phi,f^{L}}_{\phi}(\eta)=0.
\end{eqnarray}
It is easy to see
\begin{eqnarray}
\br{\sin\phi,q^{L}}_{\phi}(\eta)=0.
\end{eqnarray}
Hence, we may derive
\begin{eqnarray}
\br{\sin\phi,r^{L}}_{\phi}(\eta)=0.
\end{eqnarray}
This leads (\ref{Milne temp 3}) and completes the proof of
(\ref{Milne finite LT}).
\end{proof}

\subsubsection{Infinite Slab with $\bar S=0$}

We turn to the $\e$-Milne problem for $f(\eta,\phi)$ in the infinite slab $(\eta,\phi)\in[0,\infty)\times[-\pi,\pi)$
\begin{eqnarray}\label{Milne infinite problem LT}
\left\{
\begin{array}{rcl}\displaystyle
\sin\phi\frac{\p f}{\p\eta}+F(\eta)\cos\phi\frac{\p
f}{\p\phi}+f-\bar f&=&S(\eta,\phi),\\
f(0,\phi)&=&h(\phi)\ \ \text{for}\ \ \sin\phi>0,\\
\lim_{\eta\rt\infty}f(\eta,\phi)&=&f_{\infty}.
\end{array}
\right.
\end{eqnarray}
Define $r$ and $q$ for $f$ as $r^{L}$ and $q^{L}$ for $f^{L}$.
\begin{lemma}\label{Milne infinite LT}
Assume $\bar S(\eta)=0$ for any $\eta\in[0,\infty)$ with (\ref{Milne bounded}) and (\ref{Milne decay}). Then there
exists a solution $f(\eta,\phi)$ of the infinite slab problem
(\ref{Milne infinite problem LT}), satisfying
\begin{eqnarray}
\tnnm{r}&\leq&C\bigg(M+\frac{M}{K}\bigg)^2<\infty\label{Milne temp 8},\\
\br{\sin\phi,r}_{\phi}(\eta)&=&0\label{Milne temp 19},
\end{eqnarray}
for any $\eta\in[0,\infty)$. Also there exists a constant
$q_{\infty}=f_{\infty}\in\r$ such that the following estimates hold,
\begin{eqnarray}
\abs{q_{\infty}}&\leq&C\bigg(1+M+\frac{M}{K}\bigg)^2<\infty\label{Milne temp 17},\\
\tnm{q(\eta)-q_{\infty}}&\leq&C\bigg(\tnm{r(\eta)}+\int_{\eta}^{\infty}\abs{F(y)}\tnm{r(y)}\ud{y}+\int_{\eta}^{\infty}\lnm{S(y)}\ud{y}
\bigg)\label{Milne temp 9},\\
\tnnm{q-q_{\infty}}&\leq&C\bigg(M+\frac{M}{K}\bigg)^2<\infty\label{Milne temp 10}.
\end{eqnarray}
The
solution is unique among functions such that (\ref{Milne temp 8}),
(\ref{Milne temp 17})and (\ref{Milne temp 10}) hold.
\end{lemma}
\begin{proof}
\ \\
Step 1: Existence and estimates (\ref{Milne temp 8}) and (\ref{Milne temp 19}).\\
By the uniform estimates from Lemma \ref{Milne infinite LT}, the
solution $f^{L}$ of the finite problem (\ref{Milne finite problem
LT}) in the slab $[0,L]$ is uniformly bounded in
$L^2_{loc}([0,\infty);L^2[-\pi,\pi))$. Then there exists a
subsequence such that
\begin{eqnarray}
q^{L}&\rightharpoonup&q,\\
r^{L}&\rightharpoonup&r,
\end{eqnarray}
weakly in $L^2_{loc}([0,\infty);L^2[-\pi,\pi))$. Also, $f=q+r$
satisfies the boundary condition at $\eta=0$. This shows the
existence of the solution. Then property (\ref{Milne temp 8})
naturally holds due to the weak lower semi-continuity of norm
$\tnnm{\cdot}$.
Also, the orthogonal relation (\ref{Milne temp 19}) is preserved. \\
\ \\
Step 2: Estimates (\ref{Milne temp 17}), (\ref{Milne temp 9}) and (\ref{Milne temp 10}).\\
We continue using the notation in Step 5 of the proof of Lemma
\ref{Milne finite LT}. Recall (\ref{Milne temp 81}) to (\ref{Milne temp 82}) with $\l=0$ and $L=\infty$. We have
\begin{eqnarray}
\beta(\eta)=\br{\sin^2\phi,f}_{\phi}(\eta)
\end{eqnarray}
and
\begin{eqnarray}\label{Milne t 21}
\frac{\ud{\beta}}{\ud{\eta}}=D(\eta,\phi),
\end{eqnarray}
where
\begin{eqnarray}
D(\eta,\phi)=-\br{\sin\phi,r}_{\phi}+F(\eta)\br{\cos(2\phi),r}_{\phi}+\br{\sin\phi,S}_{\phi}(\eta).
\end{eqnarray}
The orthogonal relation (\ref{Milne temp 19}) implies
\begin{eqnarray}
D(\eta,\phi)=F(\eta)\br{\cos(2\phi),r}_{\phi}+\br{\sin\phi,S}_{\phi}(\eta).
\end{eqnarray}
Hence, we can integrate (\ref{Milne t 21}) over $[0,\eta]$ to show
\begin{eqnarray}
\beta(\eta)-\beta(0)=\int_0^{\eta}F(y)\br{\cos(2\phi),r}_{\phi}(y)\ud{y}+\int_0^{\eta}\br{\sin\phi,S}_{\phi}(y)\ud{y}.
\end{eqnarray}
Based on Lemma \ref{Milne force}, since $F\in L^1[0,\infty)\cap
L^2[0,\infty)$, $r\in L^2([0,\infty)\times[-\pi,\pi))$, and $S$
exponentially decays, by (\ref{Milne t 05}) and (\ref{Milne temp
8}), there exists some constant $\beta_{\infty}$ such that
$\beta_{\infty}=\lim_{\eta\rt\infty}\beta(\eta)$ satisfying
\begin{eqnarray}
\abs{\beta_{\infty}}&\leq&\abs{\beta(0)}+\abs{\int_0^{\infty}F(y)\br{\cos(2\phi),r}_{\phi}(y)\ud{y}}
+\abs{\int_0^{\infty}\br{\sin\phi,S}_{\phi}(y)\ud{y}}\\
&\leq&8+192\tnm{h}^2+200\int_{0}^{\infty}\tnm{S(\eta)}^2\ud{\eta}+2\pi\tnnm{F}
\tnnm{r}\leq C\bigg(1+M+\frac{M}{K}\bigg)^2\nonumber.
\end{eqnarray}
We define $q_{\infty}=\beta_{\infty}/\tnm{\sin\phi}^2$. Hence, the
estimate of $\abs{q_{\infty}}$ in (\ref{Milne temp 17}) is valid.
Moreover,
\begin{eqnarray}
\beta_{\infty}-\beta(\eta)=\int_{\eta}^{\infty}D(y)\ud{y}=\int_{\eta}^{\infty}F(y)\br{\cos(2\phi),r}_{\phi}(y)\ud{y}+
\int_{\eta}^{\infty}\br{\sin\phi,S}_{\phi}(y)\ud{y}.
\end{eqnarray}
Note
\begin{eqnarray}
\beta(\eta)=\br{\sin^2\phi,f}_{\phi}(\eta)=\br{\sin^2\phi,q}_{\phi}(\eta)+\br{\sin^2\phi,r}_{\phi}(\eta)
=q(\eta)\tnm{\sin\phi}^2+\br{\sin^2\phi,r}_{\phi}(\eta).
\end{eqnarray}
Thus we can estimate
\begin{eqnarray}\label{Milne t 22}
&&\tnm{\sin\phi}^2\tnm{q(\eta)-q_{\infty}}\\
&=&\sqrt{2\pi}\tnm{\sin\phi}^2\lnm{q(\eta)-q_{\infty}}
=\sqrt{2\pi}\bigg(\beta(\eta)-\br{\sin^2\phi,r}_{\phi}(\eta)-\beta_{\infty}\bigg)\nonumber\\
&\leq&
\sqrt{2\pi}\bigg(\abs{\br{\sin^2\phi,r}_{\phi}(\eta)}+\int_{\eta}^{\infty}\abs{F(y)\br{\cos(2\phi),r(y)}_{\phi}\ud{y}}\ud{\eta}
+\int_{\eta}^{\infty}\abs{\br{\sin\phi,S}_{\phi}(y)}\ud{y}\bigg)\nonumber\\
&\leq&2\pi^2\bigg(\tnm{r(\eta)}+\int_{\eta}^{\infty}\abs{F(y)}\tnm{r(y)}\ud{y}+\int_{\eta}^{\infty}\lnm{S(y)}\ud{y}\bigg)\nonumber.
\end{eqnarray}
This implies (\ref{Milne temp 9}). Furthermore, we
integrate (\ref{Milne t 22}) over $\eta\in[0,\infty)$.
Cauchy's inequality and (\ref{force temp 3}) imply
\begin{eqnarray}
&&\int_0^{\infty}\bigg(\int_{\eta}^{\infty}\abs{F(y)}\tnm{r(y)}\ud{y}\bigg)^2\ud{\eta}\leq
\tnnm{r}^2\int_0^{\infty}\int_{\eta}^{\infty}\abs{F(y)}^2\ud{y}\ud{\eta}\leq
C.
\end{eqnarray}
The exponential decays shows
\begin{eqnarray}
\int_0^{\infty}\bigg(\int_{\eta}^{\infty}\lnm{S(y)}\ud{y}\bigg)^2\ud{\eta}\leq
C.
\end{eqnarray}
Hence, the estimate of $\tnnm{q-q_{\infty}}$ in (\ref{Milne temp
10}) naturally follows.\\
\ \\
Step 3: Uniqueness\\
In order to show the uniqueness of the solution, we assume there are
two solutions $f_1$ and $f_2$ to the equation (\ref{Milne infinite
problem LT}) satisfying (\ref{Milne temp 8}) and (\ref{Milne temp
19}). Then $f'=f_1-f_2$ satisfies the equation
\begin{eqnarray}\label{uniqueness equation}
\left\{
\begin{array}{rcl}\displaystyle
\sin\phi\frac{\p f'}{\p\eta}+F(\eta)\cos\phi\frac{\p
f'}{\p\phi}+f'-\bar f'&=&0,\\
f'(0,\phi)&=&0\ \ \text{for}\ \ \sin\phi>0,\\
\lim_{\eta\rt\infty}f'(\eta,\phi)&=&f'_{\infty}.
\end{array}
\right.
\end{eqnarray}
Similarly, we can define $r'$ and $q'$. Multiplying $e^{-V(\eta)}f'$
on both sides of (\ref{uniqueness equation}) and integrating over
$\phi\in[-\pi,\pi)$ yields
\begin{eqnarray}\label{uniqueness temp}
\half\frac{\ud{}}{\ud{\eta}}\bigg(\br{
f',f'\sin\phi}_{\phi}(\eta)e^{-V(\eta)}\bigg)=-\bigg(\tnm{r'(\eta)}^2e^{-V(\eta)}\bigg)\leq0.
\end{eqnarray}
This is due to the fact
\begin{eqnarray}
&&\half\frac{\ud{}}{\ud{\eta}}\bigg(\br{
f',f'\sin\phi}_{\phi}(\eta)e^{-V(\eta)}\bigg)\\
&=&\bigg(\br{
f',\frac{\ud{f'}}{\ud{\eta}}\sin\phi}_{\phi}(\eta)e^{-V(\eta)}\bigg)+\half\bigg(F(\eta)\br{
f',f'\sin\phi}_{\phi}(\eta)e^{-V(\eta)}\bigg)\nonumber\\
&=&\bigg(\br{
f',\frac{\ud{f'}}{\ud{\eta}}\sin\phi}_{\phi}(\eta)e^{-V(\eta)}\bigg)+\bigg(F(\eta)\br{
f',\frac{\ud{f'}}{\ud{\phi}}\cos\phi}_{\phi}(\eta)e^{-V(\eta)}\bigg)\nonumber.
\end{eqnarray}
Thus, we have
\begin{eqnarray}\label{Milne t 06}
\gamma(\eta)=\half\br{f',f'\sin\phi}_{\phi}(\eta)e^{-V(\eta)}.
\end{eqnarray}
is decreasing. Since $r'\in L^2([0,\infty)\times[-\pi,\pi))$ and
$q'-q'_{\infty}\in L^2([0,\infty)\times[-\pi,\pi))$, there exists a
convergent subsequence $\eta_k\rt\infty$ satisfying
$\tnm{r'(\eta_k)}\rt0$ and $q'(\eta_k)-q'_{\infty}\rt0$. Hence, this
implies
\begin{eqnarray}
\half\br{r',r'\sin\phi}_{\phi}(\eta_k)e^{-V(\eta_k)}\rt0.
\end{eqnarray}
Also, due to the fact that $q'(\eta_k)$ is independent of $\phi$ and
it is finite, we have
\begin{eqnarray}
\gamma(\eta_k)\rt0.
\end{eqnarray}
By the monotonicity, $\gamma(\eta)$ decreases to
zero and $\gamma(\eta)\geq0$. Then we can integrate (\ref{uniqueness
temp}) over $\eta\in[0,\infty)$ to obtain
\begin{eqnarray}
\gamma(\infty)-\gamma(0)=-2\int_0^{\infty}\tnm{r'(y)}^2e^{-V(y)}\ud{y},
\end{eqnarray}
which implies
\begin{eqnarray}
\gamma(0)=\br{f',f'\sin\phi}_{\phi}(0)e^{-V(0)}=2\int_0^{\infty}\tnm{r'(y)}^2e^{-V(y)}\ud{y}.
\end{eqnarray}
Also, we know
\begin{eqnarray}
0\leq\half\br{f',f'\sin\phi}_{\phi}(0)e^{-V(0)}&=&\half\br{
f',f'\sin\phi}_{\phi}(0)\leq\int_{\sin\phi>0}
(f')^2(\phi)\sin\phi\ud{\phi}=0.
\end{eqnarray}
Naturally, we have
\begin{eqnarray}
\br{f',f'\sin\phi}_{\phi}(0)e^{-V(0)}=2\int_0^{\infty}\tnm{r'(y)}^2e^{-V(y)}\ud{y}=0.
\end{eqnarray}
Hence, we have $r'=0$ and $f'(\eta,\phi)=q'(\eta)$. Plugging this
into the equation (\ref{uniqueness equation}) reveals $\px q'=0$.
Therefore, $f'(\eta,\phi)=C$ for some constant $C$. Naturally the
boundary data leads to $C=0$. In conclusion, we must have $f'=0$,
which means $f_1=f_2$, and the uniqueness follows.
\end{proof}

\subsubsection{$\bar S\neq0$ Case}

Consider the $\e$-Milne problem for $f(\eta,\phi)$ in $(\eta,\phi)\in[0,\infty)\times[-\pi,\pi)$ with a general source term
\begin{eqnarray}\label{Milne remark problem}
\left\{
\begin{array}{rcl}\displaystyle
\sin\phi\frac{\p f}{\p\eta}+F(\eta)\cos\phi\frac{\p
f}{\p\phi}+f-\bar f&=&S(\eta,\phi),\\
f(0,\phi)&=&h(\phi)\ \ \text{for}\ \ \sin\phi>0,\\
\lim_{\eta\rt\infty}f(\eta,\phi)&=&f_{\infty}.
\end{array}
\right.
\end{eqnarray}
\begin{lemma}\label{Milne infinite LT general}
Assume (\ref{Milne bounded}) and (\ref{Milne decay}) hold. Then there exists a solution $f(\eta,\phi)$ of the problem
(\ref{Milne remark problem}), satisfying
\begin{eqnarray}
\tnnm{r}&<&C\bigg(1+M+\frac{M}{K}\bigg)^2\leq\infty\label{Milne temp 39},\\
\br{\sin\phi,r}_{\phi}(\eta)&=&-\int_{\eta}^{\infty}e^{V(\eta)-V(y)}\bar
S(y)\ud{y}\label{Milne temp 40}.
\end{eqnarray}
Also there exists a constant $q_{\infty}=f_{\infty}\in\r$ such that the
following estimates hold,
\begin{eqnarray}
\abs{q_{\infty}}&\leq&C\bigg(1+M+\frac{M}{K}\bigg)^2<\infty\label{Milne temp 41},\\
\tnm{q(\eta)-q_{\infty}}&\leq&C\bigg(\tnm{r(\eta)}+\int_{\eta}^{\infty}\abs{F(y)}\tnm{r(y)}\ud{y}+\int_{\eta}^{\infty}\lnm{S(y)}\ud{y}
\bigg)\label{Milne temp 42},\\
\tnnm{q-q_{\infty}}&\leq&C\bigg(1+M+\frac{M}{K}\bigg)^2<\infty\label{Milne temp 43}.
\end{eqnarray}
The
solution is unique among functions satisfying
$\tnnm{f-f_{\infty}}<\infty$.
\end{lemma}
\begin{proof}
We can apply superposition property for this linear problem, i.e.
write $S=\bar S+(S-\bar S)=S_Q+S_R$. Then we solve the problem by
the following steps. For simplicity, we just call the estimates
(\ref{Milne temp 39}),
(\ref{Milne temp 41}), (\ref{Milne temp 42}) and (\ref{Milne temp 43}) as the $L^2$ estimates. \\
\ \\
Step 1: Construction of auxiliary function $f^1$.\\
We first solve $f^1$ as the solution to
\begin{eqnarray}
\left\{
\begin{array}{rcl}\displaystyle
\sin\phi\frac{\p f^1}{\p\eta}+F(\eta)\cos\phi\frac{\p
f^1}{\p\phi}+f^1-\bar f^1&=&S_R(\eta,\phi),\\
f^1(0,\phi)&=&h(\phi)\ \ \text{for}\ \ \sin\phi>0,\\
\lim_{\eta\rt\infty}f^1(\eta,\phi)&=&f_{\infty}^1.
\end{array}
\right.
\end{eqnarray}
Since $\bar S_R=0$, by Lemma \ref{Milne infinite LT}, we know there
exists a unique solution $f^1$
satisfying the $L^2$ estimate.\\
\ \\
Step 2: Construction of auxiliary function $f^2$.\\
We seek a function $f^{2}$ satisfying
\begin{eqnarray}\label{Milne temp 83}
-\frac{1}{2\pi}\int_{-\pi}^{\pi}\bigg(\sin\phi\frac{\p
f^{2}}{\p\eta}+F(\eta)\cos\phi\frac{\p
f^{2}}{\p\phi}\bigg)\ud{\phi}+S_Q=0.
\end{eqnarray}
The following analysis shows this type of function can always be
found. An integration by parts transforms the equation (\ref{Milne temp 83}) into
\begin{eqnarray}\label{Milne t 07}
-\int_{-\pi}^{\pi}\sin\phi\frac{\p
f^{2}}{\p\eta}\ud{\phi}-\int_{-\pi}^{\pi}F(\eta)\sin\phi
f^{2}\ud{\phi}+2\pi S_Q=0.
\end{eqnarray}
Setting
\begin{eqnarray}
f^{2}(\phi,\eta)=a(\eta)\sin\phi.
\end{eqnarray}
and plugging this ansatz into (\ref{Milne t 07}), we have
\begin{eqnarray}
-\frac{\ud{a}}{\ud{\eta}}\int_{-\pi}^{\pi}\sin^2\phi\ud{\phi}-F(\eta)a(\eta)\int_{-\pi}^{\pi}\sin^2\phi\ud{\phi}+2\pi
S_Q=0.
\end{eqnarray}
Hence, we have
\begin{eqnarray}
-\frac{\ud{a}}{\ud{\eta}}-F(\eta)a(\eta)+2S_Q=0.
\end{eqnarray}
This is a first order linear ordinary differential equation, which
possesses infinite solutions. We can directly solve it to obtain
\begin{eqnarray}
a(\eta)=e^{-\int_0^{\eta}F(y)\ud{y}}\bigg(a(0)+\int_0^{\eta}e^{\int_0^yF(z)\ud{z}}2S_Q(y)\ud{y}\bigg).
\end{eqnarray}
We may take
\begin{eqnarray}
a(0)=-\int_0^{\infty}e^{\int_0^yF(z)\ud{z}}2S_Q(y)\ud{y}.
\end{eqnarray}
Based on the exponential decay of $S_Q$, we can directly verify
$a(\eta)$ decays exponentially to zero as $\eta\rt\infty$ and $f^2$
satisfies the $L^2$ estimate.\\
\ \\
Step 3: Construction of auxiliary function $f^3$.\\
Based on above construction, we can directly verify
\begin{eqnarray}\label{Milne temp 84}
\int_{-\pi}^{\pi}\bigg(-\sin\phi\frac{\p
f^{2}}{\p\eta}-F(\eta)\cos\phi\frac{\p f^{2}}{\p\phi}-f^{2}+\bar
f^{2}+S_Q\bigg)\ud{\phi}=0.
\end{eqnarray}
Then we can solve $f^3$ as the solution to
\begin{eqnarray}
\\
\left\{
\begin{array}{rcl}\displaystyle
\sin\phi\frac{\p f^3}{\p\eta}+F(\eta)\cos\phi\frac{\p
f^3}{\p\phi}+f^3-\bar f^3&=&-\sin\phi\dfrac{\p
f^{2}}{\p\eta}-F(\eta)\cos\phi\dfrac{\p
f^{2}}{\p\phi}-f^{2}+\bar f^{2}+S_Q,\\
f^3(0,\phi)&=&-a(0)\sin\phi\ \ \text{for}\ \ \sin\phi>0,\\
\lim_{\eta\rt\infty}f^3(\eta,\phi)&=&f_{\infty}^3\nonumber.
\end{array}
\right.
\end{eqnarray}
By (\ref{Milne temp 84}), we can apply Lemma \ref{Milne infinite LT}
to obtain a unique solution $f^3$
satisfying the $L^2$ estimate.\\
\ \\
Step 4: Construction of auxiliary function $f^4$.\\
We now define $f^4=f^2+f^3$ and an explicit verification shows
\begin{eqnarray}
\left\{
\begin{array}{rcl}\displaystyle
\sin\phi\frac{\p f^4}{\p\eta}+F(\eta)\cos\phi\frac{\p
f^4}{\p\phi}+f^4-\bar f^4&=&S_Q(\eta,\phi),\\
f^4(0,\phi)&=&0\ \ \text{for}\ \ \sin\phi>0,\\
\lim_{\eta\rt\infty}f^4(\eta,\phi)&=&f_{\infty}^4,
\end{array}
\right.
\end{eqnarray}
and $f^4$
satisfies the $L^2$ estimate.\\
\ \\
In summary, we deduce that $f^1+f^4$ is the solution of (\ref{Milne
remark problem}) and satisfies the $L^2$ estimate. A direct
computation of $\br{\sin\phi,f^i}_{\phi}(\eta)$ for $i=1,2,3,4$
leads to (\ref{Milne temp 40}). From $\tnnm{f-f_{\infty}}<\infty$ ,
we deduce $\tnnm{\bar f-f_{\infty}}<\infty$, a similar argument as
in Lemma \ref{Milne infinite LT} shows the uniqueness of solution.
\end{proof}
Combining all above, we have the following theorem.
\begin{theorem}\label{Milne lemma 6}
For the $\e$-Milne problem (\ref{Milne problem}), there exists a unique
solution $f(\eta,\phi)$ satisfying the estimates
\begin{eqnarray}
\tnnm{f-f_{\infty}}\leq C\bigg(1+M+\frac{M}{K}\bigg)<\infty,
\end{eqnarray}
for some real number $f_{\infty}$ satisfying
\begin{eqnarray}
\abs{f_{\infty}}\leq C\bigg(1+M+\frac{M}{K}\bigg)^2<\infty.
\end{eqnarray}
\end{theorem}

\subsection{$L^{\infty}$ Estimates}

\subsubsection{Finite Slab}

Consider the $\e$-transport problem for $f(\eta,\phi)$ in a finite slab $(\eta,\phi)\in[0,L]\times[-\pi,\pi)$
\begin{eqnarray}\label{Milne finite problem LI}
\left\{
\begin{array}{rcl}\displaystyle
\sin\phi\frac{\p f}{\p\eta}+F(\eta)\cos\phi\frac{\p
f}{\p\phi}+f&=&H(\eta,\phi),\\
f(0,\phi)&=&h(\phi)\ \ \text{for}\ \ \sin\phi>0,\\
f(L,\phi)&=&f(L,R\phi).
\end{array}
\right.
\end{eqnarray}
Define the energy as follows:
\begin{eqnarray}
E(\eta,\phi)=\cos\phi e^{-V(\eta)}.
\end{eqnarray}
In the plane $(\eta,\phi)\in[0,\infty)\times[-\pi,\pi)$, on the curve
$\phi=\phi(\eta)$ with constant energy, we can see
\begin{eqnarray}
\frac{\ud{E}}{\ud{\eta}}=\frac{\p E}{\p\eta}+\frac{\p
E}{\p\phi}\frac{\p\phi}{\p\eta}=\cos\phi
F(\eta)e^{-V(\eta)}-\sin\phi e^{-V(\eta)}\frac{\p\phi}{\p\eta}=0,
\end{eqnarray}
which further implies
\begin{eqnarray}
\frac{\p\phi}{\p\eta}=\frac{\cos\phi F(\eta)}{\sin\phi}.
\end{eqnarray}
Plugging this into the equation (\ref{Milne finite problem LI}), on
this curve, we deduce
\begin{eqnarray}
\frac{\ud{f}}{\ud{\eta}}=\frac{\p f}{\p\eta}+\frac{\p
f}{\p\phi}\frac{\p\phi}{\p\eta}=\frac{1}{\sin\phi}\bigg(\sin\phi\frac{\p
f}{\p\eta}+\cos\phi F(\eta)\frac{\p f}{\p\phi}\bigg).
\end{eqnarray}
Hence, this curve with constant energy is exactly the
characteristics of the equation (\ref{Milne finite problem LI}).
Also, on this curve the equation can be simplified as follows:
\begin{eqnarray}
\sin\phi\frac{\ud{f}}{\ud{\eta}}+f=H.
\end{eqnarray}
An implicit function
$\eta^+(\eta,\phi)$ can be determined through
\begin{eqnarray}
\abs{E(\eta,\phi)}=e^{-V(\eta^+)}.
\end{eqnarray}
which means $(\eta^+,\phi_0)$ with $\sin\phi_0=0$ is on the same characteristics as $(\eta,\phi)$.
Define the quantities for $0\leq\eta'\leq\eta^+$ as follows:
\begin{eqnarray}
\phi'(\phi,\eta,\eta')&=&\cos^{-1}(\cos\phi e^{V(\eta')-V(\eta)}),\\
R\phi'(\phi,\eta,\eta')&=&-\cos^{-1}(\cos\phi e^{V(\eta')-V(\eta)})=-\phi'(\phi,\eta,\eta'),
\end{eqnarray}
where the inverse trigonometric function can be defined
single-valued in the domain $[0,\pi)$ and the quantities are always well-defined due to the monotonicity of $V$.
Finally we put
\begin{eqnarray}
G_{\eta,\eta'}(\phi)&=&\int_{\eta'}^{\eta}\frac{1}{\sin(\phi'(\phi,\eta,\xi))}\ud{\xi}.
\end{eqnarray}
We can rewrite the solution to the equation (\ref{Milne finite problem LI}) along
the characteristics
as follows:\\
\ \\
Case I:\\
For $\sin\phi>0$,
\begin{eqnarray}\label{Milne t 08}
f(\eta,\phi)&=&h(\phi'(\phi,\eta,0))\exp(-G_{\eta,0})
+\int_0^{\eta}\frac{H(\eta',\phi'(\phi,\eta,\eta'))}{\sin(\phi'(\phi,\eta,\eta'))}\exp(-G_{\eta,\eta'})\ud{\eta'}.
\end{eqnarray}
\ \\
Case II:\\
For $\sin\phi<0$ and $\abs{E(\eta,\phi)}\leq e^{-V(L)}$,
\begin{eqnarray}\label{Milne t 09}
\\
f(\eta,\phi)&=&h(\phi'(\phi,\eta,0))\exp(-G_{L,0}-G_{L,\eta})\nonumber\\
&+&\bigg(\int_0^{L}\frac{H(\eta',\phi'(\phi,\eta,\eta'))}{\sin(\phi'(\phi,\eta,\eta'))}
\exp(-G_{L,\eta'}-G_{L,\eta})\ud{\eta'}
+\int_{\eta}^{L}\frac{H(\eta',R\phi'(\phi,\eta,\eta'))}{\sin(\phi'(\phi,\eta,\eta'))}\exp(G_{\eta,\eta'})\ud{\eta'}\bigg)\nonumber.
\end{eqnarray}
\ \\
Case III:\\
For $\sin\phi<0$ and $\abs{E(\eta,\phi)}\geq e^{-V(L)}$,
\begin{eqnarray}\label{Milne t 10}
\\
f(\eta,\phi)&=&h(\phi'(\phi,\eta,0))\exp(-G_{\eta^+,0}-G_{\eta^+,\eta})\nonumber\\
&+&\bigg(\int_0^{\eta^+}\frac{H(\eta',\phi'(\phi,\eta,\eta'))}{\sin(\phi'(\phi,\eta,\eta'))}
\exp(-G_{\eta^+,\eta'}-G_{\eta^+,\eta})\ud{\eta'}
+\int_{\eta}^{\eta^+}\frac{H(\eta',R\phi'(\phi,\eta,\eta'))}{\sin(\phi'(\phi,\eta,\eta'))}\exp(G_{\eta,\eta'})\ud{\eta'}\bigg)\nonumber.
\end{eqnarray}

\subsubsection{Infinite Slab}

Consider the $\e$-transport problem for $f(\eta,\phi)$ in an infinite slab $(\eta,\phi)\in[0,\infty)\times[-\pi,\pi)$
\begin{eqnarray}\label{Milne infinite problem LI}
\left\{
\begin{array}{rcl}\displaystyle
\sin\phi\frac{\p f}{\p\eta}+F(\eta)\cos\phi\frac{\p
f}{\p\phi}+f&=&H(\eta,\phi),\\
f(0,\phi)&=&h(\phi)\ \ \text{for}\ \ \sin\phi>0,\\
\lim_{\eta\rt\infty}f(\eta,\phi)&=&f_{\infty}.
\end{array}
\right.
\end{eqnarray}
We can define the solution via taking limit $L\rt\infty$ in (\ref{Milne t 08}), (\ref{Milne t 09}) and (\ref{Milne t 10}) as follows:
\begin{eqnarray}
f(\eta,\phi)=\a h(\phi)+\t H(\eta,\phi),
\end{eqnarray}
where\\
\ \\
Case I: \\For $\sin\phi>0$,
\begin{eqnarray}\label{Milne temp 12}
\a h(\phi)&=&h(\phi'(\phi,\eta,0))\exp(-G_{\eta,0})\\
\t
H(\eta,\phi)&=&\int_0^{\eta}\frac{H(\eta',\phi'(\phi,\eta,\eta'))}{\sin(\phi'(\phi,\eta,\eta'))}\exp(-G_{\eta,\eta'})\ud{\eta'}.
\end{eqnarray}
\ \\
Case II: \\
For $\sin\phi<0$ and $\abs{E(\eta,\phi)}\leq e^{-V_{\infty}}$,
\begin{eqnarray}\label{Milne temp 14}
\a h(\phi)&=&0\\
\t
H(\eta,\phi)&=&\int_{\eta}^{\infty}\frac{H(\eta',R\phi'(\phi,\eta,\eta'))}{\sin(\phi'(\phi,\eta,\eta'))}\exp(G_{\eta,\eta'})\ud{\eta'}.
\end{eqnarray}
\ \\
Case III: \\
For $\sin\phi<0$ and $\abs{E(\eta,\phi)}\geq e^{-V_{\infty}}$,
\begin{eqnarray}\label{Milne temp 13}
\a h(\phi)&=&h(\phi'(\phi,\eta,0))\exp(-G_{\eta^+,0}-G_{\eta^+,\eta})\\
\t
H(\eta,\phi)&=&\bigg(\int_0^{\eta^+}\frac{H(\eta',\phi'(\phi,\eta,\eta'))}{\sin(\phi'(\phi,\eta,\eta'))}
\exp(-G_{\eta^+,\eta'}-G_{\eta^+,\eta})\ud{\eta'}\\
&&+
\int_{\eta}^{\eta^+}\frac{H(\eta',R\phi'(\phi,\eta,\eta'))}{\sin(\phi'(\phi,\eta,\eta'))}\exp(G_{\eta,\eta'})\ud{\eta'}\bigg)\nonumber.
\end{eqnarray}
Notice that
\begin{eqnarray}
\lim_{L\rt\infty}\exp(-G_{L,\eta})=0,
\end{eqnarray}
for $\sin\phi<0$ and $\abs{E(\eta,\phi)}\leq e^{-V_{\infty}}$.
Hence, above derivation is valid. In order to achieve the estimate
of $f$, we need to control $\t H$ and $\a h$.

\subsubsection{Preliminaries}

We first give several technical lemmas to be used for proving
$L^{\infty}$ estimates of $f$.
\begin{lemma}\label{Milne lemma 1}
For any $0\leq\beta\leq1$, we have
\begin{eqnarray}\label{Milne temp 51}
\lnm{e^{\beta\eta}\a h}\leq \lnm{h}.
\end{eqnarray}
In particular,
\begin{eqnarray}
\lnm{\a h}\leq \lnm{h}\label{Milne temp 52}.
\end{eqnarray}
\end{lemma}
\begin{proof}
Since $\phi'$ is always in the domain $[0,\pi)$, we naturally have
\begin{eqnarray}
0\leq\sin(\phi'(\phi,\eta,\xi))\leq 1,
\end{eqnarray}
which further implies
\begin{eqnarray}
\frac{1}{\sin(\phi'(\phi,\eta,\xi))}\geq 1.
\end{eqnarray}
Combined with the fact $\eta^+\geq\eta$, we deduce
\begin{eqnarray}
\exp(-G_{\eta,0})&\leq&e^{-\eta}\\
\exp(-G_{\eta^+,0}-G_{\eta^+,\eta})&\leq&\exp(-G_{\eta^+,0})\leq\exp(-G_{\eta,0})\leq
e^{-\eta}.
\end{eqnarray}
Hence, our result easily follows.
\end{proof}
\begin{lemma}\label{Milne lemma 2}
The integral operator $\t$ satisfies
\begin{eqnarray}\label{Milne temp 53}
\lnnm{\t H}\leq C\lnnm{H},
\end{eqnarray}
and for any $0\leq\beta\leq1/2$
\begin{eqnarray}\label{Milne temp 54}
\lnnm{e^{\beta\eta}\t H}\leq C\lnnm{e^{\beta\eta}H},
\end{eqnarray}
where $C$ is a universal constant independent of data.
\end{lemma}
\begin{proof}
For (\ref{Milne temp 53}), when $\sin\phi>0$
\begin{eqnarray}
\abs{\t
H}&\leq&\int_0^{\eta}\abs{H(\eta',\phi'(\phi,\eta,\eta'))}\frac{1}{\sin(\phi'(\phi,\eta,\eta'))}\exp(-G_{\eta,\eta'})\ud{\eta'}\\
&\leq&\lnnm{H}\int_0^{\eta}\frac{1}{\sin(\phi'(\phi,\eta,\eta'))}\exp(-G_{\eta,\eta'})\ud{\eta'}\nonumber.
\end{eqnarray}
We can directly estimate
\begin{eqnarray}\label{Milne t 11}
\int_0^{\eta}\frac{1}{\sin(\phi'(\phi,\eta,\eta'))}\exp(-G_{\eta,\eta'})\ud{\eta'}\leq\int_0^{\infty}e^{-z}\ud{z}=1,
\end{eqnarray}
and (\ref{Milne temp 53}) naturally follows. For $\sin\phi<0$ and
$\abs{E(\eta,\phi)}\leq e^{-V_{\infty}}$,
\begin{eqnarray}
\abs{\t
H}&\leq&\int_{\eta}^{\infty}\abs{H(\eta',\phi'(\phi,\eta,\eta'))}\frac{1}{\sin(\phi'(\phi,\eta,\eta'))}\exp(G_{\eta,\eta'})\ud{\eta'}\\
&\leq&\lnnm{H}\int_{\eta}^{\infty}\frac{1}{\sin(\phi'(\phi,\eta,\eta'))}\exp(G_{\eta,\eta'})\ud{\eta'}\nonumber.
\end{eqnarray}
we have
\begin{eqnarray}
\int_{\eta}^{\infty}\frac{1}{\sin(\phi'(\phi,\eta,\eta'))}\exp(G_{\eta,\eta'})\ud{\eta'}\leq\int_{-\infty}^0e^{z}\ud{z}=1,
\end{eqnarray}
and (\ref{Milne temp 53}) easily follows. The case $\sin\phi<0$ and
$\abs{E(\eta,\phi)}\geq e^{-V_{\infty}}$ can be proved combining
above two
techniques, so we omit it here.\\
For (\ref{Milne temp 54}), when $\sin\phi>0$, $\eta\geq\eta'$ and
$\beta<1/2$, since $G_{\eta,\eta'}\geq\eta-\eta'$, we have
\begin{eqnarray}
\beta(\eta-\eta')-G_{\eta,\eta'}\leq
\beta(\eta-\eta')-\half(\eta-\eta')-\half G_{\eta,\eta'}\leq -\half
G_{\eta,\eta'}.
\end{eqnarray}
Then it is natural that
\begin{eqnarray}
\int_0^{\eta}\frac{1}{\sin(\phi'(\phi,\eta,\eta'))}\exp(\beta(\eta-\eta')-G_{\eta,\eta'})\ud{\eta'}
&\leq&\int_0^{\eta}\frac{1}{\sin(\phi'(\phi,\eta,\eta'))}\exp(-G_{\eta,\eta'}/2)\ud{\eta'}\\
&\leq&\int_0^{\infty}e^{-z/2}\ud{z}=2\nonumber.
\end{eqnarray}
This leads to
\begin{eqnarray}
\abs{e^{\beta\eta}\t H}&\leq&
e^{\beta\eta}\int_0^{\eta}\abs{H(\eta',\phi'(\phi,\eta,\eta'))}\frac{1}{\sin(\phi'(\phi,\eta,\eta'))}\exp(-G_{\eta,\eta'})\ud{\eta'}\\
&\leq&\lnnm{e^{\beta\eta}H}\int_0^{\eta}\frac{1}{\sin(\phi'(\phi,\eta,\eta'))}\exp(\beta(\eta-\eta')-G_{\eta,\eta'})\ud{\eta'}\nonumber\\
&\leq&C\lnnm{e^{\beta\eta}H},\nonumber
\end{eqnarray}
and (\ref{Milne temp 54}) naturally follows. For $\sin\phi<0$ and
$\abs{E(\eta,\phi)}\leq e^{-V_{\infty}}$, note for $\eta'\geq\eta$
\begin{eqnarray}
\beta(\eta-\eta')+G_{\eta,\eta'}\leq
\beta(\eta-\eta')+\half(\eta-\eta')+\half G_{\eta,\eta'}\leq \half
G_{\eta,\eta'}.
\end{eqnarray}
Then (\ref{Milne temp 54}) holds by obvious modifications of
$\sin\phi>0$ case. The case $\sin\phi<0$ and $\abs{E(\eta,\phi)}\geq
e^{-V_{\infty}}$ can be shown by combining above two cases, so we
omit it here.
\end{proof}
\begin{lemma}\label{Milne lemma 3}
For any $\delta>0$ there is a constant $C(\delta)>0$ independent of
data such that
\begin{eqnarray}
\ltnm{\t H}\leq C(\delta)\tnnm{H}+\delta\lnnm{H}\label{Milne temp
11}.
\end{eqnarray}
\end{lemma}
\begin{proof}
We divide the proof into several steps.\\
\ \\
Step 1: The case of $\sin\phi>0$.\\
We consider
\begin{eqnarray}
\int_{\sin\phi>0}\abs{\t
H(\eta)}^2\ud{\phi}&=&\int_{\sin\phi>0}\bigg(\int_0^{\eta}\frac{H(\eta',\phi'(\phi,\eta,\eta'))}{\sin(\phi'(\phi,\eta,\eta'))}
\exp(-G_{\eta,\eta'})\ud{\eta'}\bigg)^2\ud{\phi}\\
&=&\int\bigg(\int{\bf{1}}_{\{\sin(\phi'(\phi,\eta,\eta'))>m\}}
\ldots\bigg)^2+\int\bigg(\int{\bf{1}}_{\{\sin(\phi'(\phi,\eta,\eta'))\leq m\}}
\ldots\bigg)^2\nonumber\\
&=&I_1+I_2\nonumber.
\end{eqnarray}
By Cauchy's inequality and (\ref{Milne t 11}), we get
\begin{eqnarray}\label{Milne temp 55}
\\
I_1&\leq&\int_{\sin\phi>0}\bigg(\int_0^{\eta}\abs{H(\eta',\phi'(\phi,\eta,\eta'))}^2\ud{\eta'}\bigg)
\bigg(\int_0^{\eta}{\bf{1}}_{\{\sin(\phi'(\phi,\eta,\eta'))>m\}}\frac{\exp(-2G_{\eta,\eta'})}{\sin^2(\phi'(\phi,\eta,\eta'))}
\ud{\eta'}\bigg)\ud{\phi}\nonumber\\
&\leq&\frac{1}{m}\tnnm{H}^2\int_{\sin\phi>0}
\bigg(\int_0^{\eta}{\bf{1}}_{\{\sin(\phi'(\phi,\eta,\eta'))>m\}}\frac{\exp(-2G_{\eta,\eta'})}{\sin(\phi'(\phi,\eta,\eta'))}
\ud{\eta'}\bigg)\ud{\phi}\nonumber\\
&\leq&\frac{\pi}{m}\tnnm{H}^2\nonumber.
\end{eqnarray}
On the other hand, for $\eta'\leq\eta$, we can directly estimate
$\phi'(\phi,\eta,\eta')\geq\phi$. Hence, we have the relation
\begin{eqnarray}
\sin\phi\leq\sin(\phi'(\phi,\eta,\eta')).
\end{eqnarray}
Therefore, we can directly estimate $I_2$ as follows:
\begin{eqnarray}
I_2&\leq&\lnnm{H}^2\int_{\sin\phi>0}
\bigg(\int_0^{\eta}{\bf{1}}_{\{\sin(\phi'(\phi,\eta,\eta'))\leq m\}}
\frac{1}{\sin(\phi'(\phi,\eta,\eta'))}\exp(-G_{\eta,\eta'})\ud{\eta'}\bigg)^2\ud{\phi}\\
&\leq&\lnnm{H}^2\int_{\sin\phi>0}
\bigg(\int_0^{\eta}{\bf{1}}_{\{\sin\phi\leq m\}}
\frac{1}{\sin(\phi'(\phi,\eta,\eta'))}\exp(-G_{\eta,\eta'})\ud{\eta'}\bigg)^2\ud{\phi}\nonumber\\
&=&\lnnm{H}^2\int_{\sin\phi>0}{\bf{1}}_{\{\sin\phi\leq m\}}
\bigg(\int_0^{\eta}
\frac{1}{\sin(\phi'(\phi,\eta,\eta'))}\exp(-G_{\eta,\eta'})\ud{\eta'}\bigg)^2\ud{\phi}\nonumber.
\end{eqnarray}
Note
\begin{eqnarray}
\int_0^{\eta}
\frac{1}{\sin(\phi'(\phi,\eta,\eta'))}\exp(-G_{\eta,\eta'})\ud{\eta'}\leq\int_{0}^{\infty}e^{-z}\ud{z}=1.
\end{eqnarray}
Therefore, for $m$ sufficiently small, we have
\begin{eqnarray}\label{Milne temp 56}
I_2&\leq&\lnnm{H}^2\int_{\sin\phi>0}{\bf{1}}_{\{\sin\phi\leq
m\}}\ud{\phi}\leq 4m.
\end{eqnarray}
Summing up (\ref{Milne temp 55}) and (\ref{Milne temp 56}), for $m$
sufficiently small, we deduce (\ref{Milne temp 11}).\\
\ \\
Step 2: The case of $\sin\phi<0$ and $\abs{E(\eta,\phi)}\leq e^{-V_{\infty}}$.\\
We can decompose
\begin{eqnarray}
&&\int_{\sin\phi<0}{\bf{1}}_{\{\abs{E(\eta,\phi)}\leq
e^{-V_{\infty}}\}}\abs{\t H}^2\ud{\phi}\\
&=&\int_{\sin\phi<0}{\bf{1}}_{\{\abs{E(\eta,\phi)}\leq
e^{-V_{\infty}}\}}\bigg(\int_{\eta}^{\infty}
\frac{H(\eta',R\phi'(\phi,\eta,\eta'))}{\sin(\phi'(\phi,\eta,\eta'))}
\exp(G_{\eta,\eta'})\ud{\eta'}\bigg)^2\ud{\phi}\nonumber\\
&=&\int\bigg(\int
{\bf{1}}_{\{\sin(\phi'(\phi,\eta,\eta'))>m\}}\ldots\bigg)^2+\int\bigg(\int{\bf{1}}_{\{\sin(\phi'(\phi,\eta,\eta'))\leq m\}}
{\bf{1}}_{\{\eta'-\eta\geq\sigma\}}
\ldots\bigg)^2\nonumber\\
&&\int\bigg(\int{\bf{1}}_{\{\sin(\phi'(\phi,\eta,\eta'))\leq m\}}
{\bf{1}}_{\{\eta'-\eta\leq\sigma\}}
\ldots\bigg)^2\nonumber\\
&=&I_1+I_2+I_3\nonumber.
\end{eqnarray}
We can directly estimate $I_1$ as follows:
\begin{eqnarray}\label{Milne temp 57}
\\
I_1&\leq&\int_{\sin\phi<0}\bigg(\int_{\eta}^{\infty}\abs{H(\eta',R\phi'(\phi,\eta,\eta'))}^2\ud{\eta'}\bigg)
\bigg(\int_{\eta}^{\infty}{\bf{1}}_{\{\sin(\phi'(\phi,\eta,\eta'))>m\}}\frac{\exp(2G_{\eta,\eta'})}{\sin^2(\phi'(\phi,\eta,\eta'))}
\ud{\eta'}\bigg)\ud{\phi}\nonumber\\
&\leq&\frac{1}{m}\tnnm{H}^2\int_{\sin\phi<0}
\bigg(\int_{\eta}^{\infty}{\bf{1}}_{\{\sin(\phi'(\phi,\eta,\eta'))>m\}}\frac{\exp(2G_{\eta,\eta'})}{\sin(\phi'(\phi,\eta,\eta'))}
\ud{\eta'}\bigg)\ud{\phi}\nonumber\\
&\leq&\frac{\pi}{m}\tnnm{H}^2\nonumber.
\end{eqnarray}
On the other hand, for $I_2$ we have
\begin{eqnarray}
\\
I_2&\leq&\lnnm{H}^2\int_{\sin\phi<0}
{\bf{1}}_{\{\abs{E(\eta,\phi)}\leq
e^{-V_{\infty}}\}}\bigg(\int_{\eta}^{\infty}{\bf{1}}_{\{\sin(\phi'(\phi,\eta,\eta'))\leq
m\}}{\bf{1}}_{\{\eta'-\eta\geq\sigma\}}
\frac{\exp(G_{\eta,\eta'})}{\sin(\phi'(\phi,\eta,\eta'))}\ud{\eta'}\bigg)^2\ud{\phi}\nonumber.
\end{eqnarray}
Note
\begin{eqnarray}
G_{\eta,\eta'}=\int_{\eta'}^{\eta}\frac{1}{\sin(\phi'(\phi,\eta,y))}\ud{y}\leq-\frac{\eta'-\eta}{m}=-\frac{\sigma}{m}.
\end{eqnarray}
Then we can obtain
\begin{eqnarray}\label{Milne temp 58}
I_2&\leq&\lnnm{H}^2\int_{\sin\phi<0}
\bigg(\int^{-\sigma/m}_{\infty}e^z\ud{z}\bigg)^2\ud{\phi}=4
e^{-\frac{\sigma}{m}}\lnnm{H}^2.
\end{eqnarray}
For $I_3$, we can estimate as follows:
\begin{eqnarray}
\\
I_3&\leq&\lnnm{H}^2\int_{\sin\phi<0}
{\bf{1}}_{\{\abs{E(\eta,\phi)}\leq
e^{-V_{\infty}}\}}\bigg(\int_{\eta}^{\infty}{\bf{1}}_{\{\sin(\phi'(\phi,\eta,\eta'))\leq
m\}}{\bf{1}}_{\{\eta'-\eta\leq\sigma\}}
\frac{\exp(G_{\eta,\eta'})}{\sin(\phi'(\phi,\eta,\eta'))}\ud{\eta'}\bigg)^2\ud{\phi}\nonumber\\
&\leq&\lnnm{H}^2\int_{\sin\phi<0}{\bf{1}}_{\{\abs{E(\eta,\phi)}\leq
e^{-V_{\infty}}\}}
\bigg(\int_{\eta}^{\eta+\sigma}{\bf{1}}_{\{\sin(\phi'(\phi,\eta,\eta'))\leq
m\}}{\bf{1}}_{\{\eta'-\eta\leq\sigma\}}
\frac{\exp(G_{\eta,\eta'})}{\sin(\phi'(\phi,\eta,\eta'))}\ud{\eta'}\bigg)^2\ud{\phi}\nonumber.
\end{eqnarray}
Note
\begin{eqnarray}
\int_{\eta}^{\infty}
\frac{1}{\sin(\phi'(\phi,\eta,\eta'))}\exp(G_{\eta,\eta'})\ud{\eta'}\leq\int_{-\infty}^{0}e^{z}\ud{z}=1.
\end{eqnarray}
Then $1\leq\alpha=e^{V(\eta')-V(\eta)}\leq
e^{V(\eta+\sigma)-V(\eta)}\leq 1+4\sigma$ due to (\ref{force temp
6}), and for $\eta'\in[\eta,\eta+\sigma]$,
$\sin\phi'(\phi,\eta,\eta'))=\sin\bigg(\cos^{-1}(\alpha\cos\phi)\bigg)$,
$\sin(\phi'(\phi,\eta,\eta'))<m$ lead to
\begin{eqnarray}
\abs{\sin\phi}&=&\sqrt{1-\cos^2\phi}=\sqrt{1-\frac{\cos^2\phi'(\phi,\eta,\eta')}{\alpha^2}}
=\frac{\sqrt{\alpha^2-\bigg(1-\sin^2\phi'(\phi,\eta,\eta')\bigg)}}{\alpha}\\
&\leq&\frac{\sqrt{\alpha^2-1+m^2}}{\alpha}\leq\frac{\sqrt{(1+4\sigma)^2-1+m^2}}{\alpha}\leq
\sqrt{9\sigma+m^2}.\nonumber
\end{eqnarray}
Hence, we can obtain
\begin{eqnarray}\label{Milne temp 59}
I_3&\leq&\lnnm{H}^2\int_{\sin\phi<0}{\bf{1}}_{\{\sin(\phi'(\phi,\eta,\eta'))\leq
m\}}\ud{\phi}\leq\lnnm{H}^2
\int_{\sin\phi<0}{\bf{1}}_{\{\abs{\sin\phi}\leq \sqrt{9\sigma+m^2}\}}\ud{\phi}\\
&\leq&4\sqrt{9\sigma+m^2}\nonumber.
\end{eqnarray}
Summarizing (\ref{Milne temp 57}), (\ref{Milne temp 58}) and
(\ref{Milne temp 59}), for sufficiently small $\sigma$,
we can always choose $m$ small enough to guarantee the relation (\ref{Milne temp 11}).\\
\ \\
Step 3: The case of $\sin\phi<0$ and $\abs{E(\eta,\phi)}\geq e^{-V_{\infty}}$.\\
We can decompose $\t H$ as follows:
\begin{eqnarray}
&&\t
H(\eta,\phi)\\
&=&\bigg(\int_0^{\eta^+}\ldots
\exp(-G_{\eta^+,\eta'}-G_{\eta^+,\eta})\ud{\eta'}+
\int_{\eta}^{\eta^+}\ldots\exp(G_{\eta,\eta'})\ud{\eta'}\bigg)\nonumber\\
&=&\int_0^{\eta}\ldots
\exp(-G_{\eta^+,\eta'}-G_{\eta^+,\eta})\ud{\eta'}+\bigg(\int_{\eta}^{\eta^+}\ldots
\exp(-G_{\eta^+,\eta'}-G_{\eta^+,\eta})\ud{\eta'}+
\int_{\eta}^{\eta^+}\ldots\exp(G_{\eta,\eta'})\ud{\eta'}\bigg)\nonumber\\
&=&I_1+I_2\nonumber.
\end{eqnarray}
For $I_1$, we can apply a similar argument as in Step 1 and for
$I_2$, a similar argument as in Step 2 conclude the lemma.
\end{proof}

\subsubsection{Estimates of $\e$-Milne Equation}

Consider the equation satisfied by $z=f-f_{\infty}$ as follows:
\begin{eqnarray}\label{difference equation}
\left\{
\begin{array}{rcl}\displaystyle
\sin\phi\frac{\p z}{\p\eta}+F(\eta)\cos\phi\frac{\p
z}{\p\phi}+z&=&\bar z+S,\\
z(0,\phi)&=&p(\phi)=h(\phi)-f_{\infty}\ \ \text{for}\ \ \sin\phi>0,\\
\lim_{\eta\rt\infty}z(\eta,\phi)&=&0.
\end{array}
\right.
\end{eqnarray}
\begin{lemma}\label{Milne lemma 4}
Assume (\ref{Milne bounded}) and (\ref{Milne decay}) hold. Then there
exists a constant $C$ depending on the data such that the solution
of equation (\ref{difference equation}) verifies
\begin{eqnarray}
\lnnm{z}\leq C\bigg(1+M+\frac{M}{K}+\tnnm{z}\bigg).
\end{eqnarray}
\end{lemma}
\begin{proof}
We first show the following important facts:
\begin{eqnarray}
\tnnm{\bar z}&\leq&\tnnm{z}\label{Milne temp 61},\\
\lnnm{\bar z}&\leq&\ltnm{z}\label{Milne temp 62}.
\end{eqnarray}
We can directly derive them by Cauchy's inequality as follows:
\begin{eqnarray}
\\
\tnnm{\bar
z}^2&=&\int_0^{\infty}\int_{-\pi}^{\pi}\bigg(\frac{1}{2\pi}\bigg)^2\bigg(\int_{-\pi}^{\pi}z(\eta,\phi)\ud{\phi}\bigg)^2\ud{\phi}\ud{\eta}
\leq\int_0^{\infty}\int_{-\pi}^{\pi}\bigg(\frac{1}{2\pi}\bigg)\bigg(\int_{-\pi}^{\pi}z^2(\eta,\phi)\ud{\phi}\bigg)\ud{\phi}\ud{\eta}\nonumber\\
&=&\int_0^{\infty}\bigg(\int_{-\pi}^{\pi}z^2(\eta,\phi)\ud{\phi}\bigg)\ud{\eta}=\tnnm{z}^2\nonumber.\\
\\
\lnnm{\bar z}^2&=&\sup_{\eta}\bar
z^2(\eta)=\sup_{\eta}\bigg(\frac{1}{2\pi}\int_{-\pi}^{\pi}z(\eta,\phi)\ud{\phi}\bigg)^2
\leq\sup_{\eta}\bigg(\frac{1}{2\pi}\bigg)^2\bigg(\int_{-\pi}^{\pi}z^2(\eta,\phi)\ud{\phi}\bigg)\bigg(\int_{-\pi}^{\pi}1^2\ud{\phi}\bigg)\nonumber\\
&=&\sup_{\eta}\bigg(\int_{-\pi}^{\pi}z^2(\eta,\phi)\ud{\phi}\bigg)=\ltnm{z}^2\nonumber.
\end{eqnarray}
By (\ref{difference equation}), $z=\a p+\t(\bar z+S)$ leads to
\begin{eqnarray}
\t(\bar z+S)=z-\a p,
\end{eqnarray}
Then by Lemma \ref{Milne lemma 3}, (\ref{Milne temp 61}) and
(\ref{Milne temp 62}), we can show
\begin{eqnarray}\label{Milne temp 63}
\ltnm{z-\a p}&\leq& C(\delta)\bigg(\tnnm{\bar z}+\tnnm{S}\bigg)+\delta\bigg(\lnnm{\bar z}+\lnnm{S}\bigg)\\
&\leq&
C(\delta)\bigg(\tnnm{z}+\tnnm{S}\bigg)+\delta\bigg(\ltnm{z}+\lnnm{S}\bigg)\nonumber.
\end{eqnarray}
Therefore, based on Lemma \ref{Milne lemma 1} and (\ref{Milne temp
63}), we can directly estimate
\begin{eqnarray}
\ltnm{z}&\leq&\lnm{\a p}+C(\delta)\bigg(\tnnm{z}+\tnnm{S}\bigg)+\delta\bigg(\ltnm{z}+\lnnm{S}\bigg)\\
&\leq&\lnm{p}+C(\delta)\bigg(\tnnm{z}+\tnnm{S}\bigg)+\delta\bigg(\ltnm{z}+\lnnm{S}\bigg)\nonumber.
\end{eqnarray}
We can take $\delta=1/2$ to obtain
\begin{eqnarray}\label{Milne temp 64}
\ltnm{z}\leq C\bigg(\tnnm{z}+\tnnm{S}+\lnnm{S}+\lnm{p}\bigg).
\end{eqnarray}
Therefore, based on Lemma \ref{Milne lemma 2}, (\ref{Milne temp 64})
and (\ref{Milne temp 62}), we can achieve
\begin{eqnarray}
\\
\lnnm{z}&\leq&\lnnm{\a p}+\lnnm{\t(\bar z+S)}\leq
C\bigg(\lnm{p}+\lnnm{\bar z}+\lnnm{S}\bigg)\nonumber\\
&\leq&C\bigg(\lnm{p}+\lnnm{S}+\ltnm{z}\bigg)\leq
C\bigg(\lnm{p}+\lnnm{S}+\tnnm{S}+\tnnm{z}\bigg)\nonumber.
\end{eqnarray}
Since $\lnm{p}$, $\tnnm{S}$ and $\lnnm{S}$ are finite, our result
easily follows.
\end{proof}
Lemma \ref{Milne lemma 4} naturally implies the following.
\begin{theorem}\label{Milne lemma 5}
The solution $f(\eta,\phi)$ to the Milne problem (\ref{Milne
problem}) satisfies
\begin{eqnarray}
\lnnm{f-f_{\infty}}\leq C\bigg(1+M+\frac{M}{K}+\tnnm{f-f_{\infty}}\bigg).
\end{eqnarray}
\end{theorem}
Combining Theorem \ref{Milne lemma 5} and Theorem \ref{Milne lemma
6}, we deduce the main theorem.
\begin{theorem}\label{Milne theorem 1}
There exists a unique solution $f(\eta,\phi)$ to the $\e$-Milne problem
(\ref{Milne problem}) satisfying
\begin{eqnarray}
\lnnm{f-f_{\infty}}\leq C\bigg(1+M+\frac{M}{K}\bigg).
\end{eqnarray}
\end{theorem}

\subsection{Exponential Decay}

In this section, we prove the spatial decay of the solution to the
Milne problem.
\begin{theorem}\label{Milne theorem 2}
Assume (\ref{Milne bounded}) and (\ref{Milne decay}) hold. For $K_0>0$ sufficiently small, the solution $f(\eta,\phi)$ to the
$\e$-Milne problem (\ref{Milne problem}) satisfies
\begin{eqnarray}\label{Milne temp 85}
\lnnm{e^{K_0\eta}(f-f_{\infty})}\leq C\bigg(1+M+\frac{M}{K}\bigg),
\end{eqnarray}
\end{theorem}
\begin{proof}
Define $Z=e^{K_0\eta}g$ for $z=f-f_{\infty}$.
We divide the analysis into several steps:\\
\ \\
Step 1: We have
\begin{eqnarray}\label{Milne temp 72}
\tnnm{Z}^2=\int_0^{\infty}e^{2K_0\eta}\bigg(\int_{-\pi}^{\pi}(f(\eta,\phi)-f_{\infty})^2\ud{\phi}\bigg)\ud{\eta}\leq
C\bigg(1+M+\frac{M}{K}\bigg)^2.
\end{eqnarray}
\emph{The proof of (\ref{Milne temp 72}):} The orthogonal property
(\ref{Milne temp 3}) reveals
\begin{eqnarray}
\br{f,f\sin\phi}_{\phi}(\eta)=\br{r,r\sin\phi}_{\phi}(\eta).
\end{eqnarray}
Multiplying $e^{2K_0\eta}f$ on both sides of equation (\ref{Milne
problem}) and integrating over $\phi\in[-\pi,\pi)$, we obtain
\begin{eqnarray}\label{Milne temp 15}
\\
\half\frac{\ud{}}{\ud{\eta}}\bigg(e^{2K_0\eta}\br{r,r\sin\phi}_{\phi}(\eta)\bigg)
+\half
F(\eta)\bigg(e^{2K_0\eta}\br{r,r\sin\phi}_{\phi}(\eta)\bigg)
-e^{2K_0\eta}\bigg(K_0\br{r,r\sin\phi}_{\phi}(\eta)-\br{r,
r}_{\phi}(\eta)\bigg) \nonumber\\
=e^{2K_0\eta}\br{S,f}_{\phi}(\eta)\nonumber.
\end{eqnarray}
For $K_0<\min\{1/2,K\}$, we have
\begin{eqnarray}\label{Milne temp 16}
\frac{3}{2}\tnm{r(\eta)}^2\geq-K_0\br{r,r\sin\phi}_{\phi}(\eta)+\br{r,r}_{\phi}(\eta)\geq
\half\tnm{r(\eta)}^2.
\end{eqnarray}
Similar to the proof of Lemma \ref{Milne finite LT} and Lemma
\ref{Milne infinite LT}, formula as (\ref{Milne temp 15}) and
(\ref{Milne temp 16}) imply
\begin{eqnarray}\label{Milne temp 71}
\tnnm{e^{K_0\eta}r}^2=\int_0^{\infty}e^{2K_0\eta}\br{r,r}_{\phi}(\eta)\ud{\eta}\leq
C\bigg(1+M+\frac{M}{K}\bigg)^2.
\end{eqnarray}
From (\ref{Milne temp 9}), Cauchy's inequality and
(\ref{force temp 3}), we can deduce
\begin{eqnarray}
\\
&&\int_0^{\infty}e^{2K_0\eta}\bigg(\int_{-\pi}^{\pi}(f(\eta,\phi)-f_{\infty})^2\ud{\phi}\bigg)\ud{\eta}\nonumber\\
&\leq&\int_0^{\infty}e^{2K_0\eta}\bigg(\int_{-\pi}^{\pi}r^2(\eta,\phi)\ud{\phi}\bigg)\ud{\eta}+
\int_0^{\infty}e^{2K_0\eta}\bigg(\int_{-\pi}^{\pi}(q(\eta)-q_{\infty})^2\ud{\phi}\bigg)\ud{\eta}
\nonumber\\
&\leq&\int_0^{\infty}e^{2K_0\eta}\tnm{r(\eta)}^2\ud{\eta}\nonumber\\
&&+\int_0^{\infty}e^{2K_0\eta}\bigg(\int_{\eta}^{\infty}\abs{F(y)}\tnm{r(y)}\ud{y}\bigg)^2\ud{\eta}
+\int_0^{\infty}e^{2K_0\eta}\bigg(\int_{\eta}^{\infty}\lnm{S(y)}\ud{y}\bigg)^2\ud{\eta}
\nonumber\\
&\leq&C\bigg(1+M+\frac{M}{K}\bigg)^2\nonumber\\
&&+C\bigg(\int_0^{\infty}e^{2K_0\eta}\tnm{r(\eta)}^2\ud{\eta}\bigg)
\bigg(\int_0^{\infty}\int_{\eta}^{\infty}e^{2K_0(\eta-y)}F^2(y)\ud{y}\ud{\eta}\bigg)
+\int_0^{\infty}e^{2K_0\eta}\bigg(\int_{\eta}^{\infty}\lnm{S(y)}\ud{y}\bigg)^2\ud{\eta}
\nonumber\\
&\leq&C\bigg(1+M+\frac{M}{K}\bigg)^2\nonumber\\
&&+C\bigg(\int_0^{\infty}e^{2K_0\eta}\tnm{r(\eta)}^2\ud{\eta}\bigg)
\bigg(\int_0^{\infty}\int_{\eta}^{\infty}F^2(y)\ud{y}\ud{\eta}\bigg)
+\int_0^{\infty}e^{2K_0\eta}\bigg(\int_{\eta}^{\infty}\lnm{S(y)}\ud{y}\bigg)^2\ud{\eta}
\nonumber\\
&\leq&C\bigg(1+M+\frac{M}{K}\bigg)^2\nonumber.
\end{eqnarray}
This completes the proof of (\ref{Milne temp 72}) when $\bar S=0$.
By the method introduced in Lemma \ref{Milne infinite LT general},
we can extend above $L^2$ estimates to the general $S$ case. Note
all the auxiliary functions
constructed in Lemma \ref{Milne infinite LT general} satisfy the estimates (\ref{Milne temp 72}). \\
\ \\
Step 2: We have
\begin{eqnarray}\label{Milne temp 75}
\lnnm{Z}\leq C\bigg(1+M+\frac{M}{K}+\tnnm{Z}\bigg).
\end{eqnarray}
\emph{Proof of (\ref{Milne temp 75}):} $Z$ satisfies the equation
\begin{eqnarray}\label{decay equation}
\left\{
\begin{array}{rcl}\displaystyle
\sin\phi\frac{\p Z}{\p\eta}+F(\eta)\cos\phi\frac{\p
Z}{\p\phi}+Z&=&\bar Z+e^{K_0\eta}S+K_0\sin\phi Z,\\
Z(0,\phi)&=&p(\phi)=h(\phi)-f_{\infty}\ \ \text{for}\ \ \sin\phi>0.
\end{array}
\right.
\end{eqnarray}
Since we know $Z=\a p+\t(\bar Z+e^{K_0\eta}S+K_0\sin\phi Z)$
leads to
\begin{eqnarray}
\t(\bar Z+e^{K_0\eta}S+K_0\sin\phi Z)=Z-\a p,
\end{eqnarray}
then by Lemma \ref{Milne lemma 3}, (\ref{Milne temp 61}) and
(\ref{Milne temp 62}), we can show
\begin{eqnarray}\label{Milne temp 73}
&&\ltnm{Z-\a p}\\
&\leq& C(\delta)\bigg(\tnnm{\bar
Z}+\tnnm{e^{K_0\eta}S}+K_0\tnnm{Z}\bigg)
+\delta\bigg(\lnnm{\bar Z}+\lnnm{e^{K_0\eta}S}+K_0\lnnm{Z}\bigg)\nonumber\\
&\leq&
C(\delta)\bigg(\tnnm{Z}+\tnnm{e^{K_0\eta}S}+K_0\tnnm{Z}\bigg)+\delta\bigg(\ltnm{Z}+\lnnm{e^{K_0\eta}S}+K_0\lnnm{Z}\bigg)\nonumber.
\end{eqnarray}
Therefore, based on Lemma \ref{Milne lemma 1} and (\ref{Milne temp
63}), we can directly estimate
\begin{eqnarray}
&&\ltnm{Z}\\
&\leq&\lnm{\a
p}+C(\delta)\bigg(\tnnm{Z}+\tnnm{e^{K_0\eta}S}+K_0\tnnm{Z}\bigg)\nonumber\\
&&+\delta\bigg(\ltnm{Z}+\lnnm{e^{K_0\eta}S}+K_0\lnnm{Z}\bigg)\nonumber\\
&\leq&\lnm{p}+2C(\delta)\bigg(\tnnm{Z}+\tnnm{e^{K_0\eta}S}\bigg)+\delta\bigg(\ltnm{Z}+\lnnm{e^{K_0\eta}S}+K_0\lnnm{Z}\bigg)\nonumber.
\end{eqnarray}
We can take $\delta=1/2$ to obtain
\begin{eqnarray}\label{Milne temp 74}
\ltnm{Z}\leq
C\bigg(\tnnm{Z}+\tnnm{S}+\lnnm{S}+\lnm{p}+K_0\lnnm{Z}\bigg).
\end{eqnarray}
Then based on Lemma \ref{Milne lemma 1}, Lemma \ref{Milne lemma 2}
and Lemma \ref{Milne lemma 3}, we can deduce
\begin{eqnarray}
\lnnm{Z}&\leq& \lnm{e^{K_0\eta}\a p}+\lnnm{e^{K_0\eta}\t S}+\lnnm{\bar Z}\\
&\leq& \lnm{p}+\lnnm{e^{K_0\eta}S}+\lnnm{\bar Z}\nonumber\\
&\leq& \lnm{p}+\lnnm{e^{K_0\eta}S}+\ltnm{Z}\nonumber\\
&\leq&
C\bigg(\tnnm{Z}+\tnnm{e^{K_0\eta}S}+\lnnm{e^{K_0\eta}S}+\lnm{p}+K_0\lnnm{Z}\bigg)\nonumber.
\end{eqnarray}
Taking $K_0$ sufficiently small, this completes the proof of
(\ref{Milne temp 75}).\\
\ \\
Combining (\ref{Milne temp 72}) and (\ref{Milne temp 75}), we
deduce (\ref{Milne temp 85}).
\end{proof}

\subsection{Maximum Principle}

\begin{theorem}\label{Milne theorem 3}
The solution $f(\eta,\phi)$ to the $\e$-Milne problem (\ref{Milne
problem}) with $S=0$ satisfies the maximum principle, i.e.
\begin{eqnarray}
\min_{\sin\phi>0}h(\phi)\leq f(\eta,\phi)\leq
\max_{\sin\phi>0}h(\phi).
\end{eqnarray}
\end{theorem}
\begin{proof}
We claim it is sufficient to show $f(\eta,\phi)\leq0$ whenever
$h(\phi)\leq0$. Suppose this claim is justified.
Denote $m=\min_{\sin\phi>0}h(\phi)$ and
$M=\max_{\sin\phi>0}h(\phi)$. Then $f^1=f-M$
satisfies the equation
\begin{eqnarray}
\left\{
\begin{array}{rcl}\displaystyle
\sin\phi\frac{\p f^1}{\p\eta}+F(\eta)\cos\phi\frac{\p
f^1}{\p\phi}+f^1-\bar f^1&=&0,\\
f^1(0,\phi)&=&h(\phi)-M\ \ \text{for}\ \ \sin\phi>0,\\
\lim_{\eta\rt\infty}f^1(\eta,\phi)&=&f^1_{\infty}.
\end{array}
\right.
\end{eqnarray}
Hence, $h-M\leq0$ implies $f^1\leq0$ which is actually $f\leq M$. On
the other hand, $f^2=m-f$ satisfies the equation
\begin{eqnarray}
\left\{
\begin{array}{rcl}\displaystyle
\sin\phi\frac{\p f^2}{\p\eta}+F(\eta)\cos\phi\frac{\p
f^2}{\p\phi}+f^2-\bar f^2&=&0,\\
f^2(0,\phi)&=&m-h(\phi)\ \ \text{for}\ \ \sin\phi>0,\\
\lim_{\eta\rt\infty}f^2(\eta,\phi)&=&f^2_{\infty}.
\end{array}
\right.
\end{eqnarray}
Thus, $m-h\leq0$ implies $f^2\leq0$ which further leads to $f\geq
m$. Therefore, the maximum principle is established.\\
\ \\
We now prove if $h(\phi)\leq0$, we have $f(\eta,\phi)\leq0$.
We divide the proof into several steps:\\
\ \\
Step 1: Penalized $\e$-Milne problem in a finite slab.\\
Assuming $h(\phi)\leq0$, we then consider the penalized Milne
problem for $f^L_{\l}(\eta,\phi)$ in the finite slab $(\eta,\phi)\in[0,L]\times[-\pi,\pi)$
\begin{eqnarray}\label{Milne finite problem LT penalty_}
\left\{
\begin{array}{rcl}\displaystyle
\l f_{\l}^{L}+\sin\phi\frac{\p f_{\l}^{L}}{\p\eta}+F(\eta)\cos\phi\frac{\p
f_{\l}^{L}}{\p\phi}+f_{\l}^{L}-\bar f_{\l}^{L}&=&0,\\
f_{\l}^{L}(0,\phi)&=&h(\phi)\ \ \text{for}\ \ \sin\phi<0,\\
f_{\l}^{L}(L,\phi)&=&f_{\l}^{L}(L,R\phi).
\end{array}
\right.
\end{eqnarray}
In order to construct the solution of (\ref{Milne finite problem LT
penalty_}), we iteratively define the sequence
$\{f^{L}_{m}\}_{m=1}^{\infty}$ as $f^{L}_{0}=0$ and
\begin{eqnarray}
\left\{
\begin{array}{rcl}\displaystyle
\l f^{L}_{m}+\sin\phi\frac{\p
f^{L}_{m}}{\p\eta}+F(\eta)\cos\phi\frac{\p
f^{L}_{m}}{\p\phi}+f^{L}_{m}-\bar f^{L}_{m-1}&=&0,\\
f^{L}_{m}(0,\phi)&=&h(\phi)\ \ \text{for}\ \ \sin\phi<0,\\
f^{L}_{m}(L,\phi)&=&f^{L}_{m}(L,R\phi).
\end{array}
\right.
\end{eqnarray}
Along the characteristics, it is easy to see we always have
$f^{L}_{m}<0$. In the proof of Lemma \ref{Milne finite LT}, we have
shown $f^{L}_{m}$ converges strongly in
$L^{\infty}([0,L]\times[-\pi,\pi))$ to $f_{\l}^{L}$ which satisfies
(\ref{Milne finite problem LT penalty_}). Also, $f_{\l}^{L}$ satisfies
\begin{eqnarray}
\lnnm{f_{\l}^{L}}\leq \frac{1+\l}{\l}\lnm{h}.
\end{eqnarray}
Naturally, we obtain $f_{\l}^{L}\in L^2([0,L]\times[-\pi,\pi))$
and $f_{\l}^{L}\leq0$.\\
\ \\
Step 2: $\e$-Milne problem in a finite slab.\\
Consider the Milne problem for $f^{L}(\eta,\phi)$ in a finite slab $(\eta,\phi)\in[0,L]\times[-\pi,\pi)$
\begin{eqnarray}\label{Milne finite problem LT_}
\left\{
\begin{array}{rcl}\displaystyle
\sin\phi\frac{\p f^L}{\p\eta}+F(\eta)\cos\phi\frac{\p
f^L}{\p\phi}+f^L-\bar f^L&=&0,\\
f^L(0,\phi)&=&h(\phi)\ \ \text{for}\ \ \sin\phi<0,\\
f^L(L,\phi)&=&f^L(L,R\phi).
\end{array}
\right.
\end{eqnarray}
In the proof of Lemma \ref{Milne finite LT}, we have shown $f_{\l}^{L}$
is uniformly bounded in $L^{2}([0,L)\times[-\pi,\pi))$ with respect
to $\l$, which implies we can take weakly convergent subsequence
$f_{\l}^{L}\rightharpoonup f^L$ as $\l\rt0$ with $f^L\in
L^2([0,L]\times[-\pi,\pi))$.
Naturally, we have $f^L(\eta,\phi)\leq0$.\\
\ \\
Step 3: $\e$-Milne problem in an infinite slab.\\
Finally, in the proof of Lemma \ref{Milne infinite LT}, by taking
$L\rt\infty$, we have
\begin{eqnarray}
f^L\rightharpoonup f\ \ in\ \ L^2_{loc}([0,L)\times[-\pi,\pi),
\end{eqnarray}
where $f$ satisfies (\ref{Milne problem}). Certainly, we have
$f(\eta,\phi)\leq0$. This justifies the claim in Step 1. Hence, we
complete the proof.
\end{proof}
\begin{remark}\label{Milne remark}
Note that when $F=0$, then all the previous proofs can be recovered
and Theorem \ref{Milne theorem 1}, Theorem \ref{Milne theorem 2} and
Theorem \ref{Milne theorem 3} still hold. Hence, we can deduce the
well-posedness, decay and maximum principle of the classical Milne
problem
\begin{eqnarray}\label{classical Milne problem}
\left\{
\begin{array}{rcl}\displaystyle
\sin\phi\frac{\p f}{\p\eta}+f-\bar f&=&S(\eta,\phi),\\
f(0,\phi)&=&h(\phi)\ \ \text{for}\ \ \sin\phi>0,\\
\lim_{\eta\rt\infty}f(\eta,\phi)&=&f_{\infty}.
\end{array}
\right.
\end{eqnarray}
\end{remark}

\section{Proof of Theorem \ref{main 1}}

We divide the proof into several steps:\\
\ \\
Step 1: Remainder definitions.\\
We may rewrite the asymptotic expansion as follows:
\begin{eqnarray}
u^{\e}&\sim&\sum_{k=0}^{\infty}\e^k\u_k+\sum_{k=0}^{\infty}\e^k\ub_k.
\end{eqnarray}
The remainder can be defined as
\begin{eqnarray}\label{pf 1}
R_N&=&u^{\e}-\sum_{k=0}^{N}\e^k\u_k-\sum_{k=0}^{N}\e^k\ub_k=u^{\e}-\q_N-\qb_N,
\end{eqnarray}
where
\begin{eqnarray}
\q_N&=&\sum_{k=0}^{N}\e^k\u_k,\\
\qb_N&=&\sum_{k=0}^{N}\e^k\ub_k.
\end{eqnarray}
Noting the equation (\ref{transport temp}) is equivalent to the
equation (\ref{transport}), we write $\ll$ to denote the neutron
transport operator as follows:
\begin{eqnarray}
\ll u&=&\e\vw\cdot\nx u+u-\bar u\\
&=&\sin\phi\frac{\p
u}{\p\eta}-\frac{\e}{1-\e\eta}\cos\phi\bigg(\frac{\p
u}{\p\phi}+\frac{\p u}{\p\theta}\bigg)+u-\bar u.\nonumber
\end{eqnarray}
\ \\
Step 2: Estimates of $\ll \q_N$.\\
The interior contribution can be estimated as
\begin{eqnarray}
\ll\q_0=\e\vw\cdot\nx \q_0+\q_0-\bar \q_0&=&\e\vw\cdot\nx
\u_0+(\u_0-\bu_0)=\e\vw\cdot\nx \u_0.
\end{eqnarray}
We have
\begin{eqnarray}
\abs{\e\vw\cdot\nx \u_0}\leq C\e\abs{\nx \u_0}\leq C\e.
\end{eqnarray}
This implies
\begin{eqnarray}
\abs{\ll \q_0}\leq C\e.
\end{eqnarray}
Similarly, for higher order term, we can estimate
\begin{eqnarray}
\ll\q_N=\e\vw\cdot\nx \q_N+\q_N-\bar \q_N&=&\e^{N+1}\vw\cdot\nx
\u_N.
\end{eqnarray}
We have
\begin{eqnarray}
\abs{\e^{N+1}\vw\cdot\nx \u_N}\leq C\e^{N+1}\abs{\nx \u_N}\leq
C\e^{N+1}.
\end{eqnarray}
This implies
\begin{eqnarray}\label{pf 2}
\abs{\ll \q_N}\leq C\e^{N+1}.
\end{eqnarray}
\ \\
Step 3: Estimates of $\ll \qb_N$.\\
The boundary layer solution is
$\ub_k=(f_k^{\e}-f_k^{\e}(\infty))\cdot\psi_0=\v_k\psi_0$ where
$f_k^{\e}(\eta,\theta,\phi)$ solves the $\e$-Milne problem and $\v_k=f_k^{\e}-f_k^{\e}(\infty)$. Notice
$\psi_0\psi=\psi_0$, so the boundary layer contribution can be
estimated as
\begin{eqnarray}\label{remainder temp 1}
\ll\qb_0&=&\sin\phi\frac{\p
\qb_0}{\p\eta}-\frac{\e}{1-\e\eta}\cos\phi\bigg(\frac{\p
\qb_0}{\p\phi}+\frac{\p \qb_0}{\p\theta}\bigg)+\qb_0-\bar
\qb_0\\
&=&\sin\phi\bigg(\psi_0\frac{\p
\v_0}{\p\eta}+\v_0\frac{\p\psi_0}{\p\eta}\bigg)-\frac{\psi_0\e}{1-\e\eta}\cos\phi\bigg(\frac{\p
\v_0}{\p\phi}+\frac{\p \v_0}{\p\theta}\bigg)+\psi_0 \v_0-\psi_0\bar\v_0\nonumber\\
&=&\sin\phi\bigg(\psi_0\frac{\p
\v_0}{\p\eta}+\v_0\frac{\p\psi_0}{\p\eta}\bigg)-\frac{\psi_0\psi\e}{1-\e\eta}\cos\phi\bigg(\frac{\p
\v_0}{\p\phi}+\frac{\p \v_0}{\p\theta}\bigg)+\psi_0 \v_0-\psi_0\bar\v_0\nonumber\\
&=&\psi_0\bigg(\sin\phi\frac{\p
\v_0}{\p\eta}-\frac{\e\psi}{1-\e\eta}\cos\phi\frac{\p
\v_0}{\p\phi}+\v_0-\bar\v_0\bigg)+\sin\phi
\frac{\p\psi_0}{\p\eta}\v_0-\frac{\psi_0\e}{1-\e\eta}\cos\phi\frac{\p
\v_0}{\p\theta}\nonumber\\
&=&\sin\phi
\frac{\p\psi_0}{\p\eta}\v_0-\frac{\psi_0\e}{1-\e\eta}\cos\phi\frac{\p
\v_0}{\p\theta}\nonumber.
\end{eqnarray}
Since $\psi_0=1$ when $\eta\leq 1/(4\e)$, the effective region
of $\px\psi_0$ is $\eta\geq1/(4\e)$ which is further and further
from the origin as $\e\rt0$. By Theorem \ref{Milne theorem 2}, the
first term in (\ref{remainder temp 1}) can be controlled as
\begin{eqnarray}
\abs{\sin\phi\frac{\p\psi_0}{\p\eta}\v_0}&\leq&
Ce^{-\frac{K_0}{\e}}\leq C\e.
\end{eqnarray}
For the second term in (\ref{remainder temp 1}), we have
\begin{eqnarray}
\abs{-\frac{\psi_0\e}{1-\e\eta}\cos\phi\frac{\p
\v_0}{\p\theta}}&\leq&C\e\abs{\frac{\p \v_0}{\p\theta}}\leq C\e.
\end{eqnarray}
This implies
\begin{eqnarray}
\abs{\ll \qb_0}\leq C\e.
\end{eqnarray}
Similarly, for higher order term, we can estimate
\begin{eqnarray}\label{remainder temp 2}
\ll\qb_N&=&\sin\phi\frac{\p
\qb_N}{\p\eta}-\frac{\e}{1-\e\eta}\cos\phi\bigg(\frac{\p
\qb_N}{\p\phi}+\frac{\p \qb_N}{\p\theta}\bigg)+\qb_N-\bar
\qb_N\\
&=&\sum_{i=0}^k\e^i\sin\phi
\frac{\p\psi_0}{\p\eta}\v_i-\frac{\psi_0\e^{k+1}}{1-\e\eta}\cos\phi\frac{\p
\v_k}{\p\theta}\nonumber.
\end{eqnarray}
Away from the origin, the first term in (\ref{remainder temp 2}) can
be controlled as
\begin{eqnarray}
\abs{\sum_{i=0}^k\e^i\sin\phi \frac{\p\psi_0}{\p\eta}\v_i}&\leq&
Ce^{-\frac{K_0}{\e}}\leq C\e^{k+1}.
\end{eqnarray}
For the second term in (\ref{remainder temp 2}), we have
\begin{eqnarray}
\abs{-\frac{\psi_0\e^{k+1}}{1-\e\eta}\cos\phi\frac{\p
\v_k}{\p\theta}}&\leq&C\e^{k+1}\abs{\frac{\p \v_k}{\p\theta}}\leq
C\e^{k+1}.
\end{eqnarray}
This implies
\begin{eqnarray}\label{pf 3}
\abs{\ll \qb_N}\leq C\e^{k+1}.
\end{eqnarray}
\ \\
Step 4: Proof of (\ref{main theorem 1}).\\
In summary, since $\ll u^{\e}=0$, collecting (\ref{pf 1}), (\ref{pf 2}) and (\ref{pf 3}), we can prove
\begin{eqnarray}
\abs{\ll R_N}\leq C\e^{N+1}.
\end{eqnarray}
Consider the asymptotic expansion to $N=3$, then the remainder $R_3$
satisfies the equation
\begin{eqnarray}
\left\{
\begin{array}{rcl}
\e \vw\cdot\nabla_x R_3+R_3-\bar R_3&=&\ll R_3\ \ \text{for}\ \ \vx\in\Omega,\\
R_3(\vx_0,\vw)&=&0\ \ \text{for}\ \ \vw\cdot\vec n<0\ \ \text{and}\
\ \vx_0\in\p\Omega.
\end{array}
\right.
\end{eqnarray}
By Theorem \ref{LI estimate}, we have
\begin{eqnarray}
\im{R_3}{\Omega\times\s^1}\leq \frac{C(\Omega)}{\e^3}\im{\ll
R_3}{\Omega\times\s^1}\leq\frac{C(\Omega)}{\e^3}(C\e^4)=C(\Omega)\e.
\end{eqnarray}
Hence, we have
\begin{eqnarray}
\nm{u^{\e}-\sum_{k=0}^3\e^k\u_k-\sum_{k=0}^3\e^k\ub_k}_{L^{\infty}(\Omega\times\s^1)}=O(\e).
\end{eqnarray}
Since it is easy to see
\begin{eqnarray}
\nm{\sum_{k=1}^3\e^k\u_k+\sum_{k=1}^3\e^k\ub_k}_{L^{\infty}(\Omega\times\s^1)}=O(\e),
\end{eqnarray}
our result naturally follows. This completes the proof of (\ref{main
theorem 1}).\\
\ \\
Step 5: Basic settings to show (\ref{main theorem 2}).\\
By (\ref{classical temp 1}), the solution $\f_0$ satisfies the Milne
problem
\begin{eqnarray}
\left\{
\begin{array}{rcl}\displaystyle
\sin(\theta+\xi)\frac{\p \f_0}{\p\eta}+\f_0-\bar \f_0&=&0,\\
\f_0(0,\theta,\xi)&=&g(\theta,\xi)\ \ \text{for}\ \
\sin(\theta+\xi)>0,\\\rule{0ex}{1em}
\lim_{\eta\rt\infty}\f_0(\eta,\theta,\xi)&=&f_0(\infty,\theta).
\end{array}
\right.
\end{eqnarray}
For convenience of comparison, we make the substitution
$\phi=\theta+\xi$ to obtain
\begin{eqnarray}
\left\{
\begin{array}{rcl}\displaystyle
\sin\phi\frac{\p \f_0}{\p\eta}+\f_0-\bar \f_0&=&0,\\
\f_0(0,\theta,\phi)&=&g(\theta,\phi)\ \ \text{for}\ \
\sin\phi>0,\\\rule{0ex}{1em}
\lim_{\eta\rt\infty}\f_0(\eta,\theta,\phi)&=&f_0(\infty,\theta).
\end{array}
\right.
\end{eqnarray}
Assume (\ref{main theorem 2}) is incorrect, such that
\begin{eqnarray}
\lim_{\e\rt0}\lnm{(\uc_0+\ubc_0)-(\u_0+\ub_0)}=0.
\end{eqnarray}
Since the boundary $g(\phi)=\cos\phi$ independent of $\theta$, by
(\ref{classical temp 1}) and (\ref{expansion temp 9}), it is obvious
the limit of zeroth order boundary layer $f_0(\infty,\theta)$ and
$f_0^{\e}(\infty,\theta)$ satisfy $f_0(\infty,\theta)=C_1$ and
$f_0^{\e}(\infty,\theta)=C_2$ for some constant $C_1$ and $C_2$
independent of $\theta$. By (\ref{classical temp 2}) and
(\ref{expansion temp 8}), we can derive the interior solutions are
indeed constants $\uc_0=C_1$ and $\u_0=C_2$. Hence, we may further
derive
\begin{eqnarray}\label{compare temp 5}
\lim_{\e\rt0}\lnm{(f_0(\infty)+\ubc_0)-(f_0^{\e}(\infty)+\ub_0)}=0.
\end{eqnarray}
For $0\leq\eta\leq 1/(2\e)$, we have $\psi_0=1$, which means
$\f_0=\ubc_0+f_0(\infty)$ and $f_0^{\e}=\ub_0+f_0^{\e}(\infty)$ on
$[0,1/(2\e)]$. Define $u=f_0+2$, $U=f_0^{\e}+2$ and
$G=g+2=\cos\phi+2$, then $u(\eta,\phi)$ satisfies the equation
\begin{eqnarray}\label{compare flat equation}
\left\{
\begin{array}{rcl}\displaystyle
\sin\phi\frac{\p u}{\p\eta}+u-\bar u&=&0,\\
u(0,\phi)&=&G(\phi)\ \ \text{for}\ \ \sin\phi>0,\\\rule{0ex}{1em}
\lim_{\eta\rt\infty}u(\eta,\phi)&=&2+f_0(\infty),
\end{array}
\right.
\end{eqnarray}
and $U(\eta,\phi)$ satisfies the equation
\begin{eqnarray}\label{compare force equation}
\left\{
\begin{array}{rcl}\displaystyle
\sin\phi\frac{\p U}{\p\eta}+F(\e;\eta)\cos\phi \frac{\p
U}{\p\phi}+U-\bar
U&=&0,\\
U(0,\phi)&=&G(\phi)\ \ \text{for}\ \ \sin\phi>0,\\\rule{0ex}{1em}
\lim_{\eta\rt\infty}U(\eta,\phi)&=&2+f_0^{\e}(\infty).
\end{array}
\right.
\end{eqnarray}
Based on (\ref{compare temp 5}), we have
\begin{eqnarray}
\lim_{\e\rt0}\lnm{U(\eta,\phi)-u(\eta,\phi)}=0.
\end{eqnarray}
Then it naturally implies
\begin{eqnarray}
\lim_{\e\rt0}\lnm{\bar U(\eta)-\bar u(\eta)}=0.
\end{eqnarray}
\ \\
Step 6: Continuity of $\bar u$ and $\bar U$ at $\eta=0$.\\
For the problem (\ref{compare flat equation}), we have for any
$r_0>0$
\begin{eqnarray}
\abs{\bar u(\eta)-\bar
u(0)}&\leq&\frac{1}{2\pi}\bigg(\int_{\sin\phi\leq
r_0}\abs{u(\eta,\phi)-u(0,\phi)}\ud{\phi}+\int_{\sin\phi\geq
r_0}\abs{u(\eta,\phi)-u(0,\phi)}\ud{\phi}\bigg).
\end{eqnarray}
Since we have shown $u\in L^{\infty}([0,\infty)\times[-\pi,\pi))$,
then for any $\delta>0$, we can take $r_0$ sufficiently small such
that
\begin{eqnarray}
\frac{1}{2\pi}\int_{\sin\phi\leq
r_0}\abs{u(\eta,\phi)-u(0,\phi)}\ud{\phi}&\leq&\frac{C}{2\pi}\arcsin
r_0\leq \frac{\delta}{2}.
\end{eqnarray}
For fixed $r_0$ satisfying above requirement, we estimate the
integral on $\sin\phi\geq r_0$. By Ukai's trace theorem, $u(0,\phi)$
is well-defined in the domain $\sin\phi\geq r_0$ and is continuous.
Also, by consider the relation
\begin{eqnarray}
\frac{\p u}{\p\eta}(0,\phi)=\frac{\bar u(0)-u(0,\phi)}{\sin\phi},
\end{eqnarray}
we can obtain in this domain $\px u$ is bounded, which further
implies $u(\eta,\phi)$ is uniformly continuous at $\eta=0$. Then
there exists $\delta_0>0$ sufficiently small, such that for any
$0\leq\eta\leq\delta_0$, we have
\begin{eqnarray}
\frac{1}{2\pi}\int_{\sin\phi\geq
r_0}\abs{u(\eta,\phi)-u(0,\phi)}\ud{\phi}&\leq&\frac{1}{2\pi}\int_{\sin\phi\geq
r_0}\frac{\delta}{2}\ud{\phi}\leq\frac{\delta}{2}.
\end{eqnarray}
In summary, we have shown for any $\delta>0$, there exists
$\delta_0>0$ such that for any $0\leq\eta\leq\delta_0$,
\begin{eqnarray}
\abs{\bar u(\eta)-\bar
u(0)}\leq\frac{\delta}{2}+\frac{\delta}{2}=\delta.
\end{eqnarray}
Hence, $\bar u(\eta)$ is continuous at $\eta=0$. By a similar
argument along the characteristics, we can show $\bar U(\eta,\phi)$
is also continuous at $\eta=0$.

In the following, by the continuity, we assume for arbitrary
$\delta>0$, there exists a $\delta_0>0$ such that for any
$0\leq\eta\leq\delta_0$, we have
\begin{eqnarray}
\abs{\bar u(\eta)-\bar u(0)}&\leq&\delta\label{compare temp 1},\\
\abs{\bar U(\eta)-\bar U(0)}&\leq&\delta\label{compare temp 2}.
\end{eqnarray}
\ \\
Step 7 Milne formulation.\\
We consider the solution at a specific point $(\eta,\phi)=(n\e,\e)$ for some
fixed $n>0$. The solution along the characteristics can be rewritten
as follows:
\begin{eqnarray}\label{compare temp 3}
u(n\e,\e)=G(\e)e^{-\frac{1}{\sin\e}n\e}
+\int_0^{n\e}e^{-\frac{1}{\sin\e}(n\e-\k)}\frac{1}{\sin\e}\bar
u(\k)\ud{\k},
\end{eqnarray}
\begin{eqnarray}\label{compare temp 4}
U(n\e,\e)=G(\e_0)e^{-\int_0^{n\e}\frac{1}{\sin\phi(\zeta)}\ud{\zeta}}
+\int_0^{n\e}e^{-\int_{\k}^{n\e}\frac{1}{\sin\phi(\zeta)}\ud{\zeta}}\frac{1}{\sin\phi(\k)}\bar
U(\k)\ud{\k},
\end{eqnarray}
where we have the conserved energy along the characteristics
\begin{eqnarray}
E(\eta,\phi)=\cos\phi e^{-V(\eta)},
\end{eqnarray}
in which $(0,\e_0)$ and $(\zeta,\phi(\zeta))$ are in the same
characteristics of $(n\e,\e)$.\\
\ \\
Step 8: Estimates of (\ref{compare temp 3}).\\
We turn to the Milne problem for $u$. We have the natural estimate
\begin{eqnarray}
\int_0^{n\e}e^{-\frac{1}{\sin\e}(n\e-\k)}\frac{1}{\sin\e}\ud{\k}&=&\int_0^{n\e}e^{-\frac{1}{\e}(n\e-\k)}\frac{1}{\e}\ud{\k}+o(\e)\\
&=&e^{-n}\int_0^{n\e}e^{\frac{\k}{\e}}\frac{1}{\e}\ud{\k}+o(\e)\nonumber\\
&=&e^{-n}\int_0^ne^{\zeta}\ud{\zeta}+o(\e)\nonumber\\
&=&(1-e^{-n})+o(\e)\nonumber.
\end{eqnarray}
Then for $0<\e\leq\delta_0$, we have $\abs{\bar u(0)-\bar
u(\k)}\leq\delta$, which implies
\begin{eqnarray}
\int_0^{n\e}e^{-\frac{1}{\sin\e}(n\e-\k)}\frac{1}{\sin\e}\bar
u(\k)\ud{\k}&=&
\int_0^{n\e}e^{-\frac{1}{\sin\e}(n\e-\k)}\frac{1}{\sin\e}\bar u(0)\ud{\k}+O(\delta)\\
&=&(1-e^{-n})\bar u(0)+o(\e)+O(\delta)\nonumber.
\end{eqnarray}
For the boundary data term, it is easy to see
\begin{eqnarray}
G(\e)e^{-\frac{1}{\sin\e}n\e}&=&e^{-n}G(\e)+o(\e)
\end{eqnarray}
In summary, we have
\begin{eqnarray}
u(n\e,\e)=(1-e^{-n})\bar u(0)+e^{-n}G(\e)+o(\e)+O(\delta).
\end{eqnarray}
\ \\
Step 9: Estimates of (\ref{compare temp 4}).\\
We consider the $\e$-Milne problem for $U$. For $\e<<1$ sufficiently
small, $\psi(\e)=1$. Then we may estimate
\begin{eqnarray}
\cos\phi(\zeta)e^{-V(\zeta)}=\cos\e e^{-V(n\e)},
\end{eqnarray}
which implies
\begin{eqnarray}
\cos\phi(\zeta)=\frac{1-n\e^2}{1-\e\zeta}\cos\e.
\end{eqnarray}
and hence
\begin{eqnarray}
\sin\phi(\zeta)=\sqrt{1-\cos^2\phi(\zeta)}=\sqrt{\frac{\e(n\e-\zeta)(2-\e\zeta-n\e^2)}{(1-\e\zeta)^2}\cos^2\e+\sin^2\e}.
\end{eqnarray}
For $\zeta\in[0,\e]$ and $n\e$ sufficiently small, by Taylor's
expansion, we have
\begin{eqnarray}
1-\e\zeta&=&1+o(\e),\\
2-\e\zeta-n\e^2&=&2+o(\e),\\
\sin^2\e&=&\e^2+o(\e^3),\\
\cos^2\e&=&1-\e^2+o(\e^3).
\end{eqnarray}
Hence, we have
\begin{eqnarray}
\sin\phi(\zeta)=\sqrt{\e(\e+2n\e-2\zeta)}+o(\e^2).
\end{eqnarray}
Since $\sqrt{\e(\e+2n\e-2\zeta)}=O(\e)$, we can further estimate
\begin{eqnarray}
\frac{1}{\sin\phi(\zeta)}&=&\frac{1}{\sqrt{\e(\e+2n\e-2\zeta)}}+o(1)\\
-\int_{\k}^{n\e}\frac{1}{\sin\phi(\zeta)}\ud{\zeta}&=&\sqrt{\frac{\e+2n\e-2\zeta}{\e}}\bigg|_{\k}^{n\e}+o(\e)
=1-\sqrt{\frac{\e+2n\e-2\k}{\e}}+o(\e).
\end{eqnarray}
Then we can easily derive the integral estimate
\begin{eqnarray}
\int_0^{n\e}e^{-\int_{\k}^{n\e}\frac{1}{\sin\phi(\zeta)}\ud{\zeta}}\frac{1}{\sin\phi(\k)}\ud{\k}&=&
e^1\int_0^{n\e}e^{-\sqrt{\frac{\e+2n\e-2\k}{\e}}}\frac{1}{\sqrt{\e(\e+2n\e-2\k)}}\ud{\k}+o(\e)\\
&=&\half e^1\int_{\e}^{(1+2n)\e}e^{-\sqrt{\frac{\sigma}{\e}}}\frac{1}{\sqrt{\e\sigma}}\ud{\sigma}+o(\e)\nonumber\\
&=&\half e^1\int_{1}^{1+2n}e^{-\sqrt{\rho}}\frac{1}{\sqrt{\rho}}\ud{\rho}+o(\e)\nonumber\\
&=&e^1\int_{1}^{\sqrt{{1+2n}}}e^{-t}\ud{t}+o(\e)\nonumber\\
&=&(1-e^{1-\sqrt{1+2n}})+o(\e)\nonumber.
\end{eqnarray}
Then for $0<\e\leq\delta_0$, we have $\abs{\bar U(0)-\bar
U(\k)}\leq\delta$, which implies
\begin{eqnarray}
\int_0^{n\e}e^{-\int_{\k}^{n\e}\frac{1}{\sin\phi(\zeta)}\ud{\zeta}}\frac{1}{\sin\phi(\k)}\bar
U(\k)\ud{\k}&=&
\int_0^{n\e}e^{-\int_{\k}^{n\e}\frac{1}{\sin\phi(\zeta)}\ud{\zeta}}\frac{1}{\sin\phi(\k)}\bar U(0)\ud{\k}+O(\delta)\\
&=&(1-e^{1-\sqrt{1+2n}})\bar U(0)+o(\e)+O(\delta)\nonumber.
\end{eqnarray}
For the boundary data term, since $G(\phi)$ is $C^1$, a similar
argument shows
\begin{eqnarray}
G(\e_0)e^{-\int_0^{n\e}\frac{1}{\sin\phi(\zeta)}\ud{\zeta}}&=&e^{1-\sqrt{1+2n}}G(\sqrt{1+2n}\e)+o(\e).
\end{eqnarray}
Therefore, we have
\begin{eqnarray}
U(n\e,\e)=(1-e^{1-\sqrt{1+2n}})\bar
U(0)+e^{1-\sqrt{1+2n}}G(\sqrt{1+2n}\e)+o(\e)+O(\delta).
\end{eqnarray}
\ \\
Step 10: Proof of (\ref{main theorem 2}).\\
In summary, we have the estimate
\begin{eqnarray}
u(n\e,\e)&=&(1-e^{-n})\bar u(0)+e^{-n}G(\e)+o(\e)+O(\delta),\\
U(n\e,\e)&=&(1-e^{1-\sqrt{1+2n}})\bar
U(0)+e^{1-\sqrt{1+2n}}G(\sqrt{1+2n}\e)+o(\e)+O(\delta).
\end{eqnarray}
The boundary data is $G=\cos\phi+2$. Then by the maximum principle
in Theorem \ref{Milne theorem 3}, we can achieve $1\leq
u(0,\phi)\leq3$ and $1\leq U(0,\phi)\leq3$. Since
\begin{eqnarray}
\bar u(0)&=&\frac{1}{2\pi}\int_{-\pi}^{\pi}u(0,\phi)\ud{\phi}
=\frac{1}{2\pi}\int_{\sin\phi>0}u(0,\phi)\ud{\phi}+\frac{1}{2\pi}\int_{\sin\phi<0}u(0,\phi)\ud{\phi}\\
&=&\frac{1}{2\pi}\int_{\sin\phi>0}\cos\phi\ud{\phi}+\frac{1}{2\pi}\int_{\sin\phi<0}u(0,\phi)\ud{\phi}\nonumber\\
&=&\frac{1}{2\pi}\int_{\sin\phi<0}M\ud{\phi}+\frac{1}{2\pi}\int_{\sin\phi<0}u(0,\phi)\ud{\phi}\nonumber,
\end{eqnarray}
we naturally obtain $3/2\leq\bar u(0)\leq 5/2$. Similarly, we can
obtain $3/2\leq\bar U(0)\leq 5/2$. Furthermore, for $\e$
sufficiently small, we have
\begin{eqnarray}
G(\sqrt{1+2n}\e)&=&3+o(\e),\\
G(\e)&=&3+o(\e).
\end{eqnarray}
Hence, we can obtain
\begin{eqnarray}
u(n\e,\e)&=&\bar u(0)+e^{-n}(-\bar u(0)+3)+o(\e)+O(\delta),\\
U(n\e,\e)&=&\bar U(0)+e^{1-\sqrt{1+2n}}(-\bar
U(0)+3)+o(\e)+O(\delta).
\end{eqnarray}
Then we can see $\lim_{\e\rt0}\lnm{\bar U(0)-\bar u(0)}=0$ naturally
leads to $\lim_{\e\rt0}\lnm{(-\bar u(0)+3)-(-\bar U(0)+3)}=0$. Also,
we have $-\bar u(0)+3=O(1)$ and $-\bar U(0)+3=O(1)$. Due to the
smallness of $\e$ and $\delta$, and also $e^{-n}\neq e^{1-\sqrt{1+2n}}$, we can obtain
\begin{eqnarray}
\abs{U(n\e,\e)-u(n\e,\e)}=O(1).
\end{eqnarray}
However, above result contradicts our assumption that
$\lim_{\e\rt0}\lnm{U(\eta,\phi)-u(\eta,\phi)}=0$ for any
$(\eta,\phi)$. This completes the proof of (\ref{main theorem 2}).

\section{Neutron Transport Equation with Diffusive Boundary}

\subsection{Problem Settings}

We consider the steady neutron transport equation in a
two-dimensional unit disk with diffusive boundary. In the space
domain $\Omega=\{\vx:\ \abs{\vx}\leq 1\}$ and the velocity domain
$\Sigma=\{\vw:\ \vw\in\s^1\}$, the neutron density $u^{\e}(\vx,\vw)$
satisfies
\begin{eqnarray}
\left\{
\begin{array}{rcl}\displaystyle
\e \vw\cdot\nabla_x u^{\e}+(1+\e^2)u^{\e}-\bar u^{\e}&=&0\label{transport.}\ \ \text{in}\ \ \Omega,\\
u^{\e}(\vx_0,\vw)&=&\pp u^{\e}(\vx_0)+\e g(\vx_0,\vw)\ \ \text{for}\
\ \vw\cdot\vec n<0\ \ \text{and}\ \ \vx_0\in\p\Omega,
\end{array}
\right.
\end{eqnarray}
where
\begin{eqnarray}\label{average 2}
\bar u^{\e}(\vx)=\frac{1}{2\pi}\int_{\s^1}u^{\e}(\vx,\vw)\ud{\vw},
\end{eqnarray}
\begin{eqnarray}\label{diffusive}
\pp u^{\e}(\vx_0)=\frac{1}{2}\int_{\vw\cdot\vec
n>0}u^{\e}(\vx_0,\vw)(\vw\cdot\vec n)\ud{\vw},
\end{eqnarray}
with the Knudsen number $0<\e<<1$.

\subsection{Interior Expansion}

We define the interior expansion as follows:
\begin{eqnarray}\label{interior expansion.}
\uc(\vx,\vw)\sim\sum_{k=0}^{\infty}\e^k\uc_k(\vx,\vw),
\end{eqnarray}
where $\uc_k$ can be defined by comparing the order of $\e$ by
plugging (\ref{interior expansion.}) into the equation
(\ref{transport.}). Thus we have
\begin{eqnarray}
\uc_0-\buc_0&=&0,\label{expansion temp 1.}\\
\uc_1-\buc_1&=&-\vw\cdot\nx\uc_0,\label{expansion temp 2.}\\
\uc_2-\buc_2&=&-\vw\cdot\nx\uc_1-\uc_0,\label{expansion temp 3.}\\
\ldots\nonumber\\
\uc_k-\buc_k&=&-\vw\cdot\nx\uc_{k-1}-\uc_{k-2}.
\end{eqnarray}
\ \\
The following analysis reveals the equation satisfied by
$\uc_k$:\\
Plugging (\ref{expansion temp 1.}) into (\ref{expansion temp 2.}),
we obtain
\begin{eqnarray}
\uc_1=\buc_1-\vw\cdot\nx\buc_0.\label{expansion temp 4.}
\end{eqnarray}
Plugging (\ref{expansion temp 4.}) into (\ref{expansion temp 3.}),
we get
\begin{eqnarray}\label{expansion temp 13.}
\uc_2-\buc_2=-\vw\cdot\nx(\buc_1-\vw\cdot\nx\buc_0)=-\vw\cdot\nx\buc_1+\vw^2\Delta_x\buc_0+2w_xw_y\p_{xy}\buc_0-\uc_0.
\end{eqnarray}
Integrating (\ref{expansion temp 13.}) over $\vw\in\s^1$, we achieve
the final form
\begin{eqnarray}
\Delta_x\buc_0-\buc_0=0.
\end{eqnarray}
which further implies $\uc_0(\vx,\vw)$ satisfies the equation
\begin{eqnarray}\label{interior 1.}
\left\{ \begin{array}{rcl} \uc_0&=&\buc_0,\\
\Delta_x\buc_0-\buc_0&=&0.
\end{array}
\right.
\end{eqnarray}
In a similar fashion, $\uc_1$ satisfies
\begin{eqnarray}\label{interior 2.}
\left\{ \begin{array}{rcl} \uc_1&=&\buc_k-\vw\cdot\nx\uc_0,\\
\Delta_x\buc_1-\buc_1&=&\displaystyle-\int_{\s^1}\vw\cdot\nx\uc_{0}\ud{\vw}.\end{array}
\right.
\end{eqnarray}
Also, $\uc_k(\vx,\vw)$ for $k\geq2$ satisfies
\begin{eqnarray}\label{interior 3.}
\left\{ \begin{array}{rcl}
\uc_k&=&\buc_k-\vw\cdot\nx\uc_{k-1}-\uc_{k-2},\\
\Delta_x\buc_k-\buc_k&=&\displaystyle-\int_{\s^1}\vw\cdot\nx\uc_{k-1}\ud{\vw}-\int_{\s^1}\uc_{k-2}\ud{\vw}.
\end{array}
\right.
\end{eqnarray}

\subsection{Milne Expansion}

In order to determine the boundary condition fir $\uc_k$, it is
well-known that we need to define the boundary layer expansion. We
still perform the substitutions (\ref{substitution 1}),
(\ref{substitution 2}) and (\ref{substitution 3}) to the equation
(\ref{transport.}) to obtain the form
\begin{eqnarray}\label{classical temp.}
\left\{\begin{array}{l}\displaystyle \sin(\theta+\xi)\frac{\p
u^{\e}}{\p\eta}-\frac{\e}{1-\e\eta}\cos(\theta+\xi)\frac{\p
u^{\e}}{\p\theta}+(1+\e^2)u^{\e}-\frac{1}{2\pi}\int_{-\pi}^{\pi}u^{\e}\ud{\xi}=0,\\\rule{0ex}{1.0em}
u^{\e}(0,\theta,\xi)=\pp u^{\e}(0,\theta)+\e g(\theta,\xi)\ \
\text{for}\ \ \sin(\theta+\xi)>0,
\end{array}
\right.
\end{eqnarray}
where
\begin{eqnarray}
\pp
u^{\e}(0,\theta)=-\half\int_{\sin(\theta+\xi)<0}u^{\e}(0,\theta,\xi)\sin(\theta+\xi)\ud{\xi}.
\end{eqnarray}
\ \\
Define the Milne expansion of boundary layer as follows:
\begin{eqnarray}\label{classical expansion.}
\ubc(\eta,\theta,\phi)\sim\sum_{k=0}^{\infty}\e^k\ubc_k(\eta,\theta,\phi)
\end{eqnarray}
where $\ubc_k$ is defined by comparing the order of $\e$ via
plugging (\ref{classical expansion.}) into the equation
(\ref{classical temp.}). Thus, in a neighborhood of the boundary, we
have
\begin{eqnarray}
\sin(\theta+\xi)\frac{\p
\ubc_0}{\p\eta}+\ubc_0-\bubc_0&=&0,\label{cexpansion temp 5.}\\
\sin(\theta+\xi)\frac{\p
\ubc_1}{\p\eta}+\ubc_1-\bubc_1&=&\frac{1}{1-\e\eta}\cos(\theta+\xi)\frac{\p
\ubc_0}{\p\theta},\label{cexpansion temp 6.}\\
\sin(\theta+\xi)\frac{\p
\ubc_2}{\p\eta}+\ubc_2-\bubc_2&=&\frac{1}{1-\e\eta}\cos(\theta+\xi)\frac{\p
\ubc_1}{\p\theta}-\ubc_0,\\
\ldots\nonumber\\
\sin(\theta+\xi)\frac{\p
\ubc_k}{\p\eta}+\ubc_k-\bubc_k&=&\frac{1}{1-\e\eta}\cos(\theta+\xi)\frac{\p
\ubc_{k-1}}{\p\theta}-\ubc_{k-2},
\end{eqnarray}
where
\begin{eqnarray}
\bar
\ubc_k(\eta,\theta)=\frac{1}{2\pi}\int_{-\pi}^{\pi}\ubc_k(\eta,\theta,\xi)\ud{\xi}.
\end{eqnarray}
The construction of $\uc_k$ and $\ubc_k$ in \cite{book1} can be
summarized as follows: \\
\ \\
Step 1: Construction of $\ubc_0$.\\
Define $\psi$ and $\psi_0$ as (\ref{cut-off 1}) and (\ref{cut-off
2}). Then the zeroth order boundary layer solution can be defined as
\begin{eqnarray}\label{classical temp 1.}
\left\{
\begin{array}{rcl}
\ubc_0&=&\psi_0(\e\eta)\bigg(\f_0(\eta,\theta,\xi)-f_0(\infty,\theta)\bigg),\\
\sin(\theta+\xi)\dfrac{\p \f_0}{\p\eta}+\f_0-\bar \f_0&=&0,\\
\f_0(0,\theta,\xi)&=&\pp \f_0(0,\theta) \ \ \text{for}\ \
\sin(\theta+\xi)>0,\\\rule{0ex}{1em}
\lim_{\eta\rt\infty}\f_0(\eta,\theta,\xi)&=&f_0(\infty,\theta),
\end{array}
\right.
\end{eqnarray}
with
\begin{eqnarray}
\pp \f_0(0,\theta)=0.
\end{eqnarray}
It is easy to see $\ubc_0=\f_0=0$.\\
\ \\
Step 2: Construction of $\ubc_1$ and $\uc_0$.\\
Define the first order boundary layer solution as
\begin{eqnarray}\label{classical temp 2.}
\left\{
\begin{array}{rcl}
\ubc_1&=&\psi_0(\e\eta)\bigg(\f_1(\eta,\theta,\xi)-f_1(\infty,\theta)\bigg),\\
\sin(\theta+\xi)\dfrac{\p \f_1}{\p\eta}+\f_1-\bar
\f_1&=&\cos(\theta+\xi)\dfrac{\psi(\e\eta)}{1-\e\eta}\dfrac{\p
\ubc_0}{\p\theta},\\\rule{0ex}{1em} \f_1(0,\theta,\xi)&=&\pp
\f_1(0,\theta)+g_1(\theta,\xi) \ \ \text{for}\ \
\sin(\theta+\xi)>0,\\\rule{0ex}{1em}
\lim_{\eta\rt\infty}\f_1(\eta,\theta,\xi)&=&f_1(\infty,\theta),
\end{array}
\right.
\end{eqnarray}
with
\begin{eqnarray}
\pp \f_1(0,\theta)=0,
\end{eqnarray}
where $g_1=\vw\cdot\nx \uc_0-\pp(\vw\cdot\nx \uc_0)+g$ in which $(\vx_0,\vw)$ and $(0,\theta,\xi)$ denote the same point. Since
$\ubc_0=0$, we can obtain a unique $\f_1(\eta,\theta,\xi)\in
L^{\infty}([0,\infty)\times[-\pi,\pi)\times[-\pi,\pi))$. Hence,
$\ubc_1$ is well-defined. By the compatibility condition (\ref{Milne
compatibility condition}), we can define the zeroth order interior
solution as
\begin{eqnarray}\label{classical temp 3.}
\left\{
\begin{array}{rcl}
\uc_0&=&\bar \uc_0 ,\\\rule{0ex}{1em} \Delta_x \buc_0-\buc_0&=&0\ \
\text{in}\ \ \Omega,\\\rule{0ex}{1em}\dfrac{\p \uc_0}{\p\vec
n}&=&\dfrac{1}{\pi}\displaystyle
\int_{\sin(\theta+\xi)>0}g(\theta,\xi)\sin(\theta+\xi)\ud{\xi}\ \ \text{on}\ \
\p\Omega.
\end{array}
\right.
\end{eqnarray}
\ \\
Step 3: Construction of $\ubc_2$ and $\uc_1$.\\
Define the second order boundary layer solution as
\begin{eqnarray}\label{classical temp 4.}
\left\{
\begin{array}{rcl}
\ubc_2&=&\psi_0(\e\eta)\bigg(\f_2(\eta,\theta,\xi)-f_2(\infty,\theta)\bigg),\\
\sin(\theta+\xi)\dfrac{\p \f_2}{\p\eta}+\f_2-\bar
\f_2&=&\cos(\theta+\xi)\dfrac{\psi(\e\eta)}{1-\e\eta}\dfrac{\p
\ubc_1}{\p\theta}
-\psi(\e\eta)\ubc_0(\eta,\theta,\xi),\\\rule{0ex}{1em}
\f_2(0,\theta,\xi)&=&\pp \f_2(0,\theta)+g_2(\theta,\xi) \ \
\text{for}\ \ \sin(\theta+\xi)>0,\\\rule{0ex}{1em}
\lim_{\eta\rt\infty}\f_2(\eta,\theta,\xi)&=&f_2(\infty,\theta),
\end{array}
\right.
\end{eqnarray}
with
\begin{eqnarray}
\pp \f_2(0,\theta)=0.
\end{eqnarray}
where $g_2=\vw\cdot\nx \uc_1-\pp(\vw\cdot\nx\uc_1)+\uc_0-\pp\uc_0$.
By the compatibility condition (\ref{Milne compatibility
condition}), we define the first order interior solution as
\begin{eqnarray}\label{expansion temp 22.}
\left\{
\begin{array}{rcl}
\uc_1&=&\buc_1-\vw\cdot\nx\uc_0,\\\rule{0ex}{1em}
\Delta_x\buc_1-\buc_1&=&-\displaystyle\int_{\s^1}(\vw\cdot\nx\uc_{0})\ud{\vw}\
\ \text{in}\ \ \Omega,\\\rule{0ex}{2.0em} \dfrac{\p\uc_1}{\p\vec
n}&=&\dfrac{1}{\pi}\displaystyle\int_0^{\infty}\int_{-\pi}^{\pi}\dfrac{\psi(\e
s)}{1-\e s}\cos(\theta+\xi)\frac{\p
\ubc_1}{\p\theta}(s,\theta,\xi)\ud{\xi}\ud{s}\ \ \text{on}\ \
\p\Omega.
\end{array}
\right.
\end{eqnarray}
\ \\
Step 4: Generalization to arbitrary $k$.\\
Define the $k^{th}$ order boundary layer solution as
\begin{eqnarray}
\left\{
\begin{array}{rcl}
\ubc_k&=&\psi_0(\e\eta)\bigg(\f_k(\eta,\theta,\xi)-f_k(\infty,\theta)\bigg),\\
\sin(\theta+\xi)\dfrac{\p \f_k}{\p\eta}+\f_k-\bar
\f_k&=&\cos(\theta+\xi)\dfrac{\psi(\e\eta)}{1-\e\eta}\dfrac{\p
\ubc_{k-1}}{\p\theta}
-\psi(\e\eta)\ubc_{k-2}(\eta,\theta,\xi),\\\rule{0ex}{1em}
\f_k(0,\theta,\xi)&=&\pp \f_k(0,\theta)+g_k(\theta,\xi) \ \
\text{for}\ \ \sin(\theta+\xi)>0,\\\rule{0ex}{1em}
\lim_{\eta\rt\infty}\f_k(\eta,\theta,\xi)&=&f_k(\infty,\theta),
\end{array}
\right.
\end{eqnarray}
with
\begin{eqnarray}
\pp \f_k(0,\theta)=0.
\end{eqnarray}
where $g_k=\vw\cdot\nx
\uc_{k-1}-\pp(\vw\cdot\nx\uc_{k-1})+\uc_{k-2}-\pp\uc_{k-2}$. By the
compatibility condition (\ref{Milne compatibility condition}), we
can define the $(k-1)^{th}$ order interior solution as
\begin{eqnarray}
\\
\left\{
\begin{array}{rcl}
\uc_{k-1}&=&\buc_{k-1}-\vw\cdot\nx\uc_{k-2}-\uc_{k-3},\\\rule{0ex}{2em}
\Delta_x\buc_{k-1}-\buc_{k-1}&=&-\displaystyle\int_{\s^1}(\vw\cdot\nx\uc_{k-2})\ud{\vw}-\displaystyle\int_{\s^1}\uc_{k-3}\ud{\vw}\
\ \text{in}\ \ \Omega,\\\rule{0ex}{2.0em}
\dfrac{\p\buc_{k-1}}{\p\vec
n}&=&\dfrac{1}{\pi}\displaystyle\int_{-\pi}^{\pi}\bigg(\vw\cdot\nx(\vw\cdot\uc_{k-2}+\uc_{k-3})\bigg)\sin\phi\ud{\phi}\\
&&+\dfrac{1}{\pi}\displaystyle\int_0^{\infty}\int_{-\pi}^{\pi}\bigg(\dfrac{\psi(\e
s)}{1-\e s}\cos(\theta+\xi)\frac{\p \ubc_{k-1}}{\p\theta}-\psi(\e
s)\ubc_{k-2}\bigg)(s,\theta,\xi)\ud{\xi}\ud{s}\ \ \text{on}\ \
\p\Omega.\nonumber
\end{array}
\right.
\end{eqnarray}
In \cite[pp.143]{book1}, the author proved the expansion can be
applied to the second order of $\e$. Based on Remark \ref{Milne
remark.} to the Milne problem, in order to show the existence of a
solution $\f_2(\eta,\theta,\xi)\in
L^{\infty}([0,\infty)\times[-\pi,\pi)\times[-\pi,\pi))$, we at least
require the source term
\begin{eqnarray}
\cos(\theta+\xi)\frac{\psi}{1-\e\eta}\frac{\p \ubc_1}{\p\theta}\in
L^{\infty}([0,\infty)\times[-\pi,\pi)\times[-\pi,\pi)),
\end{eqnarray}
since $\ubc_0=0$. We thus need to require
\begin{eqnarray}
\frac{\p (\f_1-f_1(\infty,\theta))}{\p\theta}\in
L^{\infty}([0,\infty)\times[-\pi,\pi)\times[-\pi,\pi)).
\end{eqnarray}
Since $Z=\ps (\f_1-f_1(\infty))$ satisfies the equation
\begin{eqnarray}
\left\{
\begin{array}{rcl}
\sin(\theta+\xi)\dfrac{\p Z}{\p\eta}+Z-\bar Z&=&-\cos(\theta+\xi)\dfrac{\p \f_1}{\p\eta},\\
Z(0,\theta,\xi)&=&\pp Z(0,\theta)+\dfrac{\p
g_1}{\p\theta}(\theta,\xi)\ \ \text{for}\ \ \sin(\theta+\xi)>0,
\end{array}
\right.
\end{eqnarray}
in order for $Z\in
L^{\infty}([0,\infty)\times[-\pi,\pi)\times[-\pi,\pi))$, assuming
the boundary data $\ps g_1\in L^{\infty}(\Gamma^-)$, we require the
source term
\begin{eqnarray}
-\cos(\theta+\xi)\frac{\p \f_1}{\p\eta}\in
L^{\infty}([0,\infty)\times[-\pi,\pi)\times[-\pi,\pi)).
\end{eqnarray}
On the other hand, by Lemma \ref{counter theorem 1.}, we can show
for specific $g$, it holds that $\px\f_1\notin
L^{\infty}([0,\infty)\times[-\pi,\pi)\times[-\pi,\pi))$. Due to the
intrinsic singularity in the solution to (\ref{classical temp 2.}),
the construction in \cite{book1} breaks down.

\subsection{$\e$-Milne Expansion with Geometric Correction}

In order to overcome the difficulty in estimating
\begin{eqnarray}
\cos(\theta+\xi)\frac{\psi}{1-\e\eta}\frac{\p \ubc_k}{\p\theta},
\end{eqnarray}
we introduce one more substitution (\ref{substitution 4}) to
decompose this term and transform the equation (\ref{transport.})
into
\begin{eqnarray}\label{transport temp.}
\left\{\begin{array}{l}\displaystyle \sin\phi\frac{\p
u^{\e}}{\p\eta}-\frac{\e}{1-\e\eta}\cos\phi\bigg(\frac{\p
u^{\e}}{\p\phi}+\frac{\p
u^{\e}}{\p\theta}\bigg)+(1+\e^2)u^{\e}-\frac{1}{2\pi}\int_{-\pi}^{\pi}u^{\e}\ud{\phi}=0,\\\rule{0ex}{1.0em}
u^{\e}(0,\theta,\phi)=\pp u^{\e}(0,\theta)+\e g(\theta,\phi)\ \
\text{for}\ \ \sin\phi>0,
\end{array}
\right.
\end{eqnarray}
with
\begin{eqnarray}
\pp
u^{\e}(0,\theta)=-\half\int_{\sin\phi<0}u^{\e}(0,\theta,\xi)\sin\phi\ud{\phi}.
\end{eqnarray}
We define the $\e$-Milne expansion of boundary layer as follows:
\begin{eqnarray}\label{boundary layer expansion.}
\ub(\eta,\theta,\phi)\sim\sum_{k=0}^{\infty}\e^k\ub_k(\eta,\theta,\phi)
\end{eqnarray}
where $\ub_k$ can be defined by comparing the order of $\e$ via
plugging (\ref{boundary layer expansion.}) into the equation
(\ref{transport temp.}). Thus, in a neighborhood of the boundary, we
have
\begin{eqnarray}
\sin\phi\frac{\p \ub_0}{\p\eta}+\frac{\e}{1-\e\eta}\cos\phi\frac{\p
\ub_0}{\p\phi}+\ub_0-\bub_0&=&0\label{expansion temp 5.}\\
\sin\phi\frac{\p \ub_1}{\p\eta}+\frac{\e}{1-\e\eta}\cos\phi\frac{\p
\ub_1}{\p\phi}+\ub_1-\bub_1&=&\frac{1}{1-\e\eta}\cos\phi\frac{\p
\ub_0}{\p\theta}\label{expansion temp 6.}\\
\sin\phi\frac{\p \ub_2}{\p\eta}+\frac{\e}{1-\e\eta}\cos\phi\frac{\p
\ub_2}{\p\phi}+\ub_2-\bub_2&=&\frac{1}{1-\e\eta}\cos\phi\frac{\p
\ub_1}{\p\theta}-\ub_0\label{expansion temp 7.}\\
\ldots\nonumber\\
\sin\phi\frac{\p \ub_k}{\p\eta}+\frac{\e}{1-\e\eta}\cos\phi\frac{\p
\ub_k}{\p\phi}+\ub_k-\bub_k&=&\frac{1}{1-\e\eta}\cos\phi\frac{\p
\ub_{k-1}}{\p\theta}-\ub_{k-2}
\end{eqnarray}
where
\begin{eqnarray}
\bub_k(\eta,\theta)=\frac{1}{2\pi}\int_{-\pi}^{\pi}\ub_k(\eta,\theta,\phi)\ud{\phi}.
\end{eqnarray}
We refer to the cut-off function $\psi$ and $\psi_0$ as
(\ref{cut-off 1}) and (\ref{cut-off 2}), and define the force as
(\ref{force}). Define the interior expansion as follows:
\begin{eqnarray}
\u(\vx,\vw)\sim\sum_{k=0}^{\infty}\e^k\u_k(\vx,\vw)
\end{eqnarray}
where $\u_k$ satisfies the same equations as $\uc_k$ in (\ref{interior 1.}), (\ref{interior 2.}) and (\ref{interior 3.}).
Here, to highlight its dependence on $\e$ via the $\e$-Milne problem and boundary data, we add the
superscript $\e$.

The bridge between the interior solution and the boundary layer
solution is the boundary condition of (\ref{transport.}), so we
first consider the boundary condition expansion:
\begin{eqnarray}
(\u_0+\ub_0)&=&\pp (\u_0+\ub_0),\\
(\u_1+\ub_1)&=&\pp (\u_1+\ub_1)+g,\\
(\u_2+\ub_2)&=&\pp (\u_2+\ub_2),\\
\ldots\nonumber\\
(\u_k+\ub_k)&=&\pp (\u_k+\ub_k).
\end{eqnarray}
Note the fact that $\bu_k=\pp\bu_k$, we can simplify above
conditions as follows:
\begin{eqnarray}
\ub_0&=&\pp \ub_0,\\
\ub_1&=&\pp\ub_1+(\vw\cdot\u_0-\pp(\vw\cdot\u_0))+g,\\
\ub_2&=&\pp \ub_2+(\vw\cdot\u_1-\pp(\vw\cdot\u_1))+(\u_0-\pp\u_0),\\
\ldots\nonumber\\
\ub_k&=&\pp
\ub_k+(\vw\cdot\u_{k-1}-\pp(\vw\cdot\u_{k-1}))+(\u_{k-2}-\pp\u_{k-2}).
\end{eqnarray}
The construction of $\u_k$ and $\ub_k$ are as follows:\\
\ \\
Step 1: Construction of $\ub_0$.\\
Define the zeroth order boundary layer solution as
\begin{eqnarray}\label{expansion temp 9.}
\left\{
\begin{array}{rcl}
\ub_0(\eta,\theta,\phi)&=&\psi_0(\e\eta)\bigg(f_0^{\e}(\eta,\theta,\phi)-f_0^{\e}(\infty,\theta)\bigg),\\
\sin\phi\dfrac{\p f_0^{\e}}{\p\eta}+F(\e;\eta)\cos\phi\dfrac{\p
f_0^{\e}}{\p\phi}+f_0^{\e}-\bar f_0^{\e}&=&0,\\\rule{0ex}{1em}
f_0^{\e}(0,\theta,\phi)&=&\pp f_0^{\e}(0,\theta)\ \ \text{for}\ \
\sin\phi>0,\\\rule{0ex}{1em}
\lim_{\eta\rt\infty}f_0^{\e}(\eta,\theta,\phi)&=&f_0^{\e}(\infty,\theta),
\end{array}
\right.
\end{eqnarray}
with
\begin{eqnarray}
\pp f_0^{\e}(0,\theta)=0.
\end{eqnarray}
By Theorem \ref{Milne theorem 1.}, $\ub_0$ is well-defined. It is
obvious to see $f_0^{\e}=f_0^{\e}(\infty)=0$ is the only solution.\\
\ \\
Step 2: Construction of $\ub_1$ and $\u_0$.\\
Define the first order boundary layer solution as
\begin{eqnarray}\label{expansion temp 10.}
\left\{
\begin{array}{rcl}
\ub_1(\eta,\theta,\phi)&=&\psi_0(\e\eta)\bigg(f_1^{\e}(\eta,\theta,\phi)-f_1^{\e}(\infty,\theta)\bigg),\\
\sin\phi\dfrac{\p f_1^{\e}}{\p\eta}+F(\e;\eta)\cos\phi\dfrac{\p
f_1^{\e}}{\p\phi}+f_1^{\e}-\bar
f_1^{\e}&=&\dfrac{\psi(\e\eta)}{1-\e\eta}\cos\phi\dfrac{\p
\ub_0}{\p\theta},\\\rule{0ex}{1.0em} f_1^{\e}(0,\theta,\phi)&=&\pp
f_1^{\e}(0,\theta)+g_1(\theta,\phi)\ \ \text{for}\ \
\sin\phi>0,\\\rule{0ex}{1em}
\lim_{\eta\rt\infty}f_1^{\e}(\eta,\theta,\phi)&=&f_1^{\e}(\infty,\theta),
\end{array}
\right.
\end{eqnarray}
with
\begin{eqnarray}
\pp f_1^{\e}(0,\theta)=0,
\end{eqnarray}
where
\begin{eqnarray}
g_1&=&(\vw\cdot\nx\u_0(\vx_0)-\pp(\vw\cdot\nx\u_0(\vx_0)))+g,
\end{eqnarray}
with $\vx_0$ is the same boundary point as $(0,\theta)$ and
\begin{eqnarray}
\vw&=&(-\sin(\phi-\theta),-\cos(\phi-\theta)),\\
\vec n&=&(\cos\theta,\sin\theta).
\end{eqnarray}
To solve (\ref{expansion temp 10.}), we require the compatibility
condition (\ref{Milne compatibility condition}) for the boundary
data
\begin{eqnarray}
\\
\int_{\sin\phi>0}\bigg(g+\vw\cdot\nx\u_0(\vx_0)-\pp(\vw\cdot\nx\u_0(\vx_0))\bigg)\sin\phi\ud{\phi}
+\int_0^{\infty}\int_{-\pi}^{\pi}e^{-V(s)}\frac{\psi}{1-\e
s}\cos\phi\frac{\p
\ub_0}{\p\theta}(s,\theta,\phi)\ud{\phi}\ud{s}\nonumber\\
=0.&&\nonumber
\end{eqnarray}
Note the fact
\begin{eqnarray}
&&\int_{\sin\phi>0}\bigg(\vw\cdot\nx\u_0(\vx_0)-\pp(\vw\cdot\nx\u_0(\vx_0))\bigg)\sin\phi\ud{\phi}\\
&=&
\int_{\sin\phi>0}(\vw\cdot\nx\u_0(\vx_0))\sin\phi\ud{\phi}-2\pp(\vw\cdot\nx\u_0(\vx_0))\nonumber\\
&=&\int_{\sin\phi>0}(\vw\cdot\nx\u_0(\vx_0))\sin\phi\ud{\phi}+\int_{\sin\phi<0}(\vw\cdot\nx\u_0(\vx_0))\sin\phi\ud{\phi}\nonumber\\
&=&\int_{-\pi}^{\pi}(\vw\cdot\nx\u_0(\vx_0))\sin\phi\ud{\phi}\nonumber\\
&=&-\pi\nx\bu_0(\vx_0)\cdot\vec n=-\pi\frac{\p\bu_0(\vx_0)}{\p\vec
n}.\nonumber
\end{eqnarray}
We can simplify the compatibility condition as follows:
\begin{eqnarray}
\int_{\sin\phi>0}g(\phi)\sin\phi\ud{\phi}-\pi\frac{\p\bu_0(\vx_0)}{\p\vec
n} +\int_0^{\infty}\int_{-\pi}^{\pi}e^{-V(s)}\frac{\psi}{1-\e
s}\cos\phi\frac{\p \ub_0}{\p\theta}(s,\theta,\phi)\ud{\phi}\ud{s}=0.
\end{eqnarray}
Then we have
\begin{eqnarray}
\frac{\p\bu_0(\vx_0)}{\p\vec
n}&=&\frac{1}{\pi}\int_{\sin\phi>0}g(\theta,\phi)\sin\phi\ud{\phi}
+\frac{1}{\pi}\int_0^{\infty}\int_{-\pi}^{\pi}e^{-V(s)}\frac{\psi}{1-\e
s}\cos\phi\frac{\p
\ub_0}{\p\theta}(s,\theta,\phi)\ud{\phi}\ud{s}\\
&=&\frac{1}{\pi}\int_{\sin\phi>0}g(\theta,\phi)\sin\phi\ud{\phi}.\nonumber
\end{eqnarray}
Hence, we define the zeroth order interior solution $\u_0(\vx,\vw)$ as
\begin{eqnarray}\label{expansion temp 8.}
\left\{
\begin{array}{rcl}
\u_0&=&\bu_0 ,\\\rule{0ex}{1em} \Delta_x\bu_0-\bu_0&=&0\ \ \text{in}\
\ \Omega,\\\rule{0ex}{1em}\dfrac{\p\bu_0}{\p\vec
n}&=&\dfrac{1}{\pi}\displaystyle
\int_{\sin\phi>0}g(\theta,\phi)\sin\phi\ud{\phi}\ \ \text{on}\ \
\p\Omega.
\end{array}
\right.
\end{eqnarray}
\ \\
Step 3: Analysis of $\u_0$ and $\ub_1$.\\
By Theorem \ref{Milne theorem 2.}, we can easily see $f_1^{\e}$ is
well-defined in $L^{\infty}(\Omega\times\s^1)$ and approaches
$f_1^{\e}(\infty)$ exponentially fast as $\eta\rt\infty$. Since
$f_0^{\e}=0$, then if $g\in C^r(\Gamma^-)$, it is obvious to check
$\p_n\u_0\in C^r(\p\Omega)$. Hence, by the standard elliptic
estimate, there exists a unique solution $\u_0\in W^{r+1,p}(\Omega)$
for arbitrary $p\geq2$ satisfying
\begin{eqnarray}
\nm{\bu_0}_{W^{r+1,p}(\Omega)}\leq C(\Omega)\nm{\frac{\p \bu_0}{\p\vec
n}}_{W^{r-1/p,p}(\p\Omega)},
\end{eqnarray}
which implies $\nx\u_0\in W^{r,p}(\Omega)$, $\nx\u_0\in
W^{r-1/p,p}(\p\Omega)$ and $\u_0\in C^{r,1-2/p}(\Omega)$.\\
\ \\
Step 4: Construction of $\ub_2$ and $\u_1$.\\
Define the second order boundary layer solution as
\begin{eqnarray}\label{expansion temp 11.}
\left\{
\begin{array}{rcl}
\ub_2(\eta,\theta,\phi)&=&\psi(\e\eta)\bigg(f_2^{\e}(\eta,\theta,\phi)-f_2^{\e}(\infty,\theta)\bigg),\\
\sin\phi\dfrac{\p f_2^{\e}}{\p\eta}+F(\e;\eta)\cos\phi\dfrac{\p
f_2^{\e}}{\p\phi}+f_2^{\e}-\bar
f_2^{\e}&=&\dfrac{\psi(\e\eta)}{1-\e\eta}\cos\phi\dfrac{\p
\ub_1}{\p\theta}-\psi(\e\eta)\ub_0,\\\rule{0ex}{1.0em}
f_2^{\e}(0,\theta,\phi)&=&\pp f_2^{\e}(0,\theta)+g_2(\theta,\phi)\ \
\text{for}\ \ \sin\phi>0,\\\rule{0ex}{1em}
\lim_{\eta\rt\infty}f_2^{\e}(\eta,\theta,\phi)&=&f_2^{\e}(\infty,\theta),
\end{array}
\right.
\end{eqnarray}
with
\begin{eqnarray}
\pp f_2^{\e}(0,\theta)=0,
\end{eqnarray}
where
\begin{eqnarray}
g_2&=&(\vw\cdot\nx\u_1(\vx_0)-\pp(\vw\cdot\nx\u_1(\vx_0)))+\u_0(\vx_0)-\pp\u_0(\vx_0).
\end{eqnarray}
In order for equation (\ref{expansion temp 11.}) being well-posed,
we require the compatibility condition (\ref{Milne compatibility
condition}) for the boundary data
\begin{eqnarray}
\int_{\sin\phi>0}\bigg(\vw\cdot\nx\u_1(\vx_0)-\pp(\vw\cdot\nx\u_1(\vx_0)))\bigg)\sin\phi\ud{\phi}&&\\
+\int_0^{\infty}\int_{-\pi}^{\pi}e^{-V(s)}\bigg(\frac{\psi}{1-\e
s}\cos\phi\frac{\p
\ub_1}{\p\theta}(s,\theta,\phi)-\psi\ub_0(s,\theta,\phi)\bigg)\ud{\phi}\ud{s}&=&0.\nonumber
\end{eqnarray}
Similarly, we can directly verify the relation
\begin{eqnarray}
\int_{\sin\phi>0}\bigg(\vw\cdot\nx\u_1(\vx_0)-\pp(\vw\cdot\nx\u_1(\vx_0))\bigg)\sin\phi\ud{\phi}
&=&-\pi\frac{\p\bu_1(\vx_0)}{\p\vec n}.
\end{eqnarray}
We can simplify the compatibility condition as follows:
\begin{eqnarray}
-\pi\frac{\p\bu_1(\vx_0)}{\p\vec n}
+\int_0^{\infty}\int_{-\pi}^{\pi}e^{-V(s)}\bigg(\frac{\psi}{1-\e
s}\cos\phi\frac{\p
\ub_1}{\p\theta}(s,\theta,\phi)-\psi\ub_0(s,\theta,\phi)\bigg)\ud{\phi}\ud{s}&=&0.
\end{eqnarray}
Then we have
\begin{eqnarray}
-\pi\frac{\p\bu_1(\vx_0)}{\p\vec
n}&=&\frac{1}{\pi}\int_0^{\infty}\int_{-\pi}^{\pi}e^{-V(s)}\bigg(\frac{\psi}{1-\e
s}\cos\phi\frac{\p
\ub_1}{\p\theta}(s,\theta,\phi)-\psi\ub_0(s,\theta,\phi)\bigg)\ud{\phi}\ud{s}\\
&=&\frac{1}{\pi}\int_0^{\eta}\int_{-\pi}^{\pi}e^{-V(s)}\frac{\psi}{1-\e
s}\cos\phi\frac{\p
\ub_1}{\p\theta}(s,\theta,\phi)\ud{\phi}\ud{s}.\nonumber
\end{eqnarray}
Hence, we define the first order interior solution $\u_1(\vx)$ as
\begin{eqnarray}\label{expansion temp 12.}
\left\{
\begin{array}{rcl}
\u_1&=&\bu_1-\vw\cdot\nx\u_0,\\\rule{0ex}{1em}
\Delta_x\bu_1-\bu_1&=&-\displaystyle\int_{\s^1}(\vw\cdot\nx\u_{0})\ud{\vw}\
\ \text{in}\ \ \Omega,\\\rule{0ex}{1.0em} \dfrac{\p\bu_1}{\p\vec
n}&=&\dfrac{1}{\pi}\displaystyle\int_0^{\infty}\int_{-\pi}^{\pi}e^{-V(s)}\dfrac{\psi(\e
s)}{1-\e s}\cos\phi\frac{\p
\ub_1}{\p\theta}(s,\theta,\phi)\ud{\phi}\ud{s}\ \ \text{on}\ \
\p\Omega.
\end{array}
\right.
\end{eqnarray}
\ \\
Step 5: Analysis of $\u_1$ and $\ub_2$.\\
By Theorem \ref{Milne theorem 2.}, we can easily see $f_2^{\e}$ is
well-defined in $L^{\infty}(\Omega\times\s^1)$ and approaches
$f_2^{\e}(\infty)$ exponentially fast as $\eta\rt\infty$. By above
analysis, it is obvious to check $\p_n\bu_1\in C^{r-1}(\p\Omega)$.
Hence, by the standard elliptic estimate, there exists a unique
solution $\bu_1\in W^{r,p}(\Omega)$ for arbitrary $p\geq2$
satisfying
\begin{eqnarray}
\nm{\bu_1}_{W^{r,p}(\Omega)}\leq C(\Omega)\bigg(\nm{\frac{\p
\bu_1}{\p\vec
n}}_{W^{r-1-1/p,p}(\p\Omega)}+\nm{\nx\u_0}_{W^{r-2,p}(\Omega)}\bigg)
\end{eqnarray}
which implies $\nx\u_1\in W^{r-1,p}(\Omega)$, $\nx\u_1\in
W^{r-1-1/p,p}(\p\Omega)$ and $\u_1\in C^{r-1,1-2/p}(\Omega)$.\\
\ \\
Step 6: Generalization to arbitrary $k$.\\
Then this process can proceed to arbitrary $k$ as long as $g$ is
sufficiently smooth. We can always determine $\ub_k$ and $\u_{k-1}$
simultaneously based on the compatibility condition. Define the
$k^{th}$ order boundary layer solution as
\begin{eqnarray}
\left\{
\begin{array}{rcl}
\ub_k(\eta,\theta,\phi)&=&\psi(\e\eta)\bigg(f_k^{\e}(\eta,\theta,\phi)-f_k^{\e}(\infty,\theta)\bigg),\\
\sin\phi\dfrac{\p f_k^{\e}}{\p\eta}+F(\e;\eta)\cos\phi\dfrac{\p
f_k^{\e}}{\p\phi}+f_k^{\e}-\bar
f_k^{\e}&=&\dfrac{\psi}{1-\e\eta}\cos\phi\dfrac{\p
\ub_{k-1}}{\p\theta}-\psi\ub_{k-2},\\\rule{0ex}{1.0em}
f_k^{\e}(0,\theta,\phi)&=&\pp f_k^{\e}(0,\theta)+g_k(\theta,\phi)\ \
\text{for}\ \ \sin\phi>0,\\\rule{0ex}{1em}
\lim_{\eta\rt\infty}f_k^{\e}(\eta,\theta,\phi)&=&f_k^{\e}(\infty,\theta),
\end{array}
\right.
\end{eqnarray}
with
\begin{eqnarray}
\pp f_k^{\e}(0,\theta)=0,
\end{eqnarray}
where
\begin{eqnarray}
g_k&=&(\vw\cdot\nx\u_{k-1}(\vx_0)-\pp(\vw\cdot\nx\u_{k-1}(\vx_0)))+\u_{k-2}(\vx_0)-\pp\u_{k-2}(\vx_0).
\end{eqnarray}
Hence, we define the $(k-1)^{th}$ order interior solution as
\begin{eqnarray}
\\
\left\{
\begin{array}{rcl}
\u_{k-1}&=&\bu_{k-1}-\vw\cdot\nx\u_{k-2}-\u_{k-3},\\\rule{0ex}{2em}
\Delta_x\bu_{k-1}-\bu_{k-1}&=&-\displaystyle\int_{\s^1}(\vw\cdot\nx\u_{k-2})\ud{\vw}-\displaystyle\int_{\s^1}\u_{k-3}\ud{\vw}\
\ \text{in}\ \ \Omega,\\\rule{0ex}{2.0em} \dfrac{\p\bu_{k-1}}{\p\vec
n}&=&\dfrac{1}{\pi}\displaystyle\int_{-\pi}^{\pi}\bigg(\vw\cdot\nx(\vw\cdot\u_{k-2}+\u_{k-3})\bigg)\sin\phi\ud{\phi}\\
&&+
\dfrac{1}{\pi}\displaystyle\int_0^{\infty}\int_{-\pi}^{\pi}e^{-V(s)}\bigg(\dfrac{\psi(\e
s)}{1-\e s}\cos\phi\frac{\p \ub_{k-1}}{\p\theta}-\psi(\e
s)\ub_{k-2}\bigg)(s,\theta,\phi)\ud{\phi}\ud{s}\ \ \text{on}\ \
\p\Omega.\nonumber
\end{array}
\right.
\end{eqnarray}
In particular, for $g\in C^{k+1}(\Gamma^-)$, we can define the
interior solution up to $k^{th}$ order and the boundary layer
solution up to $(k+1)^{th}$ order, i.e. up to $\u_k$ and
$\ub_{k+1}$.

\subsection{Well-Posedness of Steady Neutron Transport Equation}

In this section, we consider the well-posedness of the steady
neutron transport equation
\begin{eqnarray}\label{neutron.}
\left\{
\begin{array}{rcl}\displaystyle
\epsilon\vec w\cdot\nabla_xu+(1+\e^2)u-\bar
u&=&f(\vx,\vw)\ \ \text{in}\ \ \Omega,\\
u(\vx_0,\vw)&=&\pp u(\vx_0)+\e g(\vx_0,\vw)\ \ \text{for}\ \
\vw\cdot\vec n<0\ \ \text{and}\ \ \vx_0\in\p\Omega.
\end{array}
\right.
\end{eqnarray}
We define the $L^{\infty}$ norms in $\Omega\times\s^1$ as usual:
\begin{eqnarray}
\nm{f}_{L^{\infty}(\Omega\times\s^1)}&=&\sup_{(\vx,\vw)\in\Omega\times\s^1}\abs{f(\vx,\vw)}.
\end{eqnarray}
\begin{theorem}\label{well-posedness lemma 2.}
Assume $f(\vx,\vw)\in L^{\infty}(\Omega\times\s^1)$ and
$g(x_0,\vw)\in L^{\infty}(\Gamma^-)$. Then for the transport
equation
\begin{eqnarray}
\label{well-posedness penalty equation.} \left\{
\begin{array}{rcl}\displaystyle
\epsilon\vec w\cdot\nabla_xu+(1+\e^2)u-\bar u&=&f(\vx,\vw)\ \ \text{in}\ \ \Omega,\\
u(\vec x_0,\vec w)&=&\pp u(\vx_0)+\e g(\vx_0,\vw)\ \ \text{for}\ \
\vec x_0\in\p\Omega\ \ \text{and}\ \vw\cdot\vec n<0,
\end{array}
\right.
\end{eqnarray}
there exists a unique solution $u(\vx,\vw)\in
L^{\infty}(\Omega\times\s^1)$ satisfying
\begin{eqnarray}
\im{u}{\Omega\times\s^1}\leq C\bigg(\frac{1}{\e^2}
\im{f}{\Omega\times\s^1}+\frac{1}{\e^2}\im{g}{\Gamma^-}\bigg).
\end{eqnarray}
\end{theorem}
\begin{proof}
We iteratively construct
an approximating sequence $\{u^k\}_{k=0}^{\infty}$ where $u^0=0$ and
\begin{eqnarray}\label{penalty temp 1.}
\\
\left\{
\begin{array}{rcl}\displaystyle
\epsilon\vec w\cdot\nabla_xu_{\l}^k+(1+\e^2)u^k-\bar u^{k-1}&=&f(\vx,\vw)\ \ \text{in}\ \ \Omega,\\
u^k(\vec x_0,\vec w)&=&\pp u^{k-1}(\vx_0)+\e g(\vx_0,\vw)\ \
\text{for}\ \ \vec x_0\in\p\Omega\ \ \text{and}\ \vw\cdot\vec n<0.
\end{array}\nonumber
\right.
\end{eqnarray}
By Lemma \ref{well-posedness lemma 1}, this sequence is
well-defined and $\im{u^k}{\Omega\times\s^1}<\infty$.
We rewrite equation (\ref{penalty temp 1.}) along the
characteristics as
\begin{eqnarray}
\\
u^k(\vx,\vw)&=&(\e g+\pp u^{k-1})(\vx-\e
t_b\vw,\vw)e^{-(1+\e^2)t_b}+\int_{0}^{t_b}(f+\bar
u^{k-1})(\vx-\e(t_b-s)\vw,\vw)e^{-(1+\e^2)(t_b-s)}\ud{s}\nonumber.
\end{eqnarray}
where the backward exit time $t_b$ is defined as (\ref{exit time}).
We define the difference $v^k=u^{k}-u^{k-1}$ for $k\geq1$. Then
$v^k$ satisfies
\begin{eqnarray}
v^{k+1}(\vx,\vw)&=&\pp v^{k}(\vx-\e
t_b\vw,\vw)e^{-(1+\e^2)t_b}+\int_{0}^{t_b}\bar
v^{k}(\vx-\e(t_b-s)\vw,\vw)e^{-(1+\e^2)(t_b-s)}\ud{s}.
\end{eqnarray}
Since $\im{\bar
v^k}{\Omega\times\s^1}\leq\im{v^k}{\Omega\times\s^1}$ and $\im{\pp
v^k}{\Omega\times\s^1}\leq\im{v^k}{\Omega\times\s^1}$, we can directly estimate
\begin{eqnarray}
\im{v^{k+1}}{\Omega\times\s^1}&\leq&e^{-(1+\e^2)t_b}\im{v^{k+1}}{\Omega\times\s^1}+\im{v^{k}}{\Omega\times\s^1}\int_0^{t_b}e^{-(1+\e^2)(t_b-s)}\ud{s}\\
&\leq&e^{-(1+\e^2)t_b}\im{v^{k+1}}{\Omega\times\s^1}+\frac{1}{1+\e^2}(1-e^{-(1+\e^2)t_b})\im{v^{k}}{\Omega\times\s^1}\nonumber.
\end{eqnarray}
Hence, we naturally have
\begin{eqnarray}
\im{v^{k+1}}{\Omega\times\s^1}&\leq&\frac{1}{1+\e^2}\im{v^{k}}{\Omega\times\s^1}.
\end{eqnarray}
Thus, this is a contraction iteration. Considering $v^1=u^1$, we
have
\begin{eqnarray}
\im{v^{k}}{\Omega\times\s^1}\leq\bigg(\frac{1}{1+\e^2}\bigg)^{k-1}\im{u^{1}}{\Omega\times\s^1}.
\end{eqnarray}
for $k\geq1$. Therefore, $u^k$ converges strongly in $L^{\infty}$ to
the limiting solution $u$ satisfying
\begin{eqnarray}\label{well-posedness temp 1.}
\im{u}{\Omega\times\s^1}\leq\sum_{k=1}^{\infty}\im{v^{k}}{\Omega\times\s^1}\leq\frac{1+\e^2}{\e^2}\im{u^1}{\Omega\times\s^1}
\leq\frac{2}{\e^2}\im{u^1}{\Omega\times\s^1}.
\end{eqnarray}
Since $u^1$ satisfies the equation
\begin{eqnarray}
u^1(\vx,\vw)&=&g(\vx-\e
t_b\vw,\vw)e^{-(1+\e^2)t_b}+\int_{0}^{t_b}f(\vx-\e(t_b-s)\vw,\vw)e^{-(1+\e^2)(t_b-s)}\ud{s}\nonumber.
\end{eqnarray}
Based on Lemma \ref{well-posedness lemma 1}, we can directly
estimate
\begin{eqnarray}\label{well-posedness temp 2.}
\im{u^1}{\Omega\times\s^1}\leq
\im{f}{\Omega\times\s^1}+\im{g}{\Gamma^-}.
\end{eqnarray}
Combining (\ref{well-posedness temp 1.}) and (\ref{well-posedness
temp 2.}), we can naturally obtain the existence and estimates.
Also, the uniqueness easily follows from the energy estimates.
\end{proof}
\begin{theorem}\label{well-posedness 2.}
Assume $g(x_0,\vw)\in L^{\infty}(\Gamma^-)$. Then for the steady
neutron transport equation (\ref{transport.}), there exists a unique
solution $u^{\e}(\vx,\vw)\in L^{\infty}(\Omega\times\s^1)$
satisfying
\begin{eqnarray}
\im{u^{\e}}{\Omega\times\s^1}\leq\frac{C}{\e^{2}}\im{g}{\Gamma^-}.
\end{eqnarray}
\end{theorem}
\begin{proof}
We can apply Theorem \ref{well-posedness lemma 2.} to the equation
(\ref{transport.}). The result naturally follows.
\end{proof}

\subsection{$\e$-Milne Problem}

We consider the $\e$-Milne problem for $f^{\e}(\eta,\theta,\phi)$ in
the domain $(\eta,\theta,\phi)\in[0,\infty)\times[-\pi,\pi)\times[-\pi,\pi)$
\begin{eqnarray}\label{Milne problem.}
\left\{ \begin{array}{rcl}\displaystyle \sin\phi\frac{\p
f^{\e}}{\p\eta}+F(\e;\eta)\cos\phi\frac{\p
f^{\e}}{\p\phi}+f^{\e}-\bar f^{\e}&=&S^{\e}(\eta,\theta,\phi),\\
f^{\e}(0,\theta,\phi)&=& h^{\e}(\theta,\phi)+\pp f^{\e}(0)\ \ for\
\ \sin\phi>0,\\
\lim_{\eta\rt\infty}f^{\e}(\eta,\theta,\phi)&=&f^{\e}_{\infty}(\theta),
\end{array}
\right.
\end{eqnarray}
where
\begin{eqnarray}
\pp
f^{\e}(\eta,\theta)=-\half\int_{\sin\phi<0}f^{\e}(\eta,\theta,\phi)\sin\phi\ud{\phi},
\end{eqnarray}
$F(\e;\eta)$ is defined as (\ref{force}),
\begin{eqnarray}
\abs{h^{\e}(\theta,\phi)}\leq M,
\end{eqnarray}
and
\begin{eqnarray}\label{Milne decay.}
\abs{S^{\e}(\eta,\theta,\phi)}\leq Me^{-K\eta},
\end{eqnarray}
for $M$ and $K$ uniform in $\e$ and $\theta$.

For notational simplicity, we omit $\e$ and $\theta$ dependence in
$f^{\e}$ in this section. The same convention also applies to
$F(\e;\eta)$, $V(\e;\eta)$, $S^{\e}(\eta,\theta,\phi)$ and
$h^{\e}(\theta,\phi)$. However, our estimates are independent of
$\e$ and $\theta$.

\subsubsection{Compatibility Condition}

\begin{lemma}\label{Milne lemma 1.}
In order for the equation (\ref{Milne problem.}) to have a solution
$f\in L^{\infty}([0,\infty)\times[-\pi,\pi))$, the boundary data $h$
and the source term $S$ must satisfy the compatibility condition
\begin{eqnarray}\label{Milne compatibility condition}
\int_{\sin\phi>0}h(\phi)\sin\phi\ud{\phi}
+\int_0^{\infty}\int_{-\pi}^{\pi}e^{-V(s)}S(s,\phi)\ud{\phi}\ud{s}=0.
\end{eqnarray}
In particular, if $S=0$, then the compatibility condition reduces to
\begin{eqnarray}\label{Milne reduced compatibility condition}
\int_{\sin\phi>0}h(\phi)\sin\phi\ud{\phi}=0.
\end{eqnarray}
\end{lemma}
\begin{proof}
We just integrate over $\phi\in[-\pi,\pi)$ on both sides of
(\ref{Milne problem.}) and integrate by parts to achieve
\begin{eqnarray}\label{Milne temp 1.}
\frac{\ud{}}{\ud{\eta}}\int_{-\pi}^{\pi}
f(\eta,\phi)\sin\phi\ud{\phi}+F(\eta)\int_{-\pi}^{\pi}
f(\eta,\phi)\sin\phi\ud{\phi}=\int_{-\pi}^{\pi}S(\eta,\phi)\ud{\phi}.
\end{eqnarray}
This is a first order ordinary differential equation for
$\int_{-\pi}^{\pi} f(\eta,\phi)\sin\phi\ud{\phi}$. The initial data
is
\begin{eqnarray}
\int_{-\pi}^{\pi}
f(0,\phi)\sin\phi\ud{\phi}&=&\int_{\sin\phi>0}f(0,\phi)\sin\phi\ud{\phi}+\int_{\sin\phi<0}f(0,\phi)\sin\phi\ud{\phi}\\
&=&\int_{\sin\phi>0}\bigg(h(\phi)+\pp f(0)\bigg)\sin\phi\ud{\phi}+\int_{\sin\phi<0}f(0,\phi)\sin\phi\ud{\phi}\nonumber\\
&=&\int_{\sin\phi>0}h(\phi)\sin\phi\ud{\phi}+\int_{\sin\phi>0}\pp f(0)\sin\phi\ud{\phi}-2\pp f(0)\nonumber\\
&=&\int_{\sin\phi>0}h(\phi)\sin\phi\ud{\phi}\nonumber.
\end{eqnarray}
Hence, we can directly solve (\ref{Milne temp 1.}) as
\begin{eqnarray}
\int_{-\pi}^{\pi}
f(\eta,\phi)\sin\phi\ud{\phi}=e^{V(\eta)}\bigg(\int_{\sin\phi>0}h(\phi)\sin\phi\ud{\phi}
+\int_0^{\eta}\int_{-\pi}^{\pi}e^{-V(s)}S(s,\phi)\ud{\phi}\ud{s}\bigg).
\end{eqnarray}
The problem (\ref{Milne problem.}) requires $f(\eta,\phi)\rt
f_{\infty}$ as $\eta\rt\infty$. Hence, we must have
\begin{eqnarray}
\lim_{\eta\rt\infty}\int_{-\pi}^{\pi}
f(\eta,\phi)\sin\phi\ud{\phi}=0.
\end{eqnarray}
Therefore, the only possibility to justify above requirement is
\begin{eqnarray}
\int_{\sin\phi>0}h(\phi)\sin\phi\ud{\phi}
+\int_0^{\infty}\int_{-\pi}^{\pi}e^{-V(s)}S(s,\phi)\ud{\phi}\ud{s}=0.
\end{eqnarray}
This is the desired compatibility condition. If $S=0$, then above
condition reduces to
\begin{eqnarray}
\int_{\sin\phi>0}h(\phi)\sin\phi\ud{\phi}=0.
\end{eqnarray}
\end{proof}

\subsubsection{Reduction to In-flow $\e$-Milne Problem}

It is easy to see if $f$ is a solution to (\ref{Milne problem.}),
then $f+C$ is also a solution for any constant $C$. Hence, in order
to obtain a unique solution, we need a normalization condition
\begin{eqnarray}\label{Milne normalization}
\pp f(0)=0.
\end{eqnarray}
The following lemma tells us the problem (\ref{Milne problem.}) can
be reduced to the $\e$-Milne problem (\ref{Milne problem}) with in-flow
boundary.
\begin{lemma}\label{Milne lemma 2.}
If the boundary data $h$ and $S$ satisfy the compatibility condition
(\ref{Milne compatibility condition}), then the solution $f$ to the
$\e$-Milne problem (\ref{Milne problem}) with in-flow boundary as
$f=h$ on $\sin\phi>0$ is also a solution to the $\e$-Milne problem
(\ref{Milne problem.}) with diffusive boundary, which satisfies the
normalization condition (\ref{Milne normalization}). Furthermore,
this is the unique solution to (\ref{Milne problem.}) among the
functions satisfying (\ref{Milne normalization}) and
$\tnnm{f-f_{\infty}}<\infty$.
\end{lemma}
\begin{proof}
Consider $f$ satisfies the $\e$-Milne problem with in-flow boundary as
follows:
\begin{eqnarray}\label{Milne temp 2.}
\left\{ \begin{array}{rcl}\displaystyle \sin\phi\frac{\p
f}{\p\eta}+F(\eta)\cos\phi\frac{\p
f}{\p\phi}+f-\bar f&=&S(\eta,\phi),\\
f(0,\phi)&=& h(\phi)\ \ \text{for}\ \ \sin\phi>0,\\
\lim_{\eta\rt\infty}f(\eta,\phi)&=&f_{\infty}.
\end{array}
\right.
\end{eqnarray}
Then there exists a $f_{\infty}$ such that
$\lnnm{f-f_{\infty}}<\infty$ and $f$ decays to $f_{\infty}$
exponentially. Therefore, $z=f-f_{\infty}$ satisfies the equation
\begin{eqnarray}\label{Milne temp 3.}
\left\{ \begin{array}{rcl}\displaystyle \sin\phi\frac{\p
z}{\p\eta}+F(\eta)\cos\phi\frac{\p
z}{\p\phi}+z-\bar z&=&S(\eta,\phi),\\
z(0,\phi)&=& h(\phi)-f_{\infty}\ \ \text{for}\ \ \sin\phi>0,\\
\lim_{\eta\rt\infty}z(\eta,\phi)&=&0.
\end{array}
\right.
\end{eqnarray}
Multiplying $e^{-V(\eta)}$ on both sides of (\ref{Milne temp 3.})
and integrating over $\phi\in[-\pi,\pi)$ imply
\begin{eqnarray}\label{Milne temp 4.}
\frac{\ud{}}{\ud{\eta}}\int_{-\pi}^{\pi}e^{-V(\eta)}
z(\eta,\phi)\sin\phi\ud{\phi}=\int_{-\pi}^{\pi}e^{-V(\eta)}S(\eta,\phi)\ud{\phi}.
\end{eqnarray}
Since $z$ decays exponentially with respect to $\eta$, we know $z\in
L^1([0,\infty)\times[-\pi,\pi))$. Hence, we can integrate
(\ref{Milne temp 4.}) over $\eta\in[0,\infty)$ to obtain
\begin{eqnarray}
0-\int_{-\pi}^{\pi}e^{-V(0)}
z(0,\phi)\sin\phi\ud{\phi}=\int_0^{\infty}\int_{-\pi}^{\pi}e^{-V(s)}S(s,\phi)\ud{\phi}\ud{s},
\end{eqnarray}
which further implies
\begin{eqnarray}\label{Milne temp 5.}
\int_{-\pi}^{\pi}
z(0,\phi)\sin\phi\ud{\phi}+\int_0^{\infty}\int_{-\pi}^{\pi}e^{-V(s)}S(s,\phi)\ud{\phi}\ud{s}=0.
\end{eqnarray}
Then we may compute
\begin{eqnarray}\label{Milne temp 6.}
\int_{-\pi}^{\pi} z(0,\phi)\sin\phi\ud{\phi}&=&\int_{\sin\phi>0}
z(0,\phi)\sin\phi\ud{\phi}+\int_{\sin\phi<0}
z(0,\phi)\sin\phi\ud{\phi}\\
&=&\int_{\sin\phi>0} (h(\phi)-f_{\infty})\sin\phi\ud{\phi}-2\pp
z(0)\nonumber\\
&=&\int_{\sin\phi>0} h(\phi)\sin\phi\ud{\phi}-2\bigg(\pp
z(0)+f_{\infty}\bigg)\nonumber.
\end{eqnarray}
Combining (\ref{Milne temp 5.}), (\ref{Milne temp 6.}) and the
compatibility condition (\ref{Milne compatibility condition}), we
have
\begin{eqnarray}
\pp z(0)+f_{\infty}=0.
\end{eqnarray}
Since
\begin{eqnarray}
\pp z(0)=\pp f(0)-\pp f_{\infty}(0)=\pp f(0)-f_{\infty}.
\end{eqnarray}
naturally, we have $\pp f(0)=0$. Hence, $f$ is a solution to the
$\e$-Milne problem (\ref{Milne problem.}) with the normalization
condition (\ref{Milne normalization}). By Cauchy's inequality, we
can deduce the fact that $\tnnm{f-f_{\infty}}<\infty$ implies
$\tnnm{\bar f-f_{\infty}}<\infty$. Then by a similar argument as
Step 3 in the proof of Lemma \ref{Milne infinite LT}, we can show
the uniqueness.
\end{proof}
In summary, based on above analysis, we can utilize the known result
for $\e$-Milne problem with in-flow boundary to obtain the
well-posedness, decay and maximum principle of the solution to the
$\e$-Milne problem (\ref{Milne problem.}).
\begin{theorem}\label{Milne theorem 1.}
There exists a unique solution $f(\eta,\phi)$ to the $\e$-Milne problem
(\ref{Milne problem.}) with the normalization condition (\ref{Milne
normalization}) satisfying
\begin{eqnarray}
\lnnm{f-f_{\infty}}\leq C\bigg(1+M+\frac{M}{K}\bigg).
\end{eqnarray}
\end{theorem}
\begin{theorem}\label{Milne theorem 2.}
For $K_0>0$ sufficiently small, the solution $f(\eta,\phi)$ to the
$\e$-Milne problem (\ref{Milne problem.}) with the normalization
condition (\ref{Milne normalization}) satisfies
\begin{eqnarray}
\lnnm{e^{K_0\eta}(f-f_{\infty})}\leq C\bigg(1+M+\frac{M}{K}\bigg).
\end{eqnarray}
\end{theorem}
\begin{theorem}\label{Milne theorem 3.}
The solution to the $\e$-Milne problem (\ref{Milne problem.}) with $S=0$
and the normalization condition (\ref{Milne normalization})
satisfies the maximum principle, i.e.
\begin{eqnarray}
\min_{\sin\phi>0}h(\phi)\leq f(\eta,\phi)\leq
\max_{\sin\phi>0}h(\phi).
\end{eqnarray}
\end{theorem}
\begin{remark}\label{Milne remark.}
Note that when $F=0$, then all the previous proofs can be recovered
and Theorem \ref{Milne theorem 1.}, Theorem \ref{Milne theorem 2.}
and Theorem \ref{Milne theorem 3.} still hold. Hence, we can obtain
the well-posedness, decay and maximum principle of the classical
Milne problem
\begin{eqnarray}\label{classical Milne problem.}
\left\{ \begin{array}{rcl}\displaystyle
\sin\phi\frac{\p f}{\p\eta}+f-\bar f&=&S(\eta,\phi),\\
f(0,\phi)&=&\pp f(0)+h(\phi)\ \ \text{for}\ \ \sin\phi>0,\\
\lim_{\eta\rt\infty}f(\eta,\phi)&=&f_{\infty},
\end{array}
\right.
\end{eqnarray}
with the normalization condition (\ref{Milne normalization}). Note
that now we always have $V(\eta)=0$.
\end{remark}

\subsection{Main Results}

\begin{theorem}\label{main 2}
Assume $g(\vx_0,\vw)\in C^3(\Gamma^-)$. Then for the steady neutron
transport equation (\ref{transport.}), the unique solution
$u^{\e}(\vx,\vw)\in L^{\infty}(\Omega\times\s^1)$ satisfies
\begin{eqnarray}\label{main theorem 1.}
\lnm{u^{\e}-\u_0-\ub_0}=O(\e),
\end{eqnarray}
Moreover, if $g(\theta,\phi)=\cos\phi$, then there exists a $C>0$
such that
\begin{eqnarray}\label{main theorem 2.}
\lnm{u^{\e}-\uc_0-\ubc_0-\e\uc_1-\e\ubc_1}\geq C\e>0,
\end{eqnarray}
when $\e$ is sufficiently small.
\end{theorem}
\begin{proof}
We can divide the proof into several steps:\\
\ \\
Step 1: Remainder definitions.\\
Note that the boundary layer solution depends on $\e$ due to the
force and the interior solution also depends on $\e$ due the
boundary condition. We may rewrite the asymptotic expansion as
follows:
\begin{eqnarray}
u^{\e}&\sim&\sum_{k=0}^{\infty}\e^k\u_k+\sum_{k=0}^{\infty}\e^k\ub_k.
\end{eqnarray}
The remainder can be defined as
\begin{eqnarray}\label{pf 1_}
R_N&=&u^{\e}-\sum_{k=0}^{N}\e^k\u_k-\sum_{k=0}^{N}\e^k\ub_k=u-\q_N-\qb_N,
\end{eqnarray}
where
\begin{eqnarray}
\q_N&=&\sum_{k=0}^{N}\e^k\u_k,\\
\qb_N&=&\sum_{k=0}^{N}\e^k\ub_k.
\end{eqnarray}
Noting the equation (\ref{transport temp.}) is equivalent to the
equation (\ref{transport.}), we write $\ll$ to denote the neutron
transport operator as follows:
\begin{eqnarray}
\ll u&=&\e\vw\cdot\nx u+(1+\e^2)u-\bar u\\
&=&\sin\phi\frac{\p
u}{\p\eta}-\frac{\e}{1-\e\eta}\cos\phi\bigg(\frac{\p
u}{\p\phi}+\frac{\p u}{\p\theta}\bigg)+(1+\e^2)u-\bar u\nonumber
\end{eqnarray}
\ \\
Step 2: Estimates of $\ll \q_N$.\\
The interior contribution can be estimated as
\begin{eqnarray}
\\
\ll\q_0=\e\vw\cdot\nx \q_0+(1+\e^2)\q_0-\bar
\q_0&=&(\q_0-\bar\q_0)+\e\vw\cdot\nx \u_0+\e^2\u_0=\e\vw\cdot\nx
\u_0+\e^2\u_0.\nonumber
\end{eqnarray}
We can directly estimate
\begin{eqnarray}
\abs{\e\vw\cdot\nx \u_0}&\leq& C\e\abs{\nx\u_0}\leq C\e,\\
\abs{\e^2\u_0}&\leq&C\e^2\abs{\u_0}\leq C\e^2.
\end{eqnarray}
This implies
\begin{eqnarray}
\abs{\ll \q_0}\leq C\e.
\end{eqnarray}
For higher order term, we can estimate
\begin{eqnarray}
\ll\q_N=\e\vw\cdot\nx \q_N+(1+\e^2)\q_N-\bar
\q_N&=&\e^{N+1}\vw\cdot\nx \u_N+\e^{N+2}\u_{N}+\e^{N+1}\u_{N-1}.
\end{eqnarray}
We have
\begin{eqnarray}
\abs{\e^{N+1}\vw\cdot\nx \u_N}&\leq& C\e^{N+1}\abs{\nx\u_N}\leq
C\e^{N+1},\\
\abs{\e^{N+2}\u_{N}+\e^{N+1}\u_{N-1}}&\leq&C\e^{N+2}\abs{\u_N}+C\e^{N+1}\abs{\u_{N-1}}\leq
C\e^{N+1}.
\end{eqnarray}
This implies
\begin{eqnarray}\label{pf 2_}
\abs{\ll \q_N}\leq C\e^{N+1}.
\end{eqnarray}
\ \\
Step 3: Estimates of $\ll \qb_N$.\\
The boundary layer solution is
$\ub_k=(f_k^{\e}-f_k^{\e}(\infty))\cdot\psi_0=\v_k\psi_0$ where
$f_k^{\e}(\eta,\theta,\phi)$ solves the $\e$-Milne problem and $\v_k=f_k^{\e}-f_k^{\e}(\infty)$. Notice
$\psi_0\psi=\psi_0$, so the boundary layer contribution can be
estimated as
\begin{eqnarray}\label{remainder temp 1.}
\\
\ll\qb_0&=&\sin\phi\frac{\p
\qb_0}{\p\eta}-\frac{\e}{1-\e\eta}\cos\phi\bigg(\frac{\p
\qb_0}{\p\phi}+\frac{\p\qb_0}{\p\theta}\bigg)+(1+\e^2)\qb_0-\bar
\qb_0\nonumber\\
&=&\sin\phi\bigg(\psi_0\frac{\p
\v_0}{\p\eta}+\v_0\frac{\p\psi_0}{\p\eta}\bigg)-\frac{\psi_0\e}{1-\e\eta}\cos\phi\bigg(\frac{\p
\v_0}{\p\phi}+\frac{\p \v_0}{\p\theta}\bigg)+(1+\e^2)\psi_0\v_0-\psi_0\bar\v_0\nonumber\\
&=&\sin\phi\bigg(\psi_0\frac{\p
\v_0}{\p\eta}+\v_0\frac{\p\psi_0}{\p\eta}\bigg)-\frac{\psi_0\psi\e}{1-\e\eta}\cos\phi\bigg(\frac{\p
\v_0}{\p\phi}+\frac{\p \v_0}{\p\theta}\bigg)+(1+\e^2)\psi_0 \v_0-\psi_0\bar\v_0\nonumber\\
&=&\psi_0\bigg(\sin\phi\frac{\p
\v_0}{\p\eta}-\frac{\e\psi}{1-\e\eta}\cos\phi\frac{\p
\v_0}{\p\phi}+\v_0-\bar\v_0\bigg)+\sin\phi
\frac{\p\psi_0}{\p\eta}\v_0-\frac{\psi_0\e}{1-\e\eta}\cos\phi\frac{\p
\v_0}{\p\theta}+\e^2\psi_0\v_0\nonumber\\
&=&\sin\phi
\frac{\p\psi_0}{\p\eta}\v_0-\frac{\psi_0\e}{1-\e\eta}\cos\phi\frac{\p
\v_0}{\p\theta}+\e^2\psi_0\v_0\nonumber.
\end{eqnarray}
Since $\psi_0=1$ when $\eta\leq 1/(4\e)$, the effective region
of $\px\psi_0$ is $\eta\geq1/(4\e)$ which is further and further
from the origin as $\e\rt0$. By Theorem \ref{Milne theorem 2.}, the
first term in (\ref{remainder temp 1.}) can be controlled as
\begin{eqnarray}
\abs{\sin\phi\frac{\p\psi_0}{\p\eta}\v_0}&\leq&
Ce^{-\frac{K_0}{\e}}\leq C\e.
\end{eqnarray}
For the second term in (\ref{remainder temp 1.}), we have
\begin{eqnarray}
\abs{-\frac{\psi_0\e}{1-\e\eta}\cos\phi\frac{\p
\v_0}{\p\theta}}&\leq&C\e\abs{\frac{\p \v_0}{\p\theta}}\leq C\e.
\end{eqnarray}
For the third term in (\ref{remainder temp 1.}), we have
\begin{eqnarray}
\abs{\e^2\psi_0\v_0}\leq C\e.
\end{eqnarray}
This implies
\begin{eqnarray}
\abs{\ll \qb_0}\leq C\e.
\end{eqnarray}
For higher order term, we can estimate
\begin{eqnarray}\label{remainder temp 2.}
\ll\qb_N&=&\sin\phi\frac{\p
\qb_N}{\p\eta}-\frac{\e}{1-\e\eta}\cos\phi\bigg(\frac{\p
\qb_N}{\p\phi}+\frac{\p \qb_N}{\p\theta}\bigg)+(1+\e^2)\qb_N-\bar
\qb_N\\
&=&\sum_{i=0}^N\e^i\sin\phi
\frac{\p\psi_0}{\p\eta}\v_i-\frac{\psi_0\e^{N+1}}{1-\e\eta}\cos\phi\frac{\p
\v_N}{\p\theta}+\e^{N+2}\psi_0\v_N+\e^{N+1}\psi_0\v_{N-1}\nonumber.
\end{eqnarray}
Away from the origin, the first term in (\ref{remainder temp 2.})
can be controlled as
\begin{eqnarray}
\abs{\sum_{i=0}^N\e^i\sin\phi \frac{\p\psi_0}{\p\eta}\v_i}&\leq&
Ce^{-\frac{K_0}{\e}}\leq C\e^{N+1}.
\end{eqnarray}
For the second term in (\ref{remainder temp 2.}), we have
\begin{eqnarray}
\abs{-\frac{\psi_0\e^{N+1}}{1-\e\eta}\cos\phi\frac{\p
\v_N}{\p\theta}}&\leq&C\e^{N+1}\abs{\frac{\p \v_N}{\p\theta}}\leq
C\e^{N+1}.
\end{eqnarray}
For the third term in (\ref{remainder temp 2.}), we have
\begin{eqnarray}
\abs{\e^{N+2}\psi_0\v_N+\e^{N+1}\psi_0\v_{N-1}}\leq C\e^{N+1}.
\end{eqnarray}
This implies
\begin{eqnarray}\label{pf 3_}
\abs{\ll \qb_N}\leq C\e^{N+1}.
\end{eqnarray}
\ \\
Step 4: Proof of (\ref{main theorem 1.}).\\
In summary, since $\ll u^{\e}=0$, collecting (\ref{pf 1_}), (\ref{pf 2_}) and (\ref{pf 3_}), we can prove
\begin{eqnarray}
\abs{\ll R_N}\leq C\e^{N+1}.
\end{eqnarray}
Consider the asymptotic expansion to $N=2$, then the remainder $R_2$
satisfies the equation
\begin{eqnarray}
\left\{
\begin{array}{rcl}
\e \vw\cdot\nabla_x R_2+R_2-\bar R_2&=&\ll R_2\ \ \text{for}\ \ \vx\in\Omega,\\
R_2&=&\pp R_2\ \ \text{for}\ \ \vw\cdot\vec n<0\ \ \text{and}\ \
\vx_0\in\p\Omega.
\end{array}
\right.
\end{eqnarray}
By Theorem \ref{well-posedness lemma 2.}, we have
\begin{eqnarray}
\im{R_2}{\Omega\times\s^1}\leq \frac{C}{\e^2}\im{\ll
R_2}{\Omega\times\s^1}\leq \frac{C\e^3}{\e^2}\leq C\e.
\end{eqnarray}
Hence, we have
\begin{eqnarray}
\lnm{u^{\e}-\sum_{k=0}^2\e^k\u_k-\sum_{k=0}^2\e^k\ub_k}=O(\e).
\end{eqnarray}
Since it is easy to see
\begin{eqnarray}
\lnm{\sum_{k=1}^2\e^k\u_k+\sum_{k=1}^2\e^k\ub_k}\leq C\e,
\end{eqnarray}
our result naturally follows. This completes the proof of (\ref{main theorem 1.}).\\
\ \\
Step 5: Proof of (\ref{main theorem 2.}).\\
By (\ref{classical temp 2.}), the solution $\f_1$ satisfies the
Milne problem
\begin{eqnarray}
\left\{
\begin{array}{rcl}
\sin(\theta+\xi)\dfrac{\p \f_1}{\p\eta}+\f_1-\bar \f_1&=&0,\\
\f_1(0,\theta,\xi)&=&\pp \f_1(0,\theta)+g_1(\theta,\xi)\ \
\text{for}\ \ \sin(\theta+\xi)>0,\\\rule{0ex}{1em}
\lim_{\eta\rt\infty}\f_1(\eta,\theta,\xi)&=& f_1(\infty,\theta).
\end{array}
\right.
\end{eqnarray}
For convenience of comparison, we make the substitution
$\phi=\theta+\xi$ to obtain
\begin{eqnarray}
\left\{
\begin{array}{rcl}
\sin\phi\dfrac{\p \f_1}{\p\eta}+\f_1-\bar \f_1&=&0,\\
\f_1(0,\phi)&=&\pp \f_1(0)+g_1(\theta,\phi)\ \ \text{for}\ \
\sin\phi>0,\\\rule{0ex}{1em}
\lim_{\eta\rt\infty}\f_1(\eta,\theta,\xi)&=& f_1(\infty,\theta).
\end{array}
\right.
\end{eqnarray}
Assume (\ref{main theorem 2.}) is incorrect. For our
$g(\phi)=\cos\phi$ which is independent of $\theta$, since
$\ubc_0=\ub_0=0$ and $\uc_0=\u_0=0$, we have
\begin{eqnarray}
\lim_{\e\rt0}\lnm{(\uc_1+\ubc_1)-(\u_1+\ub_1)}=0.
\end{eqnarray}
Since now $\ubc_1$ and $\ub_1$ are independent of $\theta$, by
(\ref{expansion temp 22.}) and (\ref{expansion temp 12.}), we can
directly estimate
\begin{eqnarray}
\frac{\p \bu_1}{\p\vec
n}&=&-\int_0^{\infty}\int_{-\pi}^{\pi}e^{-V(s)}\bigg(\frac{\psi(\e
s)}{1-\e\eta}\cos\phi\frac{\p \ub_1}{\p\theta}(s,\phi)-\psi(\e
s)\ub_0(s,\phi)\bigg)\ud{\phi}\ud{s}=0,
\end{eqnarray}
and also
\begin{eqnarray}
\frac{\p \bar \uc_1}{\p\vec n}
&=&-\int_0^{\infty}\int_{-\pi}^{\pi}e^{-V(s)}\bigg(\frac{\psi(\e
s)}{1-\e\eta}\cos\phi\frac{\p \ubc_1}{\p\theta}(s,\phi)-\psi(\e
s)\ubc_0(s,\phi)\bigg)\ud{\phi}\ud{s}=0.
\end{eqnarray}
Hence, we have $\bu_1=\bar \uc_1$ in the domain, which further
implies $\u_1=\uc_1$. Therefore, we can obtain
\begin{eqnarray}
\lim_{\e\rt0}\lnm{\ubc_1-\ub_1}=0.
\end{eqnarray}
Then on the boundary of $\sin\phi>0$, these two boundary layer
solutions satisfy
\begin{eqnarray}
\ubc_1&=&g-f_1(\infty),\\
\ub_1&=&g-f_1^{\e}(\infty).
\end{eqnarray}
Naturally, we have the estimate
$\lim_{\e\rt0}\lnm{f_1^{\e}(\infty)-f_1(\infty)}=0$ based on above
assumptions. Hence, we may further derive
\begin{eqnarray}
\lim_{\e\rt\infty}\lnm{(f_1(\infty)+\ubc_1)-(f_1^{\e}(\infty)+\ub_1)}=0.
\end{eqnarray}
For $0\leq\eta\leq 1/(2\e)$, we have $\psi_0=1$, which means
$\f_1=\ubc_1+f_1(\infty)$ and $f_1^{\e}=\ub_1+f_1^{\e}(\infty)$ on
$[0,1/(2\e)]$. Since we have $\pp\ub_1(0)=-f_1^{\e}(\infty)$ and
$\pp \ubc_1(0)=-f_1(\infty)$, we have recovered the normalization
condition, i.e. $\pp f_1(0)=\pp f_1^{\e}(0)=0$. Note that $g$
satisfies the compatibility condition (\ref{Milne reduced
compatibility condition}). Therefore, the $\e$-Milne problem
satisfied by $f_1$ and $f_1^{\e}$ can be reduced to the $\e$-Milne
problem with in-flow boundary. Hence, we can naturally obtain the
desired result through the proof of Theorem \ref{main 1}.
\end{proof}

\appendix

\makeatletter
\renewcommand \theequation {%
A.%
\ifnum\c@subsection>\z@\@arabic\c@subsection.%
\fi\@arabic\c@equation} \@addtoreset{equation}{section}
\@addtoreset{equation}{subsection} \makeatother

\section{Construction of the Counterexample with In-Flow Boundary}

\begin{lemma}\label{counter theorem 1}
For the Milne problem
\begin{eqnarray}\label{counter equation}
\left\{
\begin{array}{rcl}
\sin(\theta+\xi)\dfrac{\p f}{\p\eta}+f-\bar f&=&0,\\
f(0,\theta,\xi)&=&g(\theta,\xi)\ \ \text{for}\ \
\sin(\theta+\xi)>0,\\\rule{0ex}{1em}
\lim_{\eta\rt\infty}f(\eta,\theta,\xi)&=&f(\infty,\theta),
\end{array}
\right.
\end{eqnarray}
if $g(\theta,\xi)=\cos(3(\theta+\xi))$, then we have
\begin{eqnarray}
\frac{\p f}{\p\eta}\notin L^{\infty}([0,\infty)\times[-\pi,\pi)\times[-\pi,\pi)).
\end{eqnarray}
\end{lemma}
\begin{proof}
We divide the proof into several steps: we first assume $\px f\in
L^{\infty}([0,\infty)\times[-\pi,\pi)\times[-\pi,\pi))$
and then show it can lead to a contradiction.\\
\ \\
Step 1: Maximum principle\\
Theorem \ref{Milne theorem 3} implies the solution $f$ to the
problem (\ref{counter equation}) satisfies the maximum principle,
i.e. for any $(\eta,\theta,\phi)$
\begin{eqnarray}
\min_{\sin(\theta+\xi)>0}g(\theta,\xi)\leq f(\eta,\theta,\xi)\leq
\max_{\sin(\theta+\xi)>0}g(\theta,\xi).
\end{eqnarray}
We can see the data $g(\theta,\xi)=\cos(3(\theta+\xi))$ satisfying
$g(\theta,-\theta)=1$ and $\abs{g}\leq 1$. Based on the maximum
principle, we can derive for any $(\eta,\theta,\xi)$
\begin{eqnarray}
f(\eta,\theta,\xi)\leq1.
\end{eqnarray}
Hence, certainly we have $f(0,\theta,\xi)\leq1$ for
$\sin(\theta+\xi)<0$.\\
\ \\
Step 2: Estimates of $\bar f(0,\theta)$.\\
We can directly estimate
\begin{eqnarray}
\bar
f(0,\theta)&=&\frac{1}{2\pi}\int_{-\pi}^{\pi}f(0,\theta,\xi)\ud{\xi}\\
&=&\frac{1}{2\pi}\bigg(\int_{\sin(\theta+\xi)<0}f(0,\theta,\xi)\ud{\xi}
+\int_{\sin(\theta+\xi)>0}f(0,\theta,\xi)\ud{\xi}\bigg)\nonumber\\
&\leq&\frac{1}{2\pi}\int_{\sin(\theta+\phi)>0}f(0,\theta,\xi)\ud{\xi}+\half\nonumber.
\end{eqnarray}
By the choice of $g$, we naturally have
\begin{eqnarray}
\int_{\sin(\theta+\xi)>0}f(0,\theta,\xi)\ud{\xi}&=&\int_{\sin(\theta+\xi)>0}g(\theta,\xi)\ud{\xi}=0.
\end{eqnarray}
Then this implies
\begin{eqnarray}
\bar f(0,\theta)\leq \half.
\end{eqnarray}
\ \\
Step 3: Definition of trace.\\
It is easy to see $\px f$ satisfies the Milne problem
\begin{eqnarray}
\sin(\theta+\xi)\frac{\p (\px f)}{\p\eta}+\px f-\overline{\px f}&=&0.
\end{eqnarray}
Since we have $\px f\in
L^{\infty}([0,\infty)\times[-\pi,\pi)\times[-\pi,\pi))$ which
implies $\overline{\px f}\in L^{\infty}([0,L]\times[-\pi,\pi))$, by
Ukai's trace theorem, we may define the trace of $\px f$ on $\eta=0$
satisfying $\px f(0,\theta,\phi)\in
L^{\infty}[-\pi,\pi)\times[-\pi,\pi)$.

However, we can define the trace of $\px f$ in another fashion. For
any $\xi\neq-\theta$ and $\xi\neq\pi-\theta$, we have
$\sin(\theta+\xi)\neq0$. Since we have $f\in
L^{\infty}([0,\infty)\times[-\pi,\pi)\times[-\pi,\pi))$ as well as
$\bar f\in L^{\infty}[0,\infty)\times[-\pi,\pi)$, by the Milne
problem (\ref{counter equation}), it is naturally to define for
$\eta>0$
\begin{eqnarray}
\px f(\eta,\theta,\xi)=\frac{\bar
f(\eta,\theta)-f(\eta,\theta,\xi)}{\sin(\theta+\xi)}.
\end{eqnarray}
Since $\px f\in
L^{\infty}([0,\infty)\times[-\pi,\pi)\times[-\pi,\pi))$, we know $f$
is continuous with respect to $\eta$ for a.e. $(\theta,\xi)$. Taking
$\eta\rt0$ defines the trace for $\px f$ at$(0,\theta,\xi)$
\begin{eqnarray}
\px f(0,\theta,\xi)=\frac{\bar
f(0,\theta)-f(0,\theta,\xi)}{\sin(\theta+\xi)}.
\end{eqnarray}
Since the grazing set $\{(\theta,\xi): \theta+\xi=0\ \ \text{or}\ \
\theta+\xi=\pi\}$ is zero-measured on the boundary $\eta=0$, then we
have the trace of $\px f$ is a.e. well-defined.

By the uniqueness of trace of $\px f$, above two types of traces
must coincide with each other a.e.. Then we may combine them both
and obtain $\px f(0,\theta,\xi)\in
L^{\infty}[-\pi,\pi)\times[-\pi,\pi)$ is a.e. well-defined and
satisfies the formula
\begin{eqnarray}
\px f(0,\theta,\xi)=\frac{\bar
f(0,\theta)-f(0,\theta,\xi)}{\sin(\theta+\xi)}.
\end{eqnarray}
\ \\
Step 4: Contradiction.\\
Therefore, we may consider the limiting process
\begin{eqnarray}
\lim_{\xi\rt-\theta^+}\frac{\p
f}{\p\eta}(0,\theta,\xi)=\lim_{\xi\rt-\theta^+}\frac{\bar
f(0,\theta)-f(0,\theta,\xi)}{\sin(\theta+\xi)}.
\end{eqnarray}
Since we know as $\xi\rt-\theta^+$, it follows that
\begin{eqnarray}
\sin(\theta+\xi)&\rt& 0^+\\
\bar f(0,\theta)-f(0,\theta,\xi)&\rt&\bar
f(0,\theta)-g(\theta,-\theta)=\bar f(0,\theta)-1<0.
\end{eqnarray}
Then this leads to
\begin{eqnarray}
\lim_{\xi\rt-\theta^+}\frac{\p f}{\p\eta}(0,\theta,\xi)=-\infty.
\end{eqnarray}
which means $\px f(0,\theta,\xi)\notin
L^{\infty}[-\pi,\pi)\times[-\pi,\pi)$. This contradicts our result
in the previous step. Hence, our assumption that $\px f\in
L^{\infty}([0,\infty)\times[-\pi,\pi)\times[-\pi,\pi))$ cannot be
true.\\
\ \\
Step 5: Another contradiction.\\
There is another way to show this fact. Since $\px f\in
L^{\infty}([0,L]\times[-\pi,\pi)\times[-\pi,\pi))$. Also, we have
$f\in L^{\infty}([0,L]\times[-\pi,\pi)\times[-\pi,\pi))$. Then this
implies $f$ is Lipschitz continuous in $[0,\infty)$ with respect to
$\eta$ for a.e. $(\theta,\xi)$. Hence, this implies $\bar f$ is also
Lipschitz continuous in $\eta\in[0,\infty)$. Without loss of
generality, we may assume $\nm{\px f}_{\infty}\leq M$. Thus we have
for a.e. $(\theta,\xi)\in \sin(\theta+\phi)>0$
\begin{eqnarray}
\abs{f(\eta,\theta,\xi)-f(0,\theta,\xi)}&\leq& M\eta\\
\abs{\bar f(\eta,\theta)-\bar f(0,\theta)}&\leq& M\eta.
\end{eqnarray}
Then
\begin{eqnarray}
\abs{f(\eta,\theta,\xi)-\bar
f(\eta,\theta)}&\geq&\abs{f(0,\theta,\xi)-\bar
f(0,\theta)}-\abs{f(\eta,\theta,\xi)-f(0,\theta,\xi)}-\abs{\bar
f(\eta,\theta)-\bar f(0,\theta)}\\
&\geq&\abs{\bar f(0,\theta)-g(\theta,-\theta)}-2M\eta\nonumber.
\end{eqnarray}
Since we know $\abs{\bar f(0,\theta)-g(\theta,-\theta)}\geq C>0$ for
some constant $C$. Then as long as $\eta\leq C/(4M)$, we have
\begin{eqnarray}
\abs{f(\eta,\theta,\xi)-\bar f(\eta,\theta)}&\geq&\frac{C}{2}.
\end{eqnarray}
Since for $(\theta,\xi)$ not in the grazing set, we always have
\begin{eqnarray}
\px f(\eta,\theta,\xi)=\frac{\bar
f(\eta,\theta)-f(\eta,\theta,\xi)}{\sin(\theta+\xi)}.
\end{eqnarray}
then $\abs{\px f}$ can be arbitrarily large as long as
$\sin(\theta+\xi)$ is sufficiently small, and also it possesses a
positive measure. This implies $\px f\in
L^{\infty}([0,\infty)\times[-\pi,\pi)\times[-\pi,\pi))$ cannot be
true.
\end{proof}

\section{Construction of the Counterexample with Diffusive Boundary}

\begin{lemma}\label{counter theorem 1.}
For the Milne problem
\begin{eqnarray}
\left\{
\begin{array}{rcl}
\sin(\theta+\xi)\dfrac{\p f}{\p\eta}+f-\bar f&=&0,\\
f(0,\theta,\xi)&=&\pp f(0,\theta)+g_1(\theta,\xi)\ \ \text{for}\ \
\sin(\theta+\xi)>0,\\\rule{0ex}{1em}
\lim_{\eta\rt\infty}f(\eta,\theta,\xi)&=& f(\infty,\theta),
\end{array}
\right.
\end{eqnarray}
with
\begin{eqnarray}
\pp f(0,\theta)=0.
\end{eqnarray}
If $g(\theta,\xi)=\cos(3(\theta+\xi))$, then we have
\begin{eqnarray}
\frac{\p f}{\p\eta}\notin L^{\infty}([0,\infty)\times[-\pi,\pi)\times[-\pi,\pi)).
\end{eqnarray}
\end{lemma}
\begin{proof}
For $g(\theta,\xi)=\cos(3(\theta+\xi))$, we can easily derive
$\uc_0=0$. Hence $g_1=g$. Note that $g$ satisfies the compatibility
condition (\ref{Milne reduced compatibility condition}). With the
normalization condition $\pp f(0,\theta)=0$, we can see this problem
reduces to the Milne problem with in-flow boundary
\begin{eqnarray}\label{counter equation.}
\left\{
\begin{array}{rcl}
\sin(\theta+\xi)\dfrac{\p f}{\p\eta}+f-\bar f&=&0,\\
f(0,\theta,\xi)&=&g(\theta,\xi)\ \ \text{for}\ \
\sin(\theta+\xi)>0,\\\rule{0ex}{1em}
\lim_{\eta\rt\infty}f(\eta,\theta,\xi)&=&f_1(\infty,\theta).
\end{array}
\right.
\end{eqnarray}
Therefore, we can complete the proof by Theorem \ref{counter theorem
1}.
\end{proof}

\section*{Acknowledgements}

The author thanks Raffaele Esposito, Claude Bardos and Xiongfeng Yang for stimulating
discussions. The research is supported in part by NSFC grant
10828103 and NSF grant DMS-0905255.


\begin{thebibliography}{1}


\bibitem{book10}{\sc Arkeryd, Leif; Esposito, Raffaele; Marra, Rossana; Nouri,
Anne}; {\em Ghost effect by curvature in planar Couette flow.}
Kinet. Relat. Models 4 (2011), no. 1, 109-138.

\bibitem{book2}{\sc Bardos, C.; Santos, R.; Sentis, R.};
{\em Diffusion approximation and computation of the critical size.}
Trans. Amer. Math. Soc. 284 (1984), no. 2, 617-649.

\bibitem{book3}{\sc Beals, R.; Protopopescu, V.};
{\em Abstract time-dependent transport equations.} J. Math. Anal.
Appl. 121 (1987), no. 2, 370-405.

\bibitem{book1}{\sc Bensoussan, Alain; Lions, Jacques-L.; Papanicolaou, George C.};
{\em Boundary layers and homogenization of transport processes.}
Publ. Res. Inst. Math. Sci. 15 (1979), no. 1, 53-157.

\bibitem{book6}{\sc Cercignani, Carlo; Illner, Reinhard; Pulvirenti, Mario};
{\em The mathematical theory of dilute gases.} Applied Mathematical
Sciences, 106. Springer-Verlag, New York, 1994.

\bibitem{book8}{\sc Cercignani, Carlo; Marra, R.; Esposito, R.};
{\em The Milne problem with a force term.} Transport Theory Statist.
Phys. 27 (1998), no. 1, 1-33.

\bibitem{book22}{\sc Chepurko, A. N.};
{\em Asymptotic behavior of the solution of a singularly perturbed
nonstationary neutron transport equation.} (Russian)Zh. Vychisl. Mat.
Mat. Fiz. 38 (1998), no. 2, 289-297; translation in Comput. Math.
Math. Phys. 38 (1998), no. 2, 279-287

\bibitem{book5}{\sc Esposito, R.; Guo, Y.; Kim, C.; Marra, R.};
{\em Non-isothermal boundary in the Boltzmann theory and Fourier
law.} Comm. Math. Phys. 323 (2013), no. 1, 177-239.

\bibitem{book4}{\sc Greenberg, W.; van der Mee, C.; Protopopescu, V.};
{\em Boundary value problems in abstract kinetic theory.} Operator
Theory: Advances and Applications, 23. Birkhauser Verlag, Basel,
1987.

\bibitem{book9}{\sc Guo, Yan};
{\em Decay and continuity of the Boltzmann equation in bounded
domains.} Arch. Ration. Mech. Anal. 197 (2010), no. 3, 713-809.

\bibitem{book11}{\sc Guo, Yan}; {\em Stable magnetic equilibria in a symmetric collisionless
plasma.} Comm. Math. Phys. 200 (1999), no. 1, 211-247.

\bibitem{book7}{\sc Yang, Xiongfeng};
{\em Asymptotic behavior on the Milne problem with a force term.} J.
Differential Equations 252 (2012), no. 9, 4656-4678.

\bibitem{book14}{\sc Larsen, Edward W.};
{\em Solutions of the steady, one-speed neutron transport equation
for small mean free paths.} J. Mathematical Phys. 15 (1974),
299-305.

\bibitem{book15}{\sc Larsen, Edward W.};
{\em A functional-analytic approach to the steady, one-speed neutron
transport equation with anisotropic scattering.} Comm. Pure Appl.
Math. 27 (1974), 523-545.

\bibitem{book20}{\sc Larsen, Edward W.};
{\em Asymptotic theory of the linear transport equation for small
mean free paths. II.} SIAM J. Appl. Math. 33 (1977), no. 3, 427-445.

\bibitem{book18}{\sc Larsen, Edward W.};
{\em Neutron transport and diffusion in inhomogeneous media. I.} J.
Mathematical Phys. 16 (1975), 1421-1427.

\bibitem{book19}{\sc Larsen, Edward W.; D'Arruda, Jose.};
{\em Asymptotic theory of the linear transport equation for small
mean free paths. I.} Phys. Rev. A (3) 13 (1976), no. 5, 1933-1939.

\bibitem{book13}{\sc Larsen, Edward W.; Habetler, George J.};
{\em A functional-analytic derivation of Case's full and half-range
formulas.} Comm. Pure Appl. Math. 26 (1973), 525-537.

\bibitem{book12}{\sc Larsen, Edward W.; Keller, Joseph B.};
{\em Asymptotic solution of neutron transport problems for small
mean free paths.} J. Mathematical Phys. 15 (1974), 75-81.

\bibitem{book17}{\sc Larsen, Edward W.; Zweifel, Paul F.};
{\em Steady, one-dimensional multigroup neutron transport with
anisotropic scattering.} J. Mathematical Phys. 17 (1976), no. 10,
1812-1820.

\bibitem{book16}{\sc Larsen, Edward W.; Zweifel, Paul F.};
{\em On the spectrum of the linear transport operator.} J.
Mathematical Phys. 15 (1974), 1987-1997.

\bibitem{book21}{\sc Malvagi, F.; Pomraning, G. C.};
{\em Initial and boundary conditions for diffusive linear transport
problems.} J. Math. Phys. 32 (1991), no. 3, 805-820.


\end{thebibliography}
\end{document}